\documentclass{amsart}

\usepackage{amsmath}
\usepackage{amsfonts}
\usepackage{amssymb}
\usepackage{amsthm}
\usepackage{comment}
\usepackage{epsfig}
\usepackage{psfrag}
\usepackage{mathrsfs}
\usepackage{amscd}
\usepackage[all]{xy}
\usepackage{rotating}
\usepackage{lscape}
\usepackage{amsbsy}
\usepackage{verbatim}
\usepackage{moreverb}
\usepackage{color}
\usepackage{bbm}
\usepackage{eucal}
\usepackage{stackrel}
\usepackage{stmaryrd}
\usepackage{nccmath}

\usepackage{tikz-cd}
\usetikzlibrary{patterns,shapes.geometric,arrows,decorations.markings}
\usepackage{tikz-3dplot}

\usepackage{caption}
\usepackage{subcaption}

\colorlet{lightgray}{black!15}

\tikzset{->-/.style={decoration={
  markings,
  mark=at position .5 with {\arrow{>}}},postaction={decorate}}}
\tikzset{midarrow/.style={decoration={
    markings,
    mark=at position {#1} with {\arrow{>}}},postaction={decorate}} }

\newtheorem{theorem}{Theorem}[subsection]
\newtheorem{innercustomthm}{Theorem}
\newenvironment{mythm}[1]
  {\renewcommand\theinnercustomthm{#1}\innercustomthm}
  {\endinnercustomthm}

\newtheorem{prop}[theorem]{Proposition}
\newtheorem{lemma}[theorem]{Lemma}
\newtheorem{cor}[theorem]{Corollary}
\newtheorem{conj}{Conjecture}

\numberwithin{equation}{subsection}

\theoremstyle{definition}
\newtheorem{definition}[theorem]{Definition}

\newtheorem{ack}[theorem]{Acknowledgments}
\newtheorem{observation}[theorem]{Observation}

\newtheorem{warning}[theorem]{Warning}
\newtheorem{remark}[theorem]{Remark}
\newtheorem{problem}{Problem}
\newtheorem{example}[theorem]{Example}

\newtheorem{notation}[theorem]{Notation}

\newtheorem*{conventions*}{Conventions}
\newtheorem{convention}{Convention}

\theoremstyle{remark}

\definecolor{orange}{rgb}{.95,0.5,0}
\definecolor{light-gray}{gray}{0.75}
\definecolor{brown}{cmyk}{0, 0.8, 1, 0.6}
\definecolor{plum}{rgb}{.5,0,1}

\DeclareMathOperator{\Fin}{\sf Fin}

\DeclareMathOperator{\pr}{\mathsf{pr}}
\DeclareMathOperator{\ev}{\mathsf{ev}}

\DeclareMathOperator{\bBar}{\sf Bar}
\DeclareMathOperator{\Alg}{\sf Alg}

\DeclareMathOperator{\Mod}{\sf Mod}

\DeclareMathOperator{\CAlg}{\sf CAlg}

\DeclareMathOperator{\Aut}{\sf Aut}
\DeclareMathOperator{\colim}{{\sf colim}}

\DeclareMathOperator{\Hom}{\sf Hom}
\DeclareMathOperator{\End}{\sf End}

\DeclareMathOperator{\Fun}{{\sf Fun}}
\DeclareMathOperator{\Iso}{\sf Iso}

\DeclareMathOperator{\Map}{{\sf Map}}

\DeclareMathOperator{\Ker}{\sf Ker}

\newcommand{\bit}[1]{\textbf{\textit{#1}}}
\newcommand{\racts}{\curvearrowleft}
\newcommand{\lacts}{\curvearrowright}

\DeclareMathOperator{\Cat}{{\sf Cat}}

\DeclareMathOperator{\Ar}{{\sf Ar}}

\DeclareMathOperator{\Diff}{{\sf Diff}}

\DeclareMathOperator{\op}{\mathsf{op}}

\DeclareMathOperator{\sk}{\mathsf{sk}}

\DeclareMathOperator{\Mfd}{{\cM}\mathsf{fd}}

\DeclareMathOperator{\Mfld}{{\sf Mfld}}

\DeclareMathOperator{\Emb}{\mathsf{Emb}}

\DeclareMathOperator{\Spaces}{\cS\mathsf{paces}}
\DeclareMathOperator{\Sets}{\cS\mathsf{ets}}

\DeclareMathOperator{\Disk}{{\sf Disk}}

\DeclareMathOperator{\fr}{\sf fr}
\DeclareMathOperator{\sfr}{\sf sfr}

\def\ot{\otimes}

\DeclareMathOperator{\oo}{\infty}

\DeclareMathOperator{\tl}{\triangleleft}

\newcommand{\lag}{\langle}
\newcommand{\rag}{\rangle}

\newcommand{\w}{\widetilde}
\newcommand{\un}{\underline}
\newcommand{\ov}{\overline}

\newcommand{\xra}{\xrightarrow}
\newcommand{\xla}{\xleftarrow}

\def\cE{\mathcal E}
\def\cK{\mathcal K}
\def\cM{\mathcal M}
\def\cS{\mathcal S}
\def\cU{\mathcal U}\def\cV{\mathcal V}\def\cW{\mathcal W}\def\cX{\mathcal X}

\def\BB{\mathbb B}\def\CC{\mathbb C}\def\DD{\mathbb D}

\def\LL{\mathbb L}
\def\NN{\mathbb N}\def\PP{\mathbb P}
\def\QQ{\mathbb Q}\def\RR{\mathbb R}\def\SS{\mathbb S}\def\TT{\mathbb T}

\def\ZZ{\mathbb Z}

\def\sB{\mathsf B}\def\sC{\mathsf C}\def\sD{\mathsf D}
\def\sE{\mathsf E}\def\sH{\mathsf H}
\def\sK{\mathsf K}
\def\sO{\mathsf O}\def\sP{\mathsf P}
\def\sR{\mathsf R}\def\sT{\mathsf T}

\def\bDelta{\mathbf\Delta}

\def\fB{\mathfrak B}

\def\bcD{\boldsymbol{\mathcal D}}

\def\bcL{\boldsymbol{\mathcal L}}
\def\bcM{\boldsymbol{\mathcal M}}

\def\bcM{\boldsymbol{\mathcal M}}

\DeclareMathOperator{\para}{\bDelta_{\circlearrowleft}}
\DeclareMathOperator{\bLambda}{\bf \Lambda}

\DeclareMathOperator{\uno}{\mathbbm{1}}

\DeclareMathOperator{\Braid}{{\sf Braid}_3}
\DeclareMathOperator{\Ebraid}{\w{\sf E}_{2}^{+}(\ZZ)}
\DeclareMathOperator{\GL}{\sf GL}
\DeclareMathOperator{\BGL}{\sf BGL}
\DeclareMathOperator{\SP}{\sf SP}
\DeclareMathOperator{\SL}{\sf SL}
\DeclareMathOperator{\BSL}{\sf BSL}
\DeclareMathOperator{\quot}{\sf quot}
\DeclareMathOperator{\Fr}{\sf Fr}
\DeclareMathOperator{\id}{\sf id}
\DeclareMathOperator{\Act}{\sf Act}
\DeclareMathOperator{\trans}{\sf trans}
\DeclareMathOperator{\Bdl}{\sf Bdl}
\DeclareMathOperator{\sHH}{\sf HH}
\DeclareMathOperator{\HHt}{{\sf HH}^{(2)}}
\DeclareMathOperator{\Obj}{\sf Obj}

\DeclareMathOperator{\Imm}{\sf Imm}
\DeclareMathOperator{\Aff}{\sf Aff}
\DeclareMathOperator{\EZ}{\sE_2(\ZZ)}
\DeclareMathOperator{\EpZ}{\sE^+_2(\ZZ)}
\DeclareMathOperator{\PShv}{\sf PShv}

\makeatletter
\@namedef{subjclassname@2020}{\textup{2020} Mathematics Subject Classification}
\makeatother
\usepackage[margin=1.17in]{geometry}

\begin{document}

\title{Natural symmetries of secondary Hochschild homology}

\author{David Ayala, John Francis, and Adam Howard}

\address{Department of Mathematics\\Montana State University\\Bozeman, MT 59717}
\email{david.ayala@montana.edu}
\address{Department of Mathematics\\Northwestern University\\Evanston, IL 60208}
\email{jnkf@northwestern.edu}
\address{Department of Mathematics\\Montana State University\\Bozeman, MT 59717}
\email{adam.howard1@montana.edu}
\thanks{DA was supported by the National Science Foundation under awards 1812055 and 1945639.  JF was supported by the National Science Foundation under award 1812057.
This material is based upon work supported by the National Science Foundation under Grant No. DMS-1440140, while DA and JF were in residence at the Mathematical Sciences Research Institute in Berkeley, California, during the Spring 2020 semester.}

\begin{abstract}
We identify the group of framed diffeomorphisms of the torus as a semi-direct product of the torus with the braid group on 3 strands; we also identify the topological monoid of framed local-diffeomorphisms of the torus in similar terms.
It follows that the framed mapping class group is this braid group.
We show that the group of framed diffeomorphisms of the torus acts on twice-iterated Hochschild homology, and explain how this recovers a host of familiar symmetries.  
In the case of Cartesian monoidal structures, we show that this action extends to the monoid of framed local-diffeomorphisms of the torus.
Based on this, we propose a definition of an unstable secondary cyclotomic structure, and show that iterated Hochschild homology possesses such in the Cartesian monoidal setting.

\end{abstract}



\maketitle

\tableofcontents

\begin{ack}
We thank Oscar Randal-Williams and Sam Nariman for input on earlier drafts of this paper.
\end{ack}

\section*{Introduction}

Here are the five main results in this article, all of which are motivated by the study of \bit{factorization homology} as developed in~\cite{old.fact}.  
We direct a reader to the body of the paper for definitions of terms and notation, in particular of the bolded terms, as well as precise statements and proofs.

Regard the 2-torus $\TT^2$ as a framed 2-manifold via a translation-invariant framing.
\begin{itemize}
\item[]
{\bf Theorem~\ref{Theorem A}(2a).}
There is an equivalence between continuous groups:
\[
\TT^2 \rtimes \Braid
\xra{~\simeq~}
\Diff^{\fr}(\TT^2)
~.
\]
This homomorphism is given as follows.
\begin{itemize}

\item
Translation in the group $\TT^2$ defines a continuous homomorphism $\TT^2 \to \Diff^{\fr}(\TT^2)$.

\item
Sheering in each coordinate supplies two extensions from semi-direct products,
\[
\TT^2 \underset{U_1}\rtimes \ZZ 
\longrightarrow 
\Diff^{\fr}(\TT^2)
\longleftarrow
\TT^2 \underset{U_2}\rtimes \ZZ 
~,
\]
where $U_1 = \begin{bmatrix}1 & 1 \\ 0 & 1 \end{bmatrix}$ and $U_2 = \begin{bmatrix}1 & 0 \\ -1 & 1 \end{bmatrix}$,
thereby resulting in a single extension 
\begin{equation}
\label{e100}
\TT^2 \rtimes \lag U_1,U_2\rag \longrightarrow \Diff^{\fr}(\TT^2)
\end{equation}
involving the free group on the two abstract generators $U_1$ and $U_2$.

\item
As there is an equality of matrices $U_1U_2U_1 = \begin{bmatrix} 0 & -1 \\ 1 & 0 \end{bmatrix} = U_2 U_1 U_2$, 
the restrictions of~(\ref{e100}) along the two abstractly isomorphic subgroups $\TT^2 \rtimes \lag U_1U_2U_1 \rag \cong \TT^2 \rtimes \ZZ \cong \TT^2 \rtimes \lag U_2U_1U_2 \rag$ can be identified, thereby supplying a morphism from the coequalizer among continuous groups:
\begin{equation}
\label{e101}
\TT^2 \rtimes
\Braid
~\simeq~
\TT^2 \rtimes \lag U_1,U_2 \mid U_1U_2U_2 = U_2U_1U_2\rag
\longrightarrow
\Diff^{\fr}(\TT^2)
~,
\end{equation}
involving a standard presentation of the braid group on 3 strands.
\\

\end{itemize}

\item[]
{\bf Theorem~\ref{Theorem A}(2b).}
There is an equivalence between continuous monoids:
\[
\TT^2 \rtimes \Ebraid
\xra{~\simeq~}
\Imm^{\fr}(\TT^2)
\]
involving a central extension among monoids
\[
\ZZ
\longrightarrow
\Ebraid
\longrightarrow
\sE_2^+(\ZZ)
~:=~
\Bigl\{
A \in {\sf Mat}_{2\times 2}(\ZZ) \mid
{\sf det}(A) > 0
\Bigr\}
~.
\\
\]

\item[]
{\bf Proposition~\ref{t69}.}
Let $\cX$ be an $\infty$-category.
The morphism $\Ebraid \to \sE_2^+(\ZZ) \to \End_{\sf Groups}( \TT^2)$ determines an action by $\Ebraid$ on the $\infty$-category $\cX^{{\sf g.fin}\TT^2}$ of \bit{genuine finite $\TT^2$-modules} in $\cX$.
A genuine finite $\TT^2$-module in $\cX$ that is coherently invariant with respect to this $\Ebraid$-action is simply a $\Imm^{\fr}(\TT^2)^{\op}$-module in $\cX$ (see Remark~\ref{r16}):
\[
\Mod_{\Imm^{\fr}(\TT^2)^{\op}}(\cX)
~\simeq~
\bigl(\cX^{{\sf g.fin}\TT^2}\bigr)^{\Ebraid}
~.
\]
In particular, there is a forgetful functor:
\[
\Mod_{\Imm^{\fr}(\TT^2)^{\op}}(\cX)
\longrightarrow
\cX^{{\sf g.fin}\TT^2}
~.
\]
We define an \bit{unstable secondary cyclotomic structure} to be an $\Ebraid$-invariant genuine finite $\TT^2$-module.  (See Remark~\ref{r15}.)
\\

\item[]
{\bf Theorem~\ref{t36}.}
Let $\cV$ be a symmetric monoidal $\infty$-category that is $\ot$-presentable.
Let $A$ be a \bit{2-algebra} in $\cV$.
Via factorization homology, there is a canonical action
\[
\TT^2 \rtimes \Braid \simeq \Diff^{\fr}(\TT^2) 
~\lacts~ 
\HHt(A)
\]
on the twice-iterated Hochschild homology of $A$.
This action is given as follows.
\begin{itemize}
\item
The action $\TT^2 \lacts \HHt(A)$ is Connes' cyclic operators.

\item
For $i=1,2$, the extension $\TT^2 \underset{U_i} \rtimes \ZZ \lacts \HHt(A)$ is a canonical sheering action of the Connes cyclic operators.

\item
There is an identification between the actions $\ZZ \underset{U_1 U_2 U_1} \lacts \HHt(A)$ and $\ZZ \underset{U_2 U_1 U_2} \lacts \HHt(A)$, thereby supplying the action $\TT^2 \rtimes \Braid \lacts \HHt(A)$.  
\\

\end{itemize}

\item[]
{\bf Theorem~\ref{t51}.}
Let $\cX$ be a presentable $\infty$-category in which products distribute over colimits.
Regard $\cX$ as a symmetric monoidal $\infty$-category via its Cartesian monoidal structure. 
Let $A$ be a 2-algebra in $\cX$.
Via factorization homology, the twice-iterated Hochschild homology of $A$ is canonically endowed with an unstable secondary cyclotomic structure:
\[
\Bigl(
~
(\TT^2 \rtimes \Ebraid)^{\op} \simeq \Imm^{\fr}(\TT^2)^{\op}
~\lacts~ 
\HHt(A)
~
\Bigr)
~\in~
\bigl(\cX^{{\sf g.fin}\TT^2}\bigr)^{\Ebraid}
~.
\]
In other words, $\HHt(A)$ canonically has the structure of a genuine finite $\TT^2$-module that is $\Ebraid$-invariant.  
\\

\end{itemize}
The remainder of this introduction contextualizes, then restates, these results.

\begin{conventions*}
{\small
\begin{itemize}
\item[~]

\item
We work in the $\infty$-category $\Spaces$ of spaces, or $\infty$-groupoids, an object in which is a \bit{space}.  
This $\infty$-category can be presented as the $\infty$-categorical localization of the ordinary category of compactly-generated Hausdorff topological spaces that are homotopy-equivalent with a CW complex, localized on the weak homotopy-equivalences.  
So we present some objects in $\Spaces$ by naming a topological space.  

\item
By a pullback square among spaces we mean a pullback square in the $\infty$-category $\Spaces$.
Should the square be presented by a homotopy-commutative square among topological spaces, then the canonical map from the initial term in the square to the homotopy-pullback is a weak homotopy-equivalence.  

\item
By a \bit{continuous group} (resp. \bit{continuous monoid}) we mean a group-object (resp. monoid-object) in $\Spaces$.
A continuous monoid $N$ determines a pointed $(\infty,1)$-category $\fB N$, which can be presented by the Segal space $\bDelta^{\op}\xra{\bBar_\bullet(N)}\Spaces$ which is the bar construction of $N$.
For $X \in \cX$ an object in an $\infty$-category, and for $N$ a continuous monoid, 
an \bit{action of $N$ on $X$}, denoted $N \lacts X$, is an extension
$
\lag X \rag 
\colon 
\ast 
\to 
\fB N
\xra{ \lag N \lacts X \rag}
\cX
$.
The $\infty$-category of \bit{(left) $N$-modules in $\cX$} is
\[
\Mod_N(\cX)
~:=~
\Fun( \fB N , \cX)
~.
\]
Every continuous group can be strictified to a topological group (i.e., a group-object in the ordinary category of topological spaces), but maps among such are more flexible (corresponding to maps of loop spaces), as not all topological groups are cofibrant with respect to the usual model structure.

\item
For $G\lacts X$ an action of a continuous group on a space, the space of \bit{coinvariants} is the colimit
\[
X_{/G}
~:=~
\colim\bigl(
\fB G
\xra{~\lag G \lacts X \rag~}
\Spaces
\bigr)
~\in~
\Spaces
~.
\]
Should the action $G\lacts X$ be presented by a continuous action of a topological group on a topological space, then this space of coinvariants can be presented by the homotopy-coinvariants.

\item
We work with $\infty$-operads, as developed in~\cite{HA}.
As such, they are implicitly symmetric.
Some $\infty$-operads are presented as discrete operads, such as ${\sf Assoc}$, while some are presented as topological operads, such as the little 2-disks operad $\cE_2$.

\end{itemize}

}

\end{conventions*}

\subsection{Moduli and isogeny of framed tori}
Here we restate our first result, which identifies the entire symmetries of a framed torus.

The \bit{braid group on 3 strands} can be presented as
\begin{equation}
\label{e67}
\Braid 
~\cong~
\Bigl \lag~ \tau_{1}~,~ \tau_{2}~ \mid ~ \tau_{1}\tau_{2}\tau_{1} ~=~ \tau_{2} \tau_{1} \tau_{2} ~\Bigr\rag
~.
\end{equation}
Through this presentation, there is a standard representation
\begin{equation}
\label{e63}
\Phi
\colon
\Braid
\xra{\bigl\lag~\tau_{1} ~\mapsto ~U_1~,~\tau_{2} ~ \mapsto ~U_2  \bigr\rag}
\GL_2(\ZZ)
~,
\hspace{7pt}
\text{ where }
U_1
~:=~
\begin{bmatrix} 
1 & 1
\\
0 & 1
\end{bmatrix}
\hspace{7pt}
\text{and}
\hspace{7pt}
U_2~:=~
\begin{bmatrix} 
1 & 0
\\
-1 & 1
\end{bmatrix}
~.
\end{equation} 
The homomorphism $\Phi$ defines an action $\Braid \xra{\Phi} \GL_2(\ZZ) \lacts \TT^2$
as a topological group.  
This action defines a topological group:
\[
\TT^2 \rtimes \Braid 
~.
\]

The following result, which is essentially due to Milnor, is the starting point of this paper.  
\begin{prop}
[see~\S10 of \cite{mil}]
\label{t32}
The image of $\Phi$ is the subgroup $\SL_2(\ZZ)$; the kernel of $\Phi$ is central, and is freely generated by the element $(\tau_{1}\tau_{2})^6  \in \Braid$.
Equivalently, $\Phi$ fits into a central extension among groups:
\begin{equation}
\label{e46}
1
\longrightarrow
\ZZ
\xra{~\bigl\lag (\tau_{1}\tau_{2})^6 \bigr\rag~}
\Braid
\xra{~\Phi~}
\SL_2(\ZZ)
\longrightarrow  
1
~.
\end{equation}
Furthermore, this central extension~(\ref{e46}) is classified by the element
\[
\Bigl[\BSL_2(\ZZ) 
{\xra{\sB \bigl( \RR\underset{\ZZ}\ot \bigr)}} \BSL_2(\RR) \simeq \sB^2 \ZZ
\Bigr] 
~\in~ 
\sH^2\bigl( \SL_2(\ZZ) ; \ZZ \bigr)
~.
\]
That is, there is a canonical top horizontal homomorphism defining a pullback among groups:
\[
\xymatrix{
\Braid
\ar@{-->}[rr]
\ar[d]_-{\Phi}
&&
\w{\SL}_2(\RR)  \ar[d]^-{\rm universal~cover}
\\
\SL_2(\ZZ)
\ar[rr]_-{\rm standard}^-{\RR\underset{\ZZ}\ot}
&&
\SL_2(\RR)
.
}
\]

\end{prop}

Consider the subgroup $\GL^+_2(\RR)\subset \GL_2(\RR)$ consisting of those $2\times 2$ matrices with positive determinant -- it is the connected component of the identity matrix.  
Consider the submonoid
\[
\RR\underset{\ZZ}\ot
\colon
\EpZ
~\subset~
\GL^+_2(\RR)
\]
consisting of those $2\times 2$ matrices with positive determinant whose entries are integers.
Consider the pullback\footnote{See Remark~\ref{r4} for an explicit description of the monoid $\Ebraid$.} among monoids:
\begin{equation}
\label{e77}
\xymatrix{
\Ebraid
\ar[d]_-{\Psi} \ar[rr]
&&
\w{\GL}^+_2(\RR)  \ar[d]^-{\rm universal~cover}
\\
\EpZ
\ar[rr]^-{\RR \underset{\ZZ}\ot}
&&
\GL^+_2(\RR)
.
}
\end{equation}
This morphism $\Psi$ supplies a canonical action $\Ebraid \xra{\Psi} \EpZ \lacts \TT^2$ as a topological group.
This action defines a topological monoid
\[
\TT^2 \rtimes \Ebraid
~.
\]
\begin{convention}
\label{r6}
By way of~\S\ref{ambidex}, in particular Corollary~\ref{t61}, we regard all actions of $\Braid$ and $\Ebraid$ as \bit{left}-actions.  
\end{convention}

For $\varphi \colon  \tau_{\TT^2} \cong \epsilon^2_{\TT^2}$ a framing of the torus, we introduce as Definition~\ref{d1} the continuous group of \bit{framed diffeomorphisms}, and the continuous monoid of \bit{framed local-diffeomorphisms} of the torus:
\[
\Diff^{\fr}(\TT^2,\varphi)
\qquad
\text{ and }
\qquad
\Imm^{\fr}(\TT^2,\varphi)
~.
\]
For $\varphi_0$ the \bit{standard framing} of $\TT^2$, which is invariant with respect to translation in the torus, we simply write 
\[
\Diff^{\fr}(\TT^2)
~:=~
\Diff^{\fr}(\TT^2,\varphi_0)
\qquad
\text{ and }
\qquad
\Imm^{\fr}(\TT^2)
~:=~
\Imm^{\fr}(\TT^2,\varphi_0)
~.
\]

\begin{mythm}{X}
\label{Theorem A}

\begin{enumerate}

\item[~]

\item
The map from the set of homotopy-classes of framings of $\TT^2$ to the set of framed-diffeomorphism-types of tori,
\[
\pi_0 \Fr(\TT^2)
\longrightarrow
\pi_0 \bcM^{\fr}_1
~,
\]
is canonically equivalent to the map
\[
\ZZ^2 
\times
\ZZ_{/2\ZZ}
\longrightarrow
\ZZ_{\geq 0}
~,\qquad
\left(
~
\begin{bmatrix}
u \\ v
\end{bmatrix}
~,~
\sigma
~
\right)
\longmapsto
{\sf gcd}(u,v)
~.
\]
Furthermore, a framing $\varphi\in \Fr(\TT^2)$ is homotopic to one that is translation invariant if and only if it is carried to the $0$-component of $\bcM_1^{\fr}$.

\item
Let $\varphi\colon \tau_{\TT^2} \cong \epsilon_{\TT^2} $ be a framing of the torus.

\begin{enumerate}

\item
There is a canonical identification of the continuous group of framed diffeomorphisms of $(\TT^2,\varphi)$
\[
\Diff^{\fr}(\TT^2,\varphi)
~\simeq~
\begin{cases}
\TT^2 \rtimes \Braid
&
,~
\text{if $\varphi$ is homotopic to a translation invariant framing} 
\\
( \TT^2 
\rtimes \ZZ ) \times \ZZ
&
,~
\text{if $\varphi$ is not homotopic to a translation invariant framing} 
\end{cases}
~.
\]
(See~Notation~\ref{d6} for a description of lower semi-direct product.)

\item
There is a canonical identification of the continuous monoid of framed local-diffeomorphisms of $(\TT^2,\varphi)$:
\[
\Imm^{\fr}(\TT^2,\varphi)
~\simeq~
\begin{cases}
\TT^2 \rtimes \Ebraid
&
,~
\text{if $\varphi$ is homotopic to a translation invariant framing} 
\\
\left( \TT^2 
\rtimes (\ZZ \rtimes \NN^\times ) \right) \times \ZZ
&
,~
\text{if $\varphi$ is not homotopic to a translation invariant framing} 
\end{cases}
~.
\]
(See~Notation~\ref{d6} for a description of lower semi-direct products.)

\end{enumerate}

\end{enumerate}

\end{mythm}

Taking path-components, Theorem~\ref{Theorem A}(2a) has the following immediate consequence.
\begin{cor}
\label{t38}
Let $\varphi$ be a framing of the torus.
There is a canonical identification of the framed mapping class group of $(\TT^2,\varphi)$ as a subgroup of the braid group on 3 strands:
\[
{\sf MCG}^{\fr}(\TT^2,\varphi)
~\subset~
\Braid
~.
\]
If $\varphi$ is homotopic with a translation-invariant framing, this subgroup is entire.
If $\varphi$ is not homotopic with a translation-invariant framing, this subgroup is conjugate with a standard subgroup,
\[
{\sf MCG}^{\fr}(\TT^2,\varphi)
~\underset{\rm conjugate}\cong~
\bigl\lag \tau_1 , (\tau_1 \tau_2)^6 \bigr\rag
~\cong~
\ZZ \times \ZZ
~,
\]
which is abstractly isomorphic with $\ZZ\times \ZZ$.

\end{cor}

\begin{remark}
Consider the moduli space $\bcM^{\fr}_1$ of framed tori.  
Theorem~\ref{Theorem A}(1) \& (2a) can be phrased as the assertion that $\bcM^{\fr}_1$ has $\ZZ_{\geq 0}$-many path-components, with the $0$-path-component the space of homotopy-coinvariants ${(\CC\PP^\infty)^2}_{/ \Braid}$ with respect to the action $\Braid \xra{\Phi}\GL_2(\ZZ) \lacts \sB^2 \ZZ^2 \simeq (\CC\PP^\infty)^{\times 2}$, and each other path-component the space ${(\CC\PP^\infty)^2}_{/\ZZ} \times \sB \ZZ$ in which the coinvariants are with respect to the action $\ZZ \xra{\lag U_1\rag }\GL_2(\ZZ) \lacts \sB^2 \ZZ^2 \simeq (\CC\PP^\infty)^{\times 2}.$
A neat result of Milnor (see~\S10 of~\cite{mil}) gives an isomorphism between groups: 
\[
\Braid
~\cong~ 
\pi_1 (\SS^3 \smallsetminus {\sf Trefoil} )
~.
\]
Using that $\SS^3 \smallsetminus {\sf Trefoil}$ is a path-connected 1-type, this isomorphism reveals that the $0$-path-component $(\bcM^{\fr}_1)_{0} \subset \bcM^{\fr}_1$ fits into a fiber sequence of spaces:
\[
(\CC\PP^\infty)^2
\longrightarrow
(\bcM^{\fr}_1)_{0}
\longrightarrow
\bigl( \SS^3 \smallsetminus {\sf Trefoil} \bigr)
~.
\]

\end{remark}

In~\S6 of~\cite{Dehn}, 
Dehn
identifies the oriented mapping class group of a punctured torus with parametrized boundary as the braid group on 3 strands, as it is equipped with a homomorphism to the oriented mapping class group of the torus.  
Through Corollary~\ref{t38}, this results in an identification between these mapping class groups.  
The next result lifts this identification to continuous groups; it is proved in~\S\ref{sec.proofs}.

\begin{cor}
\label{t40}
Fix a smooth framed embedding from the closed 2-disk $\DD^2 \hookrightarrow \TT^2$ extending the inclusion $\{0\} \hookrightarrow \TT^2$ of the identity element.
There are canonical identifications among continuous groups over $\Diff(\TT^2)$:
\[
\Diff^{\fr}(\TT^2~{\sf rel}~0)
~
\simeq
~
\Braid
~\simeq~
\Diff(\TT^2 ~{\sf rel}~ \DD^2)
~.\footnote{This composite equivalence between continuous groups can be witnessed by a span among continuous groups, $\Diff^{\fr}(\TT^2~{\sf rel}~0) \xla{\simeq} \Diff^{\fr}(\TT^2~{\sf rel}~\DD^2) \to \Diff(\TT^2 ~{\sf rel}~ \DD^2)$ in which the leftward map is an equivalence via routine methods.
The more novel aspect of this result can then be rephrased as the rightward map being an equivalence.  
A quick explanation of this fact is that the space of framings of $\TT^2$, fixed at $0\in \TT^2$, has contractible path-components (see Theorem~\ref{Theorem A}(1)).
}
\]
In particular, there are canonical isomorphisms among groups over ${\sf MCG}(\TT^2)$:
\[
{\sf MCG}^{\fr}(\TT^2)
~\cong~
\Braid
~\cong~
{\sf MCG}(\TT^2 \smallsetminus \BB^2 ~{\sf rel}~ \partial)
~,
\]
where $\BB^2 \subset \DD^2$ is the open 2-ball.

\end{cor}

Using Theorem~\ref{Theorem A}(2a), 
the presentation~(\ref{e67}) of the braid group $\Braid$ lends to a simple (fully homotopy coherent) description of an action by $\Diff^{\fr}(\TT^2)$.
We articulate this description as the following result, which is proved at the end of~\S\ref{sec.sheering}, and requires a bit of set-up to state.
\begin{itemize}
\item[]
{\bf Set-up.}
Let $\cX$ be an $\infty$-category.
Let $G$ be a continuous group.
Consider the $\infty$-category $\Mod_{G}(\cX)$ of $G$-modules in $\cX$.
Let $T$ be an automorphism of the continuous group $G$.
Via pullback, $T$ determines an automorphism $T^\ast \colon (G \lacts X) \mapsto  (G \xra{T} G \lacts X)$ of $\Mod_{G}(\cX)$.
Denote the $\infty$-category of $T$-invariant $G$-modules as $\Mod_{G}(\cX)^{\lag T \rag}$, an object in which is a $G$-module $(G \lacts X)$ in $\cX$ together with an identification $(G \xra{T} G \lacts X)  \simeq (G \lacts X)$ between $G$-modules in $\cX$.
Similarly, for $S$ and $T$ automorphisms of $G$, 
the $\infty$-category of $G$-modules that are both $S$ and $T$ invariant is $\Mod_{G}(\cX)^{\lag S,T \rag}$, an object in which is a $G$-module $(G \lacts  X)$ in $\cX$ together with identifications $(G \xra{S} G \lacts X)  \underset{\gamma_S}\simeq (G \lacts X)$ and $(G \xra{T} G \lacts X)  \underset{\gamma_T}\simeq (G \lacts X)$ between $G$-modules in $\cX$.
\end{itemize}
Now, via the standard homomorphism $\GL_2(\ZZ) \to \Aut_{\sf Groups}(\TT^2)$, 
regard the matrices
\[
U_1
=
\begin{bmatrix} 
1 & 1
\\
0 & 1
\end{bmatrix}
\qquad
\text{ and }
\qquad
U_2
=
\begin{bmatrix} 
1 & 0
\\
-1 & 1
\end{bmatrix}
\qquad
\text{ and }
\qquad
R
=
\begin{bmatrix} 
0 & 1
\\
-1 & 0
\end{bmatrix}
\]
as automorphisms of the continuous group $\TT^2$. 
\begin{cor}
\label{t49'}
Let $\cX$ be an $\infty$-category.
There is a pullback diagram among $\infty$-categories:
\[
\xymatrix{
\Mod_{\Diff^{\fr}(\TT^2)}(\cX)
\ar[rr]
\ar[d]
&&
\Mod_{\TT^2}(\cX)^{\lag U_1 , U_2 \rag }
\ar[d]
\\
\Mod_{\TT^2}(\cX)^{ \lag R \rag }
\ar[rr]
&&
\Mod_{\TT^2}(\cX)^{ \lag R , R \rag }
.
}
\]
In particular, for $X\in \cX$ an object, an action $\Diff^{\fr}(\TT^2) \lacts X$ is 
\begin{enumerate}
\item
an action $\TT^2 \underset{\alpha}\lacts X$~,

\item
an identification $\alpha \circ R \underset{\gamma_R}\simeq \alpha$ of this action $\alpha$ with the action $\TT^2 \xra{R} \TT^2 \underset{\alpha}\lacts X$~,

\item
for $i=1,2$, extensions of this identification $\gamma_R$ to identifications $\alpha \circ U_i \underset{\gamma_{U_i}}\simeq \alpha$~.

\end{enumerate}

\end{cor}

A generalization of Smale's conjecture to Haken manifolds, proved by Hatcher 
(see~\cite{hatcher.haken} and~\cite{hatcher}), 
gives that the standard inclusion is an equivalence between continuous groups:
\[
\Aff
\colon
\TT^3 \rtimes \GL_3(\ZZ)
\xra{~\simeq~}
\Diff(\TT^3)
~.
\]
In particular, there is an identification of the mapping class group:
$
{\sf MCG}(\TT^3)
\cong
\GL_3(\ZZ)
$.
Using these identifications, we expect our methods could be used to prove the following.  
\begin{conj}\label{conj.threetorus}
Consider the 3-torus $\TT^3 \cong \RR^3_{/\ZZ^3}$ as it is equipped with its standard framing.
There is a canonical identification between continuous groups:
\[
\Diff^{\fr}(\TT^3)
~\simeq~
\Bigl(
\TT^3
\rtimes 
\Omega \bigl( \SL_3(\RR)_{/\SL_3(\ZZ)} \bigr)
\Bigr) 
\times
\Bigl(
\Omega^2 \SS^3 \times \Omega^3 \SS^3 
\Bigr)^3
\times
\Omega^4 \SS^3
~,
\]
in which the semi-direct product is with respect to the action $\Omega \bigl( \SL_3(\RR)_{/\SL_3(\ZZ)} \bigr) \xra{\rm Puppe} \SL_3(\ZZ) \lacts \TT^3$.
In particular, there is a central extension among groups:
\[
1
\longrightarrow
\ZZ^3 \times (\ZZ_{/2\ZZ})^2
\longrightarrow
{\sf MCG}^{\fr}(\TT^3) 
\longrightarrow
\SL_3(\ZZ)
\longrightarrow
1
~.
\]

\end{conj}

\subsection{Natural symmetries of secondary Hochschild homology}
\label{sec.second}

\subsubsection{Hochschild homology}\label{sec.hoch}

\begin{notation}
Throughout~\S\ref{sec.hoch}, we fix $\cW$ to be a $\ot$-presentable symmetric monoidal $\infty$-category.  

\end{notation}

We briefly recall a definition of the Hochschild homology and record its natural symmetries.
(See~\cite{loday} for a complete account.)
Let $B \in \Alg_{\sf Assoc}(\cW)$ be an associative algebra.
Via left and right translation, regard the underlying object $B \in \cW$ as a $(B,B)$-bimodule.
For $M$ a $(B,B)$-bimodule for $B$, the \bit{Hochschild homology (of $B$ with coefficients in $M$)} is
\[
\sHH(B,M)
~:=~
B
\underset{B^{\op} \ot B} 
\ot
M
~\simeq~
\colim
\Bigl(
\bDelta^{\op}
\xra{B^{\ot \bullet} \ot M}
\cW
\Bigr)
~,
\footnote{
For $0<i<p$, the $i^{th}$ face map of this simplicial object is
\[
B^{\ot \{1,\dots,p\}}
\ot
M
~\simeq~
B^{\ot \{1,\dots, i\}}
\ot
B^{\ot \{i,i+1\}}
\ot
B^{\ot \{i+2,\dots,p\}}
\ot
M
\xra{~\id \ot \mu \ot \id \ot \id~}
B^{\ot \{1,\dots, i\}}
\ot
B
\ot
B^{\ot \{i+2,\dots,p\}}
\ot
M
\]
where $\mu$ is the binary multiplication of $A$;
the $0^{th}$ face map is
\[
B^{\ot \{1,\dots,p\}}
\ot
M
~\simeq~
B^{\{1\}}
\ot
B^{\ot \{2,\dots,p\}}
\ot
M
~\simeq~
B^{\ot \{2,\dots,p\}}
\ot
M
\ot
B^{\{1\}}
\xra{~\id \ot {\sf r.act}~}
B^{\ot \{2,\dots,p\}}
\ot
M
\]
where ${\sf r.act}$ is the right action of $B$ on $M$;
the $p^{th}$ face map is
\[
B^{\ot \{1,\dots,p\}}
\ot
M
~\simeq~
B^{\ot \{1,\dots,p-1\}}
\ot
B^{\{p\}}
\ot
M
\xra{~\id \ot {\sf l.act}~}
B^{\ot \{1,\dots,p-1\}}
\ot
M
\]
where ${\sf l.act}$ is the left action of $B$ on $M$.
}
\]
which can be constructed as the colimit of a simplicial object in $\cW$ naturally associated to the pair $(B,M)$.
This is functorial in the $(B,B)$-bimodule:
\[
{\sf BiMod}_{(B,B)}
\xra{~\sHH(B,-)~}
\cW
~.
\]

The \bit{Hochschild homology (of $B$)} is the instance in which $M=B$ as a $(B,B)$-bimodule:
\[
\sHH(B)
~:=~
B
\underset{B^{\op} \ot B} 
\ot
B
~=:~
\sHH(B,B)
~\simeq~
\bigl|
\bBar_\bullet^{\sf cyc}(B)
\bigr|
~,
\]
which can be constructed as a geometric realization of the \bit{cyclic bar complex} of $B$, as recalled in~\S\ref{sec.hoch1}.
Also recalled in~\S\ref{sec.hoch1} is a canonical action $\TT \simeq \sB \ZZ \lacts \sHH(B)$,
\[
\TT
\xra{~\bigl \lag \TT \lacts \sHH(B) \bigr \rag~}
\Aut_{\cW}\bigl( \sHH(B) \bigr)
~,
\]
which is \bit{Connes' cyclic operator} (see~\cite{connes}), and this is canonically functorial in the argument $B$:
\begin{equation}
\label{e102}
\Alg_{\sf Assoc}(\cW)
\longrightarrow
\Mod_{\TT}(\cW)
~,\qquad
B
\longmapsto
\bigl(
\TT
\lacts
\sHH(B)
\bigr)
~.
\end{equation}

\subsubsection{Secondary Hochschild homology}\label{sec.sechoch}

\begin{notation}
Throughout~\S\ref{sec.sechoch}, we fix $\cV$ to be a $\ot$-presentable symmetric monoidal $\infty$-category.  

\end{notation}

Apply~\S\ref{sec.hoch} to the case $\cW := \Alg_{\sf Assoc}(\cV)$.
For this situation, denote the $\infty$-category
\[
\Alg_2(\cV)
~:=~
\Alg_{\sf Assoc}(\cW) = \Alg_{\sf Assoc}\bigl( \Alg_{\sf Assoc}(\cV) \bigr)
~,\footnote{
Dunn's additivity (see Theorem~\ref{t52})
supplies a host of examples of 2-algebras.
In particular, a commutative algebra canonically determines a 2-algebra.
}
\]
an object in which is a \bit{2-algebra (in $\cV$)},
which is simply an associative algebra in associative algebras in $\cV$.
Using that Hochschild homology is symmetric monoidal, the Hochschild homology of the underlying associative algebra of a 2-algebra retains the structure of an associative algebra.  
For $A$ a 2-algebra in $\cV$, the \bit{secondary Hochschild homology (of $A$)} is the value 
\begin{equation}
\label{e95}
\HHt(A)~:=~\sHH\bigl( \sHH(A) \bigr)
~,
\end{equation}
This is evidently functorial in the 2-algebra, as it is equipped with the {\it two} Connes cyclic operators:
\[
\medmath{
\HHt
~\colon~
\Alg_2(\cV)
~:=~
\Alg_{\sf Assoc}\bigl( \Alg_{\sf Assoc}(\cV) \bigr)
\xra{~\Alg_{\sf Assoc}(\sHH)~}
\Mod_{\TT}\bigl( \Alg_{\sf Assoc}(\cV) \bigr)
\xra{~\sHH~}
\Mod_{\TT} \bigl( \Mod_{\TT} (\cV) \bigr)
~\simeq~
\Mod_{\TT^2}(\cV)
~.
}
\]

\begin{remark}
\label{r14}
In~\S\ref{sec.fact.hmlgy}, we show that our definition~(\ref{e95}) of secondary Hochschild homology (see Definition~\ref{d5} in the body),
agrees with factorization homology over a torus: 
$\HHt(A) \simeq \int_{\TT^2} A$.  
As such, our definition of secondary Hochschild homology is fit to 
receive a \bit{secondary trace} map, which is related to a \bit{secondary Chern character} map, from secondary $\sK$-theory.  
(See~\cite{TV} 
and
~\cite{HSS},
and~\S\ref{sec.traces} below.)

\end{remark}

\begin{warning}
Our definition of secondary Hochschild homology does not appear to agree with the definition introduced by Staic
in
~\cite{st},
and further studied in
~\cite{LJ},
where its cohomological version parametrizes certain algebraic deformations.
Indeed, their definitions are more akin to factorization homology of a pair $\int_{\SS^1 \subset \DD^2}(B \to A)$ (see
~\cite{CSS}, where this is established in the commutative context, in the language of 
higher order Hochschild homology introduced by Pirashvili
~\cite{Pi}), 
which is more similar to factorization homology $\int_{\SS^2} B$ over the 2-sphere.

\end{warning}

Theorem~\ref{Theorem A}(2a) has the following consequence, proved in~\S\ref{sec.fact.hmlgy} using factorization homology.
\begin{mythm}{Y.1}
\label{t36}
Let $A \in \Alg_2(\cV) $ be a $2$-algebra in a $\ot$-presentable symmetric monoidal $\infty$-category $\cV$.
There is a canonical action of the continuous group $\TT^2 \rtimes \Braid$ on secondary Hochschild homology:
\begin{equation}
\label{e49}
\TT^2 \rtimes \Braid
~\lacts~
\HHt(A)
~.
\end{equation}

\end{mythm}

We now explain how Theorem~\ref{t36} extends familiar, or at least expected, symmetries of $\HHt(A)$, and how the action can be phrased in terms of these expected symmetries.

Let $\cW$ be an $\ot$-presentable symmetric monoidal $\infty$-category. 
Let $B$ be an associative algebra in $\cW$.
Each endomorphism $B\xra{\sigma}B$ of the associative algebra $B$ determines a $(B,B)$-bimodule structure $B_\sigma$ on the underlying object $B$, which is characterized by $B\xra{\id}B$ being equivariant with respect to $(B,B)\xra{(\id, \sigma )} (B,B)$.  
This assignment $\sigma\mapsto B_\sigma$ canonically assembles as a functor from the space of endomorphisms of $B$ to the $\infty$-category of $(B,B)$-bimodules:
\[
\End_{\Alg(\cW)}(B)
\longrightarrow
{\sf BiMod}_{(B,B)}
~,\qquad
\sigma
\mapsto 
B_\sigma
~.
\]
This results in a composite functor
\[
\End_{\Alg(\cW)}(B)
\xra{~\sigma\mapsto B_\sigma~}
{\sf BiMod}_{(B,B)}
\xra{~\sHH(B,-)~}
\cW
~,\qquad
\sigma
\mapsto 
\sHH(B,B_\sigma)
~.
\]
This functor restricts to automorphisms of $\id\mapsto \sHH(B,B_{\id}) = \sHH(B)$ as a morphism between continuous groups:
\begin{equation}
\label{e56}
\Omega_{\id} \Aut_{\Alg(\cW)}(B)
~=~
\Omega_{\id} \End_{\Alg(\cW)}(B)
\longrightarrow
\Aut_{\cW}\bigl( \sHH(B) \bigr)
~.
\end{equation}

Now, take $\cW = \Alg(\cV)$ to be the $\infty$-category of associative algebras in an $\ot$-presentable symmetric monoidal $\infty$-category $\cV$, and take $B = \sHH(A)$ to be the Hochschild homology of a $2$-algebra $A\in \Alg_2(\cV) :=\Alg\bigl( \Alg(\cV) \bigr)$. 
The above discussion yields the \bit{sheer symmetry}:
\begin{equation}
\label{e112}
{\sf Sheer}_1
\colon
\ZZ
~\simeq~
\Omega_0 \TT
\xra{~\Omega \bigl \lag \TT \lacts \sHH(A) \bigr \rag ~}
\Omega_{\id} \Aut_{\Alg(\cV)}\bigl( \sHH(A) \bigr)
\xra{~(\ref{e56})~}
\Aut_{\cV}\Bigl( \HHt(A) \Bigr)
~.
\end{equation}
The functoriality of Connes' cyclic operators yield a $\TT^2$-action on secondary Hochschild homology of $A$:
\begin{equation}
\label{e111}
\text{ Connes'}
\colon
\TT^2
\xra{~\bigl \lag \TT^2 \lacts \HHt(A) \bigr \rag~}
\Aut_{\cV}\bigl( \HHt(A) \bigr)
~.
\end{equation}
Corollary~\ref{t45} states that the swapped iteration of Hochschild homology results in the same secondary Hochschild homology.
This yields yet another \bit{sheer symmetry}:
\begin{equation}
\label{e113}
{\sf Sheer}_2
\colon
\ZZ
~\simeq~
\Omega_0 \TT
\xra{~\Omega \bigl \lag \TT \lacts \sHH(A) \bigr \rag ~}
\Omega_{\id} \Aut_{\Alg(\cV)}\bigl( \sHH(A) \bigr)
\xra{~(\ref{e56})~}
\Aut_{\cV}\Bigl( \HHt(A) \Bigr)
~.
\end{equation}
Using Theorem~\ref{t36}, the presentation~(\ref{e67}) of the braid group $\Braid$ lends to the following result, which is proved in~\S\ref{sec.fact.hmlgy}.
\begin{cor}
\label{t55}
Let $A$ be a 2-algebra in $\cV$.
The sheer actions (\ref{e112}) \& (\ref{e113}) and Connes' cyclic operators (\ref{e111}) generate the action $\TT^2 \rtimes \Braid \underset{(\ref{e49})} \lacts \HHt(A)$ of Theorem~\ref{t36}.  
More specifically, the sheer actions and Connes' cyclic operators satisfy the following three relations, thereafter drawing the final conclusion:
\begin{enumerate}

\item
Consider the action defined by the symmetries ${\sf Sheer}_1$ and ${\sf Sheer}_2^{-1}$:
\begin{equation}
\label{e59}
{\sf Sheers}
\colon
\ZZ \amalg \ZZ
~\lacts~
\HHt(A)
~.
~
\footnote{The pushout appearing here is in the category of groups, where it is often referred to as a \emph{free product}.}
\end{equation}
Denoting the generators $\lag \tau_{1} , \tau_{2} \rag = \ZZ \amalg \ZZ$, 
consider the two natural actions
\[
\ZZ
\stackrel[\mathrm{\lag \tau_{2}\tau_{1}\tau_{2} \rag}]{\mathrm{ \lag \tau_{1}\tau_{2} \tau_{1} \rag}}{\overrightarrow{\underrightarrow{\hspace{2cm}}}}
\ZZ \amalg \ZZ
~\underset{(\ref{e59})}\lacts~
\HHt(A)
~.
\]
These two symmetries are coequalized:
\[
\Braid
\underset{(\ref{e67})}{~\cong~}
\Bigl\lag
\tau_1 , \tau_2 
\mid
\tau_1 \tau_2 \tau_1
=
\tau_2 \tau_1 \tau_2
\Bigr\rag
~\lacts~
\HHt(A)
~.
~\footnote{
Phrased more plainly, there is an identification between automorphisms of $\HHt(A)$:
\[
{\sf Sheer}_1
\circ
{\sf Sheer}_2^{-1}
\circ
{\sf Sheer}_1
~\simeq~
{\sf Sheer}_2^{-1}
\circ
{\sf Sheer}_1
\circ
{\sf Sheer}_2^{-1}
~.
\]
}
\]

\item
The actions 
$\ZZ \underset{{\sf Sheer}_1} \lacts \HHt(A)$
and
$\TT^2 \underset{\rm Connes'} \lacts \HHt(A)$
intertwine as an action 
\[
\TT^2
\underset{U_1}\rtimes
\ZZ 
~\lacts~
\HHt(A) 
~,
\]
where this semi-direct product is defined by $\ZZ\xra{\lag  U_1 \rag} \GL_2(\ZZ)\simeq \Aut_{\sf Groups}(\TT^2)$ (see~(\ref{e63})).

\item
The actions 
$\ZZ \underset{{\sf Sheer}_2} \lacts \HHt(A)$
and
$\TT^2 \underset{\text{Connes'}} \lacts \HHt(A)$
intertwine as an action 
\[
\TT^2
\underset{U_2}\rtimes
\ZZ
\underset{\cong}{\xra{\id \rtimes  (-1) }}
\TT^2 
\underset{U_2^{-1}}\rtimes
\ZZ
~\lacts~
\HHt(A) 
~,
\]
where this semi-direct product is defined by $\ZZ\xra{\lag  U_2 \rag} \GL_2(\ZZ)\simeq \Aut_{\sf Groups}(\TT^2)$ (see~(\ref{e63})).

\item[$\bullet$]
Denoting $R := U_1U_2U_2 = \begin{bmatrix} 0 & 1 \\ -1 & 0 \end{bmatrix} = U_2 U_1 U_2 \in  \GL_2(\ZZ)\simeq \Aut_{\sf Groups}(\TT^2)$, the above three points imply the two actions
\[
\TT^2 \underset{R} \rtimes  \ZZ
\stackrel[\mathrm{\id  \rtimes \lag \tau_{2}\tau_{1}\tau_{2} \rag}]{\mathrm{ \id \rtimes \lag \tau_{1}\tau_{2} \tau_{1} \rag}}{\overrightarrow{\underrightarrow{\hspace{2cm}}}}
\TT^2 \underset{U_1,U_2} \rtimes (\ZZ \amalg \ZZ)
~\underset{(\ref{e59})}\lacts~
\HHt(A)
\]
are coequalized under $\TT^2$, thusly generating the action
\[
\TT^2 
\rtimes
\Braid
\underset{\id \rtimes (\ref{e67})}{~\cong~}
\TT
\underset{U_1,U_2}
\rtimes
\Bigl\lag
\tau_1 , \tau_2 
\mid
\tau_1 \tau_2 \tau_1
=
\tau_2 \tau_1 \tau_2
\Bigr\rag
~\lacts~
\HHt(A)
~.
\]

\end{enumerate}

\end{cor}

Next, the short exact sequence~(\ref{e46}) of Proposition~\ref{t32} implies an identification between moduli spaces:
\[
  \left\{
    \begin{tabular}{c}
      extensions of $\Braid\lacts \HHt(A)$ along $\Phi$ 
      \\
      to an action $\SL_2(\ZZ)\lacts \HHt(A)$ 
    \end{tabular}
  \right\}
  ~\simeq~
  \left\{
  	\begin{tabular}{c}
    		trivializations of 
		\\
		$\ZZ\cong \Ker(\Phi) \lacts \HHt(A)$
  	\end{tabular}
  \right \}		
  ~.
\]

\begin{remark}
The action $\ZZ\cong \Ker(\Phi) \lacts \HHt(A)$ is simply an automorphism $\rho \in \Aut_{\cV}\bigl(\HHt(A) \bigr)$.
So an extension of $\Braid \lacts \HHt(A)$ along $\Phi$ to $\SL_2(\ZZ) \lacts \HHt(A)$ exists if and only if there is an equality in the set of path-components of the space of endomorphisms: $[\id_{\HHt(A)}] = [\rho]\in \pi_0 \bigl( \End_{\cV}\bigl(\HHt(A) \bigr) $.
In the case that the ambient $\infty$-category of $\cV$ is stable, this set of path-components has the canonical structure of a ring\footnote{
For example,
let $\Bbbk$ be a commutative ring and take $\cV = (\Mod_{\Bbbk}, \underset{\Bbbk} \ot )$.
Then $\HHt(A)$ may be presented as a projective chain complex over $\Bbbk$;
the ring $\pi_0 (\End_{\cV}\bigl(\HHt(A) \bigr) = \sH_0\bigl(\un{\End}^{\Bbbk}\bigl( \HHt(A) \bigr) \bigr)$ is the $0^{\rm th}$ homology of the chain complex over $\Bbbk$ of self-maps of a such a presentation of $\HHt(A)$.  
} (in which $[\rho]$ is a unit), and so the difference $[\rho]-[\id_{\HHt(A)}] \in \pi_0 (\End_{\cV}\bigl(\HHt(A) \bigr)$ obstructs such an extension to an $\SL_2(\ZZ)$-action.

\end{remark}

So we are interested in identifying the action $\Ker(\Phi)\lacts \HHt(A)$ in familiar, or at least expected, terms.
Corollary~\ref{t54} does just this, in terms of the familiar/expected symmetry of secondary Hochschild homology given by \bit{braiding-conjugation}, as we now explain.  
A starting point for this symmetry is given from the following result which was essentially due to Dunn.
Recall the topological operad $\cE_2$ of little 2-disks.
\begin{theorem}[\cite{dunn}; Theorem~5.1.2.2 of~\cite{HA}]
\label{t52}
There is a canonical equivalence from the $\infty$-category of $\cE_2$-algebras in $\cV$ to that of 2-algebras in $\cV$:
\[
\Alg_{\cE_2}(\cV)
\xra{~\simeq~}
\Alg_2(\cV)
~.
\]

\end{theorem}

After Theorem~\ref{t52}, the standard continuous action $\sO(2) \lacts \cE_2$ on the topological operad immediately implies the following.  
\begin{cor}
\label{t53}
There is a canonical action of the continuous group $\sO(2) \lacts \Alg_2(\cV)$.  
In particular, for each 2-algebra $A$ in $\cV$, 
the orbit map with respect to this action lends to a canonical 
symmetry of $A$:
\[
\beta_A
\colon
\ZZ
~\simeq~
\Omega_{\uno} {\sf SO}(2)
\xra{\simeq}
\Omega_{\uno} \sO(2)
\xra{~\Omega {\sf Orbit}_{A}~}
\Aut_{\Alg_2(\cV)}(A)
~.
\]

\end{cor}

\begin{remark}
This symmetry $\beta_A$ on each 2-algebra $A$ is \bit{braiding-conjugation}.
For instance, this symmetry $\beta_A$ is the identity on the underlying object (so, $\beta_A(1) = \id_A$), and for $\mu\in \cE_2(2)$ it supplies the commutativity of the diagram in $\cV$,
\[
\xymatrix{
A \ot A
\ar[rr]^-{ \id \ot \id}
\ar[d]_-{\mu_A}
&&
A \ot A
\ar[d]^-{\mu_A}
&&
\text{ given by the point }
\\
A
\ar[rr]^-{\id}
&&
A
,
&&
\beta_A(2)\colon \ast \xra{\lag 1 \rag} \ZZ \simeq \Omega_{\mu} \cE_2(2) \to \Omega_{\mu_A} \Hom_{\cV}(A\ot A , A)
~.
}
\]
\end{remark}

The following result is a direct consequence of Observation~\ref{t43}, 
and inspection of the action $\Braid\lacts \HHt(A)$ of Theorem~\ref{t36}, proved in~\S\ref{sec.fact.hmlgy}. 
\begin{cor}
\label{t54}
Let $A$ be a 2-algebra in $\cV$.
Through the action of Theorem~\ref{t36}, the kernel of $\Phi$ acts on $\HHt(A)$ as $\beta_A$.
Specifically, there is a canonically commutative diagram among continuous groups:
\[
\xymatrix{
\ZZ
\ar[d]^-{\cong}_-{\bigl\lag (\tau_1\tau_2)^6 \bigr\rag}
\ar[rrrr]^-{\beta_A}
&&
&&
\Aut_{\Alg_2(\cV)}(A)
\ar[d]^-{\HHt}
\\
\Ker(\Phi)
\ar[rr]
&&
\Braid
\ar[rr]^-{\rm Thm~\ref{t36}}
&&
\Aut_{\cV}\bigl(
\HHt(A)
\bigr)
.
}
\]

\end{cor}

In particular, there is the following immediate consequence of Proposition~\ref{t32}.
\begin{cor}
\label{t37}
Let $A$ be a 2-algebra in $\cV$.
An ${\sf SO}(2)$-invariant-structure on $A\in \Alg_2(\cV)$ 
determines a trivialization of the action $\Ker(\Phi) \lacts \HHt(A)$, and thereafter an extension along $\Phi$ of the actions $\Braid \to \TT^2 \rtimes \Braid \lacts \HHt(A)$ to 
actions
\[
\SL_2(\ZZ) \to \TT^2 \rtimes \SL_2(\ZZ) ~ \lacts~ \HHt(A)
~.  
\]

\end{cor}

\begin{example}
\label{r50}
Here is an example demonstrating that the action $\Braid \lacts \HHt(A)$ does not generally
extend along $\Phi$ as an action $\SL_2(\ZZ) \lacts \HHt(A)$.
Indeed, as a tautologous case, take $A = \Disk^{\fr}_{2/\RR^2}$, regarded as a $2$-algebra in $\Cat_{\infty/\Disk^{\fr}_2}$.
The unstraightening of the functor $\Disk^{\fr}_{2/\TT^2} \xra{\rm forget} \Disk^{\fr}_2 \xra{A} \Cat_{\infty/\Disk^{\fr}_2}$ is the coCartesian fibration
$
\Ar(\Disk^{\fr}_{2/\TT^2}) \xra{\ev_t} \Disk^{\fr}_{2/\TT^2}
,
$
as it is equipped with the functor $\Ar(\Disk^{\fr}_{2/\TT^2}) \xra{\ev_s} \Disk^{\fr}_{2/\TT^2}$.
This functor $\ev_s$ is a localization on the $\ev_t$-coCartesian morphisms.
Using that a colimit of a diagram in $\Cat_{\infty}$ is the localization on the coCartesian morphsims of its unstraightening, there is an equivalence in $\Cat_{\infty/\Disk^{\fr}_2}$:
\[
\int_{\TT^2} \Disk^{\fr}_{2/\RR^2}
~:=~
\colim\Bigl(
\Disk^{\fr}_{2/\TT^2} \xra{\rm forget} \Disk^{\fr}_2 \xra{A} \Cat_{\infty/\Disk^{\fr}_2}
\Bigr)
\xra{~\simeq~}
\Disk^{\fr}_{2/\TT^2}
~,
\]
which is evidently $\Diff^{\fr}(\TT^2)$-equivariant.
We therefore wish to show the action $\Ker(\Phi) \lacts \Disk^{\fr}_{2/\TT^2}$ in $\Cat_{\infty/\Disk^{\fr}_2}$ is not trivializable.
Consider the composite functor
\[
\Cat_{\infty/\Disk^{\fr}_2}
\xra{\sf Mor}
\Spaces_{/{\sf Mor}( \Disk^{\fr}_2 )}
\xra{ {\rm fiber~over~}\un{2} \to \un{1} }
\Spaces_{/\SS^1}
~,
\]
where ${\sf Mor}$ is given by taking spaces of morphisms, and the last functor is given by taking fibers along $\Disk^{\fr}_2 \xra{\pi_0} \Fin$ over the morphism $\un{2} = \{1,2\} \xra{!} \ast = \un{1}$ in $\Fin$, recognizing that ${\sf Mor}\bigl( \Disk^{\fr}_2 \bigr)_{|(\un{2} \to \un{1})} \simeq \SS^1$ is the space of 2-ary operations of the $\infty$-operad $\cE_2$.
Note that this composite functor carries the object of interest
$
\Disk^{\fr}_{2/\TT^2} \in \Cat_{\infty/\Disk^{\fr}_2}
$
to the object in $\Spaces_{/\SS^1}$,
\[
\pr
\colon
\TT^2
\times
\SS^1
~\simeq~
\SS^{\sf fib}(\sT \TT^2)
~\simeq~
{\sf Mor}\bigl( \Disk^{\fr}_{2/\TT^2} \bigr)_{|(\un{2}\to \un{1})}
\longrightarrow
{\sf Mor}\bigl( \Disk^{\fr}_2 \bigr)_{|(\un{2}\to \un{1})}
~\simeq~
\SS^1
~,
\]
involving the unit tangent bundle of $\TT^2$ and its standard framing, which is simply the projection 
Through this identification, the $\Diff^{\fr}(\TT^2)$-action is the canonical one on the unit tangent bundle $\SS^{\sf fib}(\sT \TT^2)$ as it maps to $\SS^1$.  
In particular, the restricted $\ZZ \cong \Ker(\Phi)$-action is generated by the automorphism of $(\TT^2 \times \SS^1 \xra{\pr} \SS^1) \in \Spaces_{/\SS^1}$ that is the diagram
\[
\xymatrix{
\TT^2 \times \SS^1 
\ar[rr]^-{\id}
\ar[dr]_-{\pr}
&&
\TT^2 \times \SS^1
\ar[dl]^-{\pr}
\\
&
\SS^1
&
}
\]
in which the homotopy witnessing commutativity is the image of $1\in \ZZ$ via the ma between spaces:
\[
\ZZ \simeq \Omega_{\sf id} \Map(\SS^1,\SS^1) 
\xra{~\TT^2 \times -~}
\Omega_{\pr} \Map(\TT^2 \times \SS^1 , \SS^1)
~.
\]
It is routine to verify that this map is a monomorhpism.  
In particular, this action $\ZZ \lacts (\TT^2 \times \SS^1) \in \Spaces_{/\SS^1}$ is not trivializable.
Therefore, the action by $\ZZ \cong {\sf Ker}(\Phi)$ on $\int_{\TT^2} \Disk^{\fr}_{2/\RR^2} \in \Cat_{/\Disk^{\fr}_2}$ is not trivializable.
\\

\end{example}

\subsection{Isogenic symmetries of secondary Hochschild homology}
Let $\cX$ be an $\infty$-category.
The action $\Ebraid \to \EZ \lacts \TT^2$ as a topological group determines, via precomposition, an action
\begin{equation}
\label{e86}
\Ebraid
~
\underset{\rm Obs~\ref{t60}}{\overset{(-)^T}\simeq}
~
\Ebraid^{\op}
~\lacts~
\Mod_{\TT^2}( \cX )
~.
\end{equation}
We propose the following.
(See Appendix~A of~\cite{non-com-geom} for a definition of \bit{left-lax invariance}.)
\begin{definition}
\label{d2}
The $\infty$-category of \bit{unstable secondary cyclotomic objects} in an $\infty$-category $\cX$ is that of $\TT^{2}$-modules in $\cX$ that are left-laxly invariant with respect to the action~(\ref{e86}):
\[
{\sf Cyc}^{{\sf un} (2)}(\cX)
~:=~
\Mod_{\TT^2}(\cX)^{{\sf l.lax}\Ebraid}
~.
\]

\end{definition}

\begin{remark}
\label{r15}
Informally, an unstable secondary cyclotomic object in $\cX$ consists of the following.
\begin{itemize}
\item
A $\TT^2$-module $\left( \TT^2 \underset{\alpha}\lacts X \right)$ in $\cX$~.

\item
For each $\w{A}\in \Ebraid$, a morphism between $\TT^2$-modules in $\cX$:
\[
(\w{A}^T)^\ast 
\left(\TT^2 \underset{ \alpha} \lacts X \right)
~:=~
\left(\TT^2 \xra{\Psi(\w{A}^T)} \TT^2 \underset{ \alpha} \lacts X \right)
\xra{~c_{\w{A}}~}
\left(\TT^2 \underset{ \alpha} \lacts X \right)
~.
\]

\item
For each pair $\w{A},\w{B}\in \Ebraid$, 
a commutative square among $\TT^2$-modules in $\cX$:
\[
\xymatrix{
(\w{A}^T)^\ast 
(\w{B}^T)^\ast 
\left(\TT^2 \underset{\alpha} \lacts X \right)
\ar[rr]^-{(\w{A}^T)^\ast c_{\w{B}}}
\ar[d]_-\simeq
&&
(\w{A}^T)^\ast \left(\TT^2 \underset{\alpha} \lacts X \right)
\ar[d]^-{c_{\w{A}}}
\\
\left( (\w{A}\w{B})^T \right)^\ast 
\left(\TT^2 \underset{\alpha} \lacts X \right)
\ar[rr]^-{c_{\w{A} \w{B}}}
&&
\left(\TT^2 \underset{ \alpha} \lacts X \right)
.
}
\]

\item
For each triple $\w{A},\w{B},\w{C} \in \Ebraid$, a similar commutative cube among $\TT^2$-modules in $\cX$ whose faces are (possibly pulled back from) the above commutative squares.  

\item
Etcetera.

\end{itemize}

\end{remark}

After Corollary~\ref{f6} which is proved in Appendix~\ref{sec.A}, 
Theorem~\ref{Theorem A}(2b) implies the following.
\begin{cor}
\label{t30}
For each $\infty$-category $\cX$ there are canonical equivalences among $\infty$-categories over $\cX$: 
\[
{\sf Cyc}^{{\sf un} (2)}(\cX)
~\underset{\rm Cor~\ref{f6}}\simeq~
\Mod_{\left(\TT^2 \rtimes \Ebraid \right)^{\op}}(\cX)
~\underset{\rm Thm~\ref{Theorem A}(2b)}\simeq~
\Mod_{\Imm^{\fr}(\TT^2)^{\op}}(\cX)
~.
\]

\end{cor}

For $\cX$ an $\infty$-category, the $\infty$-category of \bit{finite-genuine $\TT^2$-modules in $\cX$} is
\[
\cX^{{\sf g_{fin}.} \TT^2} ~:=~ \Fun\bigl( ({\sf Orbit}^{\sf fin}_{\TT^2})^{\op} , \cX \bigr)
~,
\]
the $\infty$-category of functors from the opposite of the $\infty$-category ${\sf Orbit}^{\sf fin}_{\TT^2}$ of transitive $\TT^2$-topological spaces with finite isotropy and spaces of $\TT^2$-equivariant maps between them.  
The action $\Ebraid \to \EpZ \lacts \TT^2$ as a topological group supplies an action,
\[
\Ebraid
~\underset{\rm Obs~\ref{t60}}\simeq~
\Ebraid^{\op}
\lacts
{\sf Orbit}_{\TT^2}^{\sf fin}
~,\qquad
A\cdot \TT^2_{/C}
~:=~
\TT^2_{/A^{-1}(C)}
~.
\]
Pre-composition by this action, in turn, supplies an action:
\begin{equation}
\label{e48}
\Ebraid
~\underset{\rm Obs~\ref{t60}}\simeq~
\Ebraid^{\op}
~ \lacts~ 
\cX^{{\sf g}_{\sf fin}. \TT^2}
~.
\end{equation}

After Theorem~\ref{Theorem A}(2b), we have the following immediate consequence of Proposition~\ref{t68}.
\begin{prop}
\label{t69}
For each $\infty$-category $\cX$, the $\infty$-category of finite-genuine $\TT^2$-modules in $\cX$ invariant with respect to~(\ref{e48}) is equivalent with unstable secondary cyclotomic objects in $\cX$:
\[
\Mod_{\Imm^{\fr}(\TT^2)^{\op}}( \cX )
~\underset{\rm Cor~\ref{t30}}\simeq~
{\sf Cyc}^{{\sf un} (2)}(\cX)
\xra{~\simeq~}
\bigl( \cX^{{\sf g}_{\sf fin}. \TT^2}\bigr)^{\Ebraid}
~.
\]
In particular, 
there is a forgetful functor:
\[
\Mod_{\Imm^{\fr}(\TT^2)^{\op}}( \cX )
\underset{\rm Cor~\ref{t30}}\simeq
{\sf Cyc}^{{\sf un} (2)}(\cX)
\longrightarrow
\cX^{{\sf g}_{\sf fin}. \TT^2}
~.
\]

\end{prop}

\begin{remark}
\label{r16}
We explain how Proposition~\ref{t69} asserts a significant cancellation of homotopy coherence data.
\begin{itemize}
\item
A finite-genuine $\TT^2$-module $V$ in $\cX$ is a specification of its \bit{$C$-fixed-points} $V^C \in \Mod_{\frac{\TT^2}{C}}(\cX)$ for each finite subgroup $C \subset \TT^2$ together with coherent compatibility.

\item
For $V$ a finite-genuine $\TT^2$-module in $\cX$, the structure of $V$ to be invariant with respect to the action $\Ebraid \underset{(\ref{e48})} \lacts \cX^{{\sf g}_{\sf fin}. \TT^2}$ is an identification $V^{C} \simeq V^{A^{-1}(C)}$ for each finite subgroup $C \subset \TT^2$ and each element $A \in \Ebraid$, coherently compatibly.  
\end{itemize}
So to name an object in $\bigl( \cX^{{\sf g}_{\sf fin}. \TT^2}\bigr)^{\Ebraid}$ a priori requires an overwhelming wrangling of coherence data.  
From this perspective, Proposition~\ref{t69} is notable: an object in $\bigl( \cX^{{\sf g}_{\sf fin}. \TT^2}\bigr)^{\Ebraid}$ is simply a $\TT^2 \rtimes \Ebraid$-module in $\cX$ -- in particular, no ``genuine'' structure is present.  
Theorem~\ref{t51}, below, is an application of this: via the theory of factorization homology, for $A$ a 2-algebra in $\cX$, its secondary Hochschild homology $\HHt(A)$ easily carries the structure of a $\Imm^{\fr}(\TT^2)^{\op}$-module; through Proposition~\ref{t69}, $\HHt(A)$ then has the structure of a finite-genuine $\TT^2$-module that is $\Ebraid$-invariant.  

\end{remark}

Corollary~\ref{t30} lends to our last main result, which is proved as Section~\S\ref{sec.B2.proof}.

\begin{mythm}{Y.2}
\label{t51}

Let $\cX$ be a presentable $\infty$-category in which finite products distribute over colimits separately in each variable.\footnote{Examples include the $\oo$-categories $\Spaces$, $\Cat_{(\oo,n)}$, $\cX$ an $\oo$-topos.}  
Regard $\cX$ as a symmetric monoidal $\infty$-category via the Cartesian symmetric monoidal structure.  
For each 2-algebra $A\in \Alg_2(\cX)$, 
the action~(\ref{e49}) of Theorem~\ref{t36} canonically extends as an unstable secondary cyclotomic structure:
\begin{equation}
\label{e50}
\Bigl(
~\left( \TT^2 \rtimes \Ebraid \right)^{\op}
~\lacts~
\HHt(A)
~
\Bigr)
~\in~
{\sf Cyc}^{{\sf un} (2)}(\cX)
~.
\end{equation}

\end{mythm}

\begin{remark}
\label{r3}
We explain a relationship between an unstable secondary cyclotomic structure and an iterated unstable cyclotomic structure.
As in the discussion preceding Proposition~\ref{t59}, one can construct a morphism between monoids,
\begin{equation}
\label{e81}
\NN^\times \times \NN^\times
\xra{~\w{\rm diagonals}~}
\Ebraid
~,
\end{equation}
lifting the inclusion $\NN^\times \times \NN^\times \underset{\rm diagonals}\subset \EpZ$ as diagonal matrices.  
With respect to~(\ref{e81}), the product isomorphism $\TT\times \TT \xra{\times} \TT^2$ is equivariant. 
For $\cX$ an $\infty$-category, this results in a forgetful functor from unstable secondary cyclotomic objects to iterated unstable cyclotomic objects:
\begin{equation}
\label{e83}
{\sf Cyc}^{{\sf un} (2)}(\cX)
\longrightarrow
{\sf Cyc}^{\sf un}\bigl( {\sf Cyc}^{\sf un}(\cX) \bigr)
~.
\end{equation}
This functor is generally not an equivalence.\footnote{
Indeed, suppose $\cX$ is an ordinary category.
Then the forgetful functor $\Mod_{\TT^2}(\cX) \xra{\simeq} \cX$ is an equivalence.
Using Proposition~\ref{t59} which identifies the group-completion of the monoid $\Ebraid$, the functor~(\ref{e83}) can then be identified as restriction
$
\Mod_{\w{\GL}_2^+(\QQ)}(\cX)
\to
\Mod_{ (\QQ_{>0}^{\times})^2 }(\cX)
$
along the inclusion $(\QQ_{>0}^{\times})^2 \hookrightarrow \w{\GL}_2^+(\QQ)$ between groups.
}

\end{remark}

\subsection{Remarks on secondary cyclotomic trace}
\label{sec.traces}

We see the role of Corollary~\ref{t30} as informing an approach to secondary cyclotomic traces.  

Let $\Bbbk$ be a commutative ring spectrum.
Let $A\in \Alg_2(\Mod_{\Bbbk})$.
Recall the $\Bbbk$-linear Dennis trace map: $\sK(A) \xra{{\sf tr}} \sHH(A)$~(see, for instance, ~\cite{CT}).
The cyclic trace map is a canonical factorization of this Dennis trace map through \bit{negative cyclic homology}
$
\sK(A)
\xra{{\sf tr}^{\TT}}
{\sHH}^-(A)
:=
\sHH(A)^{\TT}
$
~(see~\cite{GG}).
Iterating this cyclic trace map results in a map between spectra:
$
\sK\bigl( \sK(A) \bigr)
\xra{{\sf tr}^{\TT}( {\sf tr}^{\TT})}
\sHH^-\bigl( \sHH^-(A)\bigr)
.
$
The works of Toen-Vezzosi~\cite{TV},
followed up by the work of
Hoyois--Scherotzke--Sibilla
~\cite{HSS}
(see~Theorem~1.2),
suggest (from the commutative context) that this map can be refined as a \bit{secondary Chern character} map between spectra:
\[
\xymatrix{
\sK^{(2)}(A)
\ar@{-->}[rr]
&&
\HHt(A)^{\TT^2}
\\
\sK\bigl(\sK(A) \bigr)
\ar[u]
\ar[rr]^-{~{\sf tr}^{\TT}( {\sf tr}^{\TT})~}
&&
\sHH^-\bigl( \sHH^-(A) \bigr)
\ar[u]
.
}
\]
from \bit{secondary $K$-theory} to the $\TT^2$-invariants of secondary Hochschild homology.
We expect the work of Mazel-Gee--Stern~\cite{MGS} (in particular Theorem C (see~\S0.4.4)) on universal properties of secondary $\sK$-theory to yield a solution both to this, and the following.
\begin{conj}
\label{c1}
For each 2-algebra $A$ over $\Bbbk$, there is a canonical filler in the diagram among spectra:
\[
\xymatrix{
\sK^{(2)}(A)
\ar@{-->}[rr]
&&
\HHt(A)^{\TT^2 \rtimes \Braid} \ar[rr] 
&&
\HHt(A)^{\TT^2}
\\
\sK\bigl(\sK(A) \bigr)
\ar[u]
\ar[rrrr]^-{~{\sf tr}^{\TT}( {\sf tr}^{\TT})~}
&&
&&
\sHH^-\bigl( \sHH^-(A) \bigr)
.
\ar[u]
}
\]

\end{conj}

For the case in which $\Bbbk = \SS$ is the sphere spectrum, where standard notation is ${\sf THH}: = \sHH$ and referred to as \bit{topological Hochschild homology}, 
the cyclic trace map factors further as the \bit{cyclotomic trace} map,
\begin{equation}
\label{e94}
\sK(A) \xra{~{\sf tr}^{\sf Cyc}~} {\sf TC}(A) := {\sf THH}(A)^{\sf Cyc}
~,
\end{equation}
through the \bit{topological cyclotomic homology}
which is the \bit{cylotomic}-invariants with respect to a canonical \bit{cyclotomic structure} on topological Hochschild homology.  
The fantastic culminating result of~\cite{DGM} articulates a sense in which this cyclotomic trace map~(\ref{e94}) is locally constant (in the algebra $A$).
Iterating this cyclotomic trace map results in a map between spectra:
$
\sK\bigl( \sK(A) \bigr)
\xra{{\sf tr}^{\sf Cyc}( {\sf tr}^{\sf Cyc})}
{\sf TC}\bigl( {\sf TC}(A)\bigr)
.
$
which is not locally constant (in the 2-argument $A$).
As above, we expect that this iterated cyclotomic trace map can be refined as a map between spectra:
\[
\xymatrix{
\sK^{(2)}(A)
\ar@{-->}[rr]
&&
{\sf THH}^{(2)}(A)^{{\sf Cyc}\times {\sf Cyc}}
\\
\sK\bigl(\sK(A) \bigr)
\ar[u]
\ar[rr]^-{{{\sf tr}^{\sf Cyc}( {\sf tr}^{\sf Cyc})}}
&&
{\sf TC}\bigl( {\sf TC}(A)\bigr)
.
\ar[u]
}
\]
Following the developments in~\cite{cyclo}, we expect Definition~\ref{d2} of an unstable cyclotomic object to lend to a definition of a \bit{(stable) secondary cyclotomic object}, and that Theorem~
\ref{t51} lends a secondary cyclotomic structure on secondary topological Hochschild homology.
For secondary topological cyclotomic homology to be the invariants with respect to this structure, 
${\sf TC}^{(2)}(A):={\sf THH}^{(2)}(A)^{\sf Cyc^{(2)}}$, we again expect the work of Mazel-Gee--Stern (\cite{MGS}, in particular Theorem~C (see~\S0.4.4)) on secondary $\sK$-theory to further lend a secondary cyclotomic trace map, which we state as the following.
\begin{problem}
\label{c2}
Define (stable) secondary cyclotomic structure, and show that secondary topological Hochschild homology canonically possesses such.
Show that the iterated cyclotomic trace map factors through the secondary topological cyclotomic homology, compatibly with the factorization of Conjecture~\ref{c1}:
\[
\xymatrix{
\sK^{(2)}(A)
\ar@(d,-)[ddr]_-{\rm Conj~\ref{c1}}
\ar@{-->}[dr]^-{{\sf tr}^{\sf Cyc^{(2)}}}
&
&
\sK\bigl(\sK( A) \bigr)
\ar[ll]
\ar[dr]^-{~{\sf tr}^{\sf Cyc}( {\sf tr}^{\sf Cyc})~}
\\
&
{\sf TC}^{(2)}(A)
\ar[r]
\ar[d]
&
{\sf THH}^{(2)}(A)^{{\sf Cyc}\times {\sf Cyc}}
\ar[d]
&
{\sf TC}\bigl( {\sf TC}(A) \bigr)
\ar[l]
\ar[d]
\\
&
{\sf THH}^{(2)}(A)^{\TT^2 \rtimes  \Braid}
\ar[r]
&
{\sf THH}^{(2)}(A)^{\TT^2}
&
{\sf THH}^-\bigl( {\sf THH}^-(A) \bigr)
.
\ar[l]
}
\]
\end{problem}

\begin{remark}
One might be encouraged by Remark~\ref{r3} to expect that the secondary cyclotomic trace map ${\sf tr}^{{\sf Cyc}^{(2)}}$ of Conjecture~\ref{c1} is locally-constant (in the 2-algebra $A$), thereby correcting the failure for the iterated cyclotomic trace map ${\sf tr}^{\sf Cyc}( {\sf tr}^{\sf Cyc})$ to be locally-constant. 
However, we do not expect for this to be so.
Namely, the local-constancy of the cyclotomic trace map $\sK(A) \xra{{\sf tr}^{\sf Cyc}} {\sf TC}(A)$ relies in an essential way on calculations of Hesselholt (\cite{L}), which identify the fiber of the canonical map ${\sf TC}(V\rtimes A) \to {\sf TC}(A)$ associated to a square-zero extension of $A$.  
These calculations in turn rely on the fact that, for each $i \geq 0$, the canonical action $\TT \simeq \Diff^{\fr}(\TT)\lacts {\sf Conf}_i(\TT)_{\Sigma_i}$ on unordered configuration space canonically factors as a $\TT_{/\sC_i}$-torsor.  
Because the canonical action $\TT^2 \rtimes \Braid \simeq \Diff^{\fr}(\TT^2) \lacts {\sf Conf}_{i}(\TT^2)_{\Sigma_i}$ does not apparently have any such a property, 
we do not expect for the secondary cyclotomic trace map of Problem~\ref{c2} to be locally-constant.  

\end{remark}

\section{Moduli and isogeny of framed tori}

\subsection{Moduli and isogeny of tori}

Vector addition, as well as the standard vector norm, gives $\RR^2$ the structure of a topological abelian group.
Consider its closed subgroup $\ZZ^2\subset \RR^2$.  
The \bit{torus} is the quotient in the short exact sequence of topological abelian groups:
\[
0
\longrightarrow
\ZZ^2
\xra{\rm inclusion}
\RR^2 
\xra{~\quot~}
\TT^2
\longrightarrow
0
~.
\] 
Because $\RR^2$ is connected, and because $\ZZ^2$ acts cocompactly by translations on $\RR^2$, the torus $\TT^2$ is connected and compact.  
The quotient map $\RR^2 \xra{\quot} \TT^2$ endows the torus with the structure of a Lie group, and in particular a smooth manifold.
Consider the submonoid
\begin{equation*}
\EZ
~:=~
\Bigl\{
\ZZ^2 \xra{A} \ZZ^2 \mid {\sf det}(A) \neq 0
\Bigr\}
~\subset~
\End_{\sf Groups}(\ZZ^2)
~,
\end{equation*}
consisting of the cofinite endomorphisms of the group $\ZZ^2$.  
Using that the smooth map $\RR^2 \xra{\quot} \TT^2$ is a covering space and $\TT^2$ is connected, there is a canonical continuous action on the topological group:
\begin{equation}
\label{e2}
\EZ
~\lacts~
\TT^2
~,\qquad
A  q
~:=~
\quot( A\w{q} )
\qquad
{\Small \bigl(\text{for any }\w{q} \in \quot^{-1}(q) \bigr) }
~.
\footnote{
Note that (\ref{e2}) indeed does not depend on $\w{q} \in \quot^{-1}(q)$.
}
\end{equation}
This action~(\ref{e2}) defines a semi-direct product topological monoid:
\[
\TT^2 \rtimes \EZ
~.
\]

Consider the topological monoid of smooth local-diffeomorphisms of the torus:
\[
\Imm(\TT^2)
~\subset~
\Map(\TT^2 , \TT^2)
~,
\] 
which is endowed with the subspace topology of the $\sC^\infty$-topology on the set of smooth self-maps of the torus.
Notice the morphism between topological monoids:
\begin{equation}
\label{e7}
\Aff
\colon
\TT^2 \rtimes \EZ 
\longrightarrow
\Imm(\TT^2)
~,\qquad
(p, A)
\mapsto
\Bigl(
q\mapsto 
Aq + p
\Bigr)
~.
\end{equation}

\begin{observation}
\label{t21}
\begin{enumerate}

\item[~]

\item
The standard inclusion $\GL_2(\ZZ) \hookrightarrow \EZ$ witnesses the maximal subgroup.
It follows that the standard inclusion $ \TT^2 \rtimes \GL_2(\ZZ)  \hookrightarrow \TT^2 \rtimes \EZ$ witnesses the maximal subgroup, both as topological monoids and as continuous monoids.

\item
The standard monomorphism $\Diff(\TT^2) \hookrightarrow \Imm(\TT^2)$ witnesses the maximal subgroup, both as topological monoids and as continuous monoids.

\end{enumerate}
\end{observation}

We record the following classical result.
\begin{lemma}\label{t1}
The morphism~(\ref{e7}) restricts to maximal subgroups as a homotopy-equivalence:
\[
\Aff
\colon 
\TT^2 \rtimes \GL_2(\ZZ)
\xra{~\simeq~}
\Diff^{\fr}(\TT^2)
~,\qquad
(p, A)
\mapsto
\Bigl(
q\mapsto 
Aq + p
\Bigr)
~.
\]

\end{lemma}

\begin{proof}
Let $G$ be a locally path-connected topological group, which we regard as a continuous group.
Denote by $G_{\uno} \subset G$ the path-component containing the identity element in $G$.
This subspace $G_{\uno}\subset G$ is a normal subgroup, and the sequence of continuous homomorphisms
\[
1
\longrightarrow
G_{\uno}
\xra{~\rm inclusion~}
G
\xra{~\rm quotient~}
\pi_0(G)
\longrightarrow
1
\]
is a fiber-sequence among continuous groups.
This fiber sequence is evidently functorial in the argument $G$.
In particular, there is a commutative diagram among topological groups,
\[
\xymatrix{
1 \ar[d]_-= \ar[r]
&
\TT^{2} = \bigl( \TT^{2} \rtimes \GL_2(\ZZ) \bigr)_{\uno}
\ar[d]_-{\Aff_{\uno}}
\ar[r]^-{\rm inc}
&
\TT^{2} \rtimes \GL_2(\ZZ)  
\ar[d]_-{\Aff} 
\ar[r]^-{\rm quot}
&
\pi_0 \bigl( \TT^{2} \rtimes  \GL_2(\ZZ) \bigr)
=
\GL_2(\ZZ)
\ar[d]_-{\pi_0(\Aff)}
\ar[r]
&
1 \ar[d]_-=
\\
1
\ar[r]
&
\Diff(\TT^2)_{\uno}
\ar[r]^-{\rm inc}
&
\Diff(\TT^2) 
\ar[r]^-{\rm quot}
&
\pi_0 \bigl( \Diff(\TT^2) \bigr)
\ar[r]
&
1
,
}
\] 
in which the horizontal sequences are fiber sequences.  
By the 5-lemma applied to homotopy groups, we are reduced to showing the vertical homomorphisms $\Aff_{\uno}$ and $\pi_0(\Aff)$ are homotopy equivalences.

Theorem 2.D.4 of~\cite{rolf}, along with Theorem B of \cite{TT}, implies $\pi_0(\Aff)$ is an isomorphism.  
So it remains to show $\Aff_{\uno}$ is a homotopy equivalence.\footnote{See \cite{gramain}. We include a proof for the convenience of the reader.} 
With respect to the canonical continuous action $\Diff(\TT^2)_{\uno} \lacts \TT^2$, 
the orbit of the identity element $0\in \TT^2$ is the evaluation map 
\[
\ev_0
\colon 
\Diff(\TT^2)_{\uno}
\longrightarrow
\TT^2
~.
\]
Note that the composition,
\[
\id
\colon
\TT^2
\xra{~\Aff_{\uno}~}
\Diff(\TT^2)_{\uno}
\xra{~\ev_0~}
\TT^2
~,
\]
is the identity map.
So it remains to show that the homotopy-fiber of $\ev_0$ is weakly-contractible.  
The isotopy-extension theorem implies $\ev_0$ is a Serre fibration.  
So it is sufficient to show the fiber of $\ev_0$, which is the stabilizer
${\sf Stab}_0\bigl(\Diff(\TT^2)_{\uno}\bigr)$,
is weakly-contractible.  
Finally, Theorem 1b of~\cite{ee} states that this stabilizer is contractible.  

\end{proof}

\begin{remark}
\label{r10}
By the classification of compact surfaces, the moduli space $\bcM_1$ of smooth tori is path-connected, and as so is 
\[
\bcM_1
~\simeq~
{\sf BDiff}(\TT^2)
\underset{\rm Lem~\ref{t1}}{~\simeq~}
\sB\bigl(
\TT^2 \rtimes \GL_2(\ZZ)
\bigr)
~\simeq~
{( \CC\PP^\infty)^2}_{/\GL_2(\ZZ)}
\]
in which the quotient is with respect to the standard action $\GL_2(\ZZ) \lacts \sB^2 \ZZ^2 \simeq (\CC\PP^\infty)^2$.
In particular, this path-connected moduli space fits into a fiber sequence
\[
(\CC\PP^\infty)^2
\longrightarrow
\bcM_1
\longrightarrow
\BGL_2(\ZZ)
~.
\]

\end{remark}

Consider the set
$
\bcL(2) 
:=
\Bigl\{
\Lambda
\underset{\rm cofin}
\subset
\ZZ^2
\Bigr\}
$
of \bit{cofinite subgroups} of $\ZZ^2$.
\begin{observation} \label{t99}

\begin{enumerate}

\item[~]

\item
The orbit-stabilizer theorem immediately implies the composite map 
$
\TT^2 \rtimes \EZ
\xra{\pr}
\EZ
\xra{\rm Image} 
\bcL(2)
$
witnesses the quotient:
\[
\bigl(
\TT^2 \rtimes \EZ
\bigr)_{/\TT^2 \rtimes \GL_2(\ZZ)}
\xra{~\cong~}
\EZ_{/\GL_2(\ZZ)}
\xra{~\cong~}
\bcL(2)
~.
\]

\item
Using that each finite-sheeted cover over $\TT^2$ is diffeomorphic with $\TT^2$, the classification of covering spaces implies the map given by taking the image of homology
$
\Imm(\TT^2)
\xra{{\sf Image}(\sH_1)}
\bcL(2)
$
witnesses the quotient:
\[
\Imm(\TT^2)_{/\Diff(\TT^2)}
\xra{~\cong~}
\bcL(2)
~.
\]

\item
The diagram
\[
\xymatrix{
\TT^2 \rtimes \EZ 
\ar[rr]^-{\Aff}
\ar[d]_-{\pr}
&&
\Imm(\TT^2) 
\ar[d]^-{{\rm Image}(\sH_1)}
\ar[dll]_-{\sH_1}
\\
\EZ
\ar[rr]^-{\rm Image}
&&
\bcL(2)
}
\]
commutes.

\end{enumerate}

\end{observation}

\begin{cor}
\label{t22}
The morphism (\ref{e7}) between topological monoids is a homotopy-equivalence:
\[
\Aff
\colon
\TT^2 \rtimes \EZ
\xra{~\simeq~}
\Imm(\TT^2)
~.
\]

\end{cor}

\begin{proof}

Consider the morphism between fiber sequences in the $\infty$-category $\Spaces$:
\[
\xymatrix{
 \TT^2 \rtimes\EZ
\ar[rr]^-{\rm quotient}
\ar[d]_-{\Aff}
&&
\bigl(
\TT^2 \rtimes \EZ
\bigr)_{\TT^2 \rtimes \GL_2(\ZZ)}
\ar[rr]
\ar[d]^-{\Aff_{\Aff}}
&&
\sB\bigl( \TT^2 \rtimes \GL_2(\ZZ) \bigr)
\ar[d]^-{\sB \Aff}
\\
\Imm(\TT^2)
\ar[rr]^-{\rm quotient}
&&
\Imm(\TT^2)_{/\Diff(\TT^2)}
\ar[rr]
&&
\sB \Diff(\TT^2)
.
}
\]
Lemma~\ref{t1} implies the right vertical map is an equivalence.
Observation~\ref{t99} implies the middle vertical map is an equivalence.
It follows that the left vertical map is an equivalence, as desired.

\end{proof}

\subsection{Framings}

A \bit{framing} of the torus is a trivialization of its tangent bundle: $\varphi\colon  \tau_{\TT^2} \cong  \epsilon^2_{\TT^2}$.
Consider the topological \bit{space of framings} of the torus:
\[
\Fr(\TT^2)
~:=~
\Iso_{\Bdl_{\TT^2}}\bigl( \tau_{\TT^2} , \epsilon^2_{\TT^2}  \bigr)
~\subset~
\Map(  \sT \TT^2  , \TT^2 \times \RR^2 )
~,
\]
which is endowed with the subspace topology of the $\sC^\infty$-topology on the set of smooth maps between total spaces.  
The quotient map $\RR^2\xra{\quot}\TT^2$ endows the smooth manifold $\TT^2$ with a \bit{standard framing} $\varphi_0$:
for 
\[
\trans\colon \TT^2 \times \TT^2 
\xra{~(p,q)\mapsto \trans_p(q) := p+q~} 
\TT^2
\]
the abelian multiplication rule of the Lie group $\TT^2$ is
\[
(\varphi_0)^{-1}
\colon
\epsilon^2_{\TT^2}
\xra{~\cong~}
\tau_{\TT^2}
~,\qquad
\TT^2 \times \RR^2 \ni (p,v)
\mapsto
\bigl(p,\sD_{0}(\trans_p \circ \quot) (v) \bigr)
\in \sT \TT^2
\]
where $\sD_0$ is differentiation at zero.

The next sequence of observations culminates as an identification of this space of framings.
\begin{observation}
\label{t20}

\begin{enumerate}
\item[~]

\item
Postcomposition gives the topological space $\Fr(\TT^2)$ the structure of a torsor for the topological group $\Iso_{\Bdl_{\TT^2}}\bigl( \epsilon^2_{\TT^2} , \epsilon^2_{\TT^2} \bigr)$.
In particular, the orbit map of a framing $\varphi \in \Fr(\TT^2)$ is a homeomorphism:
\begin{equation}
\label{e41}
\Iso_{\Bdl_{\TT^2}}\bigl( \epsilon^2_{\TT^2} , \epsilon^2_{\TT^2} \bigr)
\xra{~\cong~}
\Fr(\TT^2)
~,\qquad
\alpha \mapsto 
\alpha \circ \varphi 
~.
\end{equation}

\item
Consider the topological space $\Map \bigl( \TT^2 , \GL_2(\RR) \bigr)$ of smooth maps from the torus to the standard smooth structure on $\GL_2(\RR)$, which is endowed with the $\sC^\infty$-topology. 
The map
\begin{equation}
\label{e43}
\Map \bigl( \TT^2 , \GL_2(\RR) \bigr)
\xra{~\cong~}
\Iso_{\Bdl_{\TT^2}}\bigl( \epsilon^2_{\TT^2} , \epsilon^2_{\TT^2} \bigr)
~,\qquad
a
\mapsto 
\Bigl(
\TT^2\times \RR^2
\xra{\bigl( p,v \bigr) \mapsto \bigl( p,a_p(v) \bigr) }
\TT^2 \times \RR^2
\Bigr)
~,
\end{equation}
is a homeomorphism.

\item
The map to the product,
\begin{equation}
\label{e44}
\Map \bigl( \TT^2 , \GL_2(\RR) \bigr)
\xra{~\cong~}
\Map\Bigl( ( 0\in \TT^2) , ( \uno \in \GL_2(\RR) \bigr) \Bigr)
\times
\GL_2(\RR)
~,\qquad
a
\mapsto 
\bigl(
~
a(0)^{-1} a~ ,~ a(0)
~
\bigr)
~,
\end{equation}
is a homeomorphism.

\item
Because both of the spaces $\TT^2$ and $\GL_2(\RR)$ are 1-types with the former path-connected, 
the map,
\[
\pi_1
\colon
\Map\Bigl( ( 0\in \TT^2) , ( \uno \in \GL_2(\RR) \bigr) \Bigr)
\xra{~\simeq~}
{\sf Hom}\Bigl( \pi_1\bigl( 0 \in \TT^2 \bigr) , \pi_1\bigl( \uno \in \GL_2(\RR) \bigr) \Bigr) 
~,
\]
is a homotopy-equivalence.

\item
Evaluation on the standard basis for $\pi_1(0\in \TT^2) \xra{\cong} \pi_1(0\in \TT)^2 \cong \ZZ^2$ defines a homeomorphism:
\begin{equation}
\label{e45}
{\sf Hom}\Bigl( \pi_1\bigl( 0 \in \TT^2 \bigr) , \pi_1\bigl( \uno \in \GL_2(\RR) \bigr) \Bigr) 
\xra{~\cong~}
\pi_1\bigl( \uno \in \GL_2(\RR)^2 \bigr)
~\cong~
\ZZ^2
~.
\end{equation}

\end{enumerate}

\end{observation}

Observation~\ref{t20}, together with the Gram--Schmidt homotopy-equivalence ${\sf GS}\colon\sO(2) \xra{\simeq} \GL_2(\RR)$, yields the following.
\begin{cor}
\label{t25}
A framing $\varphi \in \Fr(\TT^2)$ determines a composite homotopy-equivalence:
\begin{eqnarray*}
\Fr(\TT^2) 
&
\underset{\simeq}{
\xla{~(\ref{e43})\circ (\ref{e41})~}
}
&
\Map\bigl( \TT^2 , \GL_2(\RR) \bigr)
\\
\nonumber
&
\underset{\simeq}{
\xra{~(\ref{e44})~}
}
&
\Map \Bigl( \bigl( 0 \in \TT^2 \bigr) , \bigl( \uno \in \GL_2(\RR) \bigr) \Bigr) \times \GL_2(\RR) 
\\
\nonumber
&
\underset{\simeq}{
\xra{~\pi_1 \times \id ~}
}
&
{\sf Hom}\Bigl( \pi_1\bigl( 0 \in \TT^2 \bigr) , \pi_1\bigl( \uno \in \GL_2(\RR) \bigr) \Bigr) \times \GL_2(\RR) 
\\
\nonumber
&
\underset{\simeq}{
\xra{(\ref{e45}) \times \id }
}
&
\ZZ^2 \times \GL_2(\RR)
\\
\nonumber
&
\underset{\simeq}{
\xla{ \id \times {\sf GS}}
}
&
\ZZ^2 \times \sO(2)
~.
\end{eqnarray*}

\end{cor}

\begin{notation}
\label{d7}
We denote the values of the homotopy-equivalence of Corollary~\ref{t25} applied to the standard framing $\varphi_0\in \Fr(\TT^2)$:
\[
\Fr(\TT^2)
\xra{~\simeq~}
\ZZ^2 \times \GL_2(\RR)
~,\qquad
\varphi
\longmapsto
\bigl(
~
\vec{\varphi}
~,~
B_\varphi
~
\bigr)
~.
\]

\end{notation}

\subsection{Moduli of framed tori}

Consider the map:
\[
\Act\colon
\Fr(\TT^2)
\times
\Imm(\TT^2)
\longrightarrow
\Fr(\TT^2)
~,
\]
\[
( \varphi , f )
\mapsto 
\Bigl(
~
\tau_{\TT^2} 
\underset{\cong}{\xra{\sD f}}
f^\ast \tau_{\TT^2}
\underset{\cong}{ \xra{f^\ast \varphi} }
f^\ast \epsilon^2_{\TT^2}
=
\epsilon^2_{\TT^2}
~
\Bigr)
~.
\]

\begin{lemma}
\label{t50}
The map $\Act$ is a continuous right-action of the topological monoid $\Imm(\TT^2)$ on the topological space $\Fr(\TT^2)$.
In particular, there is a continuous action of the topological group $\Diff(\TT^2)$ on the topological space $\Fr(\TT^2)$. 

\end{lemma}

\begin{proof}
Consider the topological subspace of the topological space of smooth maps between total spaces of tangent bundles, which is endowed with the $\sC^\infty$-topology,
\[
{\sf Bdl}^{\sf fw.iso}(\tau_{\TT^2} , \tau_{\TT^2})
~\subset~
\Map\bigl( \sT \TT^2 , \sT \TT^2 \bigr)
~,
\]
consisting of the smooth maps between tangent bundles that are fiberwise isomorphisms.
Notice the factorization
\[
\Act\colon
\Fr(\TT^2)
\times
\Imm(\TT^2)
\xra{\id \times \sD}
\Fr(\TT^2)
\times 
{\sf Bdl}^{\sf fw.iso}(\tau_{\TT^2} , \tau_{\TT^2})
\xra{\circ}
\Fr(\TT^2)
\]
as first taking the derivative, followed by composition of bundle morphisms.  
The definition of the $\sC^\infty$-topology is so that the first map in this factorization is continuous. 
The second map in this factorization is continuous because composition is continuous with respect to $\sC^\infty$-topologies.
We conclude that $\Act$ is continuous.  

We now show that $\Act$ is an action.
Clearly, for each $\varphi \in \Fr(\TT^2)$, there is an equality $\Act( \varphi  , \id) = \varphi$.
Next, let $g,f\in \Imm(\TT^2)$, and let $\varphi \in \Fr(\TT^2)$.
The chain rule, together with universal properties for pullbacks, gives that the diagram among smooth vector bundles
{\Small
\[
\xymatrix{
\tau_{\TT^2}
\ar@(u,u)[rrrrrr]^-{\sD (g\circ f)}
\ar[rr]_-{\sD g}
&&
g^\ast 
\tau_{\TT^2}
\ar[rr]_-{g^\ast \sD f}
&&
f^\ast g^\ast \tau_{\TT^2}
\ar[d]^-{f^\ast g^\ast \varphi}
\ar[rr]_-\cong
&&
(g\circ f)^\ast \tau_{\TT^2}
\ar[d]^-{(g \circ f)^\ast \varphi}
\\
\epsilon^2_{\TT^2}
&&
g^\ast \epsilon^2_{\TT^2}
\ar[ll]_-\cong
&&
f^\ast g^\ast \epsilon^2_{\TT^2}
\ar[ll]_-\cong
&&
(g \circ f)^\ast \epsilon^2_{\TT^2}
\ar[ll]_-\cong
\ar@(d,d)[llllll]^\cong
}
\]
}
commutes. 
Inspecting the definition of $\Act$, the commutativity of this diagram implies the equality $\Act\bigl( \Act(\varphi, g) , f \bigr) = \Act( \varphi , g\circ f)$, as desired.

\end{proof}

\begin{definition}
\label{r8}
The \bit{moduli space of framed tori}\footnote{This definition is a particular case of a general definition of a moduli space of framed manifolds; see, for instance, \cite{old.fact}.
} is the space of homotopy-coinvariants with respect to this conjugation action $\Act$:
\[
\bcM_1^{\fr}
~:=~
\Fr(\TT^2)_{/\Diff(\TT^2)}
~.
\]
\end{definition}

\begin{observation}\label{t4}
Through Corollary~\ref{t25} applied to the standard framing $\varphi_0 \in \Fr(\TT^2)$, the action $\Act$ is compatible with familiar actions.
Specifically, $\Act$ fits into a commutative diagram among topological spaces:
\begin{equation*}
\xymatrix{
\Fr(\TT^2)
\times
\Imm(\TT^2)
\ar[rrr]^-{\Act}
&
&&
\Fr(\TT^2)
\\
\Map\bigl( \TT^2 , \GL_2(\RR) \bigr)
\times
\bigl(
\TT^2 \rtimes \EZ
\bigr)
\ar[u]^-{ {\rm Cor}~\ref{t25} \times \Aff}_-{\simeq}
\ar[r]^-{ \id \times \pr }
\ar[d]_-{ {\rm Cor}~\ref{t25} \times \id}^-{\simeq}
&
\Map\bigl( \TT^2 , \GL_2(\RR) \bigr)
\times
\EZ 
\ar[rr]^-{\rm value-wise}_-{\rm multiply}
\ar[d]^-{{\rm Cor}~\ref{t25} \times \id}_-{\simeq}
&&
\Map\bigl( \TT^2 , \GL_2(\RR) \bigr)
\ar[d]^-{{\rm Cor}~\ref{t25}}_-{\simeq}
\ar[u]_-{{\rm Cor}~\ref{t25}}^-{\cong}
\\
\bigl( \ZZ^2 \times \GL_2(\RR) \bigr)
\times
\bigl(
\TT^2 \rtimes \EZ
\bigr)
\ar[r]^-{ \id \times \pr }
&
\bigl( \ZZ^2\times  \GL_2(\RR) \bigr)
\times
\EZ 
\ar[rr]^-{(\vec{v},B;A)\mapsto (A^T \vec{v},BA)}
&&
\ZZ^2 \times \GL_2(\RR) 
.
}
\end{equation*}

\end{observation}

We record the following basic application of group theory.
\begin{observation}
\label{q11}
For $\vec{v} = \begin{bmatrix} p \\ q \end{bmatrix} \in \ZZ^2$,
consider the subset $T_{\vec{v}} := \left\{ P \mid P \vec{v}= {\sf gcd}(p,q)  \vec{e}_1 \right\}\subset \GL_2(\ZZ)$.
\begin{enumerate}

\item
In the case that $p\geq 0$ and $q=0$, the set $T_{\vec{v}}$ is identical with the stabilizer subgroup:
\[
T_{\vec{v}}
=
{\sf Stab}_{\GL_2}(\ZZ)( {\sf gcd}(p,q) \cdot \vec{e}_1 )
=
\begin{cases}
\GL_2(\ZZ)
&
~,
\text{ if } p=0
\\
\left\{  \begin{bmatrix} 1 & b \\ 0 & d \end{bmatrix} \right\}
=
\left\lag  
\begin{bmatrix}
1
&
0
\\
0
&
-1
\end{bmatrix} 
,
\begin{bmatrix}
1
&
1
\\
0
&
1
\end{bmatrix} 
\right\rag
\cong
\sO(1) \ltimes \ZZ
&
~,
\text{ if } p>0
\end{cases}
~,
\]
in which the semi-direct product is with respect to the standard action $\sO(1) \xra{\cong} \Aut(\ZZ)$.

\item
The set $T_{\vec{v}}$ is not empty.
Left multiplication defines a free transitive action of this stabilizer:
\[
\GL_2(\ZZ)
\lacts
T_{\vec{v}}
\qquad 
\text{for $\vec{v}= \vec{0}$}
~,
\qquad
\text{ and }
\qquad
\sO(1) \ltimes \ZZ
\lacts
T_{\vec{v}}
\qquad 
\text{for $\vec{v}\neq \vec{0}$}
~.
\]

\item
An element $P \in T_{\vec{v}}$ determines an isomorphism between groups:
\begin{eqnarray*}
{\sf Stab}_{\GL_2(\ZZ)}(\vec{v}) 
&
=
&
P^{-1}
{\sf Stab}_{\GL_2(\ZZ)}( {\sf gcd}(p,q) \cdot \vec{e}_1 )
P
\\
&
=
&
\begin{cases}
\GL_2(\ZZ)
&
~,
\text{ if } \vec{v}= \vec{0}
\\
\left\lag  
P^{-1}
\begin{bmatrix}
1
&
0
\\
0
&
-1
\end{bmatrix} 
P
,
P^{-1}
\begin{bmatrix}
1
&
1
\\
0
&
1
\end{bmatrix} 
P
\right\rag
\cong
\sO(1) \ltimes \ZZ
&
~,
\text{ if } \vec{v}\neq \vec{0}
\end{cases}
~.
\end{eqnarray*}

\item
An element $P=\begin{bmatrix} w & x \\ y & z \end{bmatrix} \in T_{\vec{v}} \cap \SL_2(\ZZ)$ determines an identification:
\[
{\sf Stab}_{\SL_2(\ZZ)}(\vec{v}) 
=
\begin{cases}
\SL_2(\ZZ)
&
~,
\text{ if } \vec{v} = \vec{0}
\\
\left\lag  
\begin{bmatrix}
1+yz
&
z^2
\\
-y^2
&
1-yz
\end{bmatrix} 
\right\rag
=
\lag P^{-1} U_1 P \rag
\cong
\ZZ
&
~,
\text{ if } \vec{v} \neq \vec{0}
\end{cases}
~.
\]

\end{enumerate}

\end{observation}

The next result is phrased in terms of spaces fitting into the diagram in which each of the two squares, and therefore their concatenated larger square, is a pullback:
\begin{equation}
\label{q10}
\xymatrix{
{(\CC\PP^\infty)^2}_{/\ZZ
}
\times
\sB\ZZ
\ar[rr]
\ar[d]
&&
{(\CC\PP^\infty)^2}_{/\Braid
}
\ar[rr]
\ar[d]
&&
{(\CC\PP^\infty)^2}_{/\GL_2(\ZZ)}
\ar[d]
\\
\sB \ZZ
\times
\sB \ZZ
\ar[rr]^-{\lag \tau_1 , (\tau_1\tau_2)^6 \rag}
\ar[d]_-{\pr}
&&
\sB \Braid
\ar[r]^-{\Phi}
&
\BSL_2(\ZZ)
\ar[r]
&
\BGL_2(\ZZ)
\\
\sB \ZZ
\ar[urrr]_-{\lag U_1\rag}
&&
&&
.
}
\end{equation}

\begin{prop}
\label{t26}
\begin{enumerate}
\item[]

\item
The standard framing $\varphi_0 \in \Fr(\TT^2)$ determines an identification between spaces:
\[
\bcM_1^{\fr} 
\xra{~\simeq~}
\Bigl(
{(\CC\PP^\infty)^2}_{/ \Braid
}
\Bigr)
~\coprod~
\Bigl(
{(\CC\PP^\infty)^2}_{/ \ZZ
}
\times 
\sB \ZZ
\Bigr)^{ \amalg \NN}
~,
\]
through which $\varphi_0$ selects the distinguished path-component. 

\item
Furthermore, the resulting map $\pi_0 \Fr(\TT^2) \to \pi_0 \bcM_1^{\fr} \xra{\cong} \{0\} \amalg \NN = \ZZ_{\geq 0}$ factors as a composition:
\[
\pi_0 \Fr(\TT^2) 
\longrightarrow
\ZZ^2
\xra{~{\sf gcd}~}
\ZZ_{\geq 0}
~,
\]
in which the second map takes the \bit{greatest common divisor}, and the first map is
\[
[\varphi]
\mapsto
\Bigl[
\TT \vee \TT = \sk_1(\TT^2)
\xra{ {\varphi \circ \varphi_0^{-1} }_{|\sk_1(\TT^2)} }
\GL_2(\RR) 
\Bigr]
\in \pi_1\bigl(\uno \in \GL_2(\RR)\bigr)^2 
~\cong~
\ZZ^2
~.
\]

\end{enumerate}

\end{prop}

\begin{proof}

The result follows upon explaining the following sequences of identifications in the $\infty$-category $\Spaces$:
\begin{eqnarray}
\nonumber
\bcM_1^{\fr} 
&
\underset{\rm Obs~\ref{t4}}
{~\simeq~}
&
\Bigl(
\ZZ^2
\times
\GL_2(\RR) 
\Bigr)_{/\TT^2 \rtimes \GL_2(\ZZ)}
\\
\label{q2}
&
\underset{\rm iterate~quotient}
{~\simeq~}
&
\Bigl(
\bigl(
\ZZ^2
\times
\GL_2(\RR) 
\bigr)_{/\TT^2}
\Bigr)_{/\GL_2(\ZZ)}
\\
\label{q3}
&
\underset{\rm trivial~\TT^2~action}
{~\simeq~}
&
\Bigl(
\ZZ^2
\times
\sB \TT^2
\times
\GL_2(\RR) 
\Bigr)_{/\GL_2(\ZZ)}
\\
\label{q4}
&
\underset{\rm groupoids~are~effective}
{~\simeq~}
&
{\ZZ^2} _{/\GL_2(\ZZ)}
\underset{ \BGL_2(\ZZ) }
\times
\bigl(
(\CC\PP^\infty)^2\times \GL_2(\RR) 
\bigr)_{/\GL_2(\ZZ)}
\\
\label{q5}
&
\underset{\rm explicit~quotient}
{~\simeq~}
&
\Bigl(
\BGL_2(\ZZ) \amalg \sB (\ZZ \rtimes \sO(1))^{\amalg \NN}
\Bigr)
\underset{ \BGL_2(\ZZ) }
\times
\bigl(
(\CC\PP^\infty)^2\times \GL_2(\RR) 
\bigr)_{/\GL_2(\ZZ)}
\\
\label{q6}
&
\underset{\rm distribute~\times~over~\amalg}
{~\simeq~}
&
\left(
\BGL_2(\ZZ)
\underset{ \BGL_2(\ZZ) }
\times
\bigl(
(\CC\PP^\infty)^2\times \GL_2(\RR) 
\bigr)_{/\GL_2(\ZZ)}
\right)
\\
\nonumber
&
&
\coprod
\left(
\sB (\ZZ \rtimes \sO(1) )
\underset{ \BGL_2(\ZZ) }
\times
\bigl(
(\CC\PP^\infty)^2\times \GL_2(\RR) 
\bigr)_{/\GL_2(\ZZ)}
\right)^{\amalg \NN}
\\
\label{q7}
&
\underset{\rm base-change}
{~\simeq~}
&
\left(
(\CC\PP^\infty)^2\times \GL_2(\RR) 
\bigr)_{/\GL_2(\ZZ)}
\right)
\coprod
\left(
(\CC\PP^\infty)^2\times \GL_2(\RR) 
\bigr)_{/\ZZ \rtimes \sO(1) }
\right)^{\amalg \NN}
\\
\label{q8}
&
\underset{\rm Lem~\ref{t66}}
{~\simeq~}
&
\left(
{
(\CC\PP^\infty)^2
}_{/\Omega \bigl(\GL_2(\RR)_{/\GL_2(\ZZ)} \bigr)}
\right)
\coprod
\left(
{
(\CC\PP^\infty)^2
}_{/\Omega \bigl(\GL_2(\RR)_{/\ZZ \rtimes \sO(1)} \bigr)}
\right)^{\amalg \NN}
\\
\label{q9}
&
\underset{\rm explicit~identifications}
{~\simeq~}
&
\left(
{
(\CC\PP^\infty)^2
}_{/\Braid}
\right)
\coprod
\left(
{
(\CC\PP^\infty)^2
}_{/\ZZ}
\times 
\sB \ZZ
\right)^{\amalg \NN}
~.
\end{eqnarray}
The first identification follows from Observation~\ref{t4}.
The bottom horizontal map in Observation~\ref{t4} reveals that the action $\ZZ^2 \times \GL_2(\RR) \racts  \TT^2 \rtimes \GL_2(\ZZ)$ can be identified as the diagonal action of the action
\begin{equation}
\label{e1}
\bigl(
\TT^2 \rtimes \GL_2(\ZZ) 
\bigr)^{\op}
\xra{~\pr~} 
\GL_2(\ZZ)^{\op}
\xra{~(-)^T~}
\GL_2(\ZZ)
\underset{\rm standard}\lacts 
\ZZ^2
\end{equation}
together with the action
\[
\bigl(
\TT^2 \rtimes \GL_2(\ZZ) 
\bigr)^{\op}
\xra{~\pr~} 
\GL_2(\ZZ)^{\op}
\xra{~\rm include~}
\GL_2(\RR)^{\op}
\underset{\rm right~mult}\lacts 
\GL_2(\RR)
~.
\]
The equivalence~(\ref{q2}) identifies the $\TT^2 \rtimes \GL_2(\ZZ)$-quotient as the $\TT^2$-quotient followed by the $\GL_2(\ZZ)$-quotient.
The equivalence~(\ref{q3})  is a consequence of the $\TT^2$-action being trivial on both factors.
The equivalence~(\ref{q4}) is an instance of the general base-change identity $(X\times Y)_{/G} \simeq (X_{/G}) \underset{\sB G}\times (Y_{/G})$.  
The equivalence~(\ref{q5}) is the orbit-stabilizer theorem, as we now explain.
By Observation~\ref{q11}, two elements $\begin{bmatrix} u \\ v \end{bmatrix} , \begin{bmatrix} s \\ t \end{bmatrix} \in \ZZ^2$ are in the same (\ref{e1})-orbit if and only if their greatest common divisors ${\sf gcd}(u,v) = {\sf gcd}(s,t)\in \ZZ_{\geq 0}$ agree.
In particular, there is a bijection between the set of (\ref{e1})-orbits and the subset
\[
\ZZ_{\geq 0}
~\cong~
\left\{
\begin{bmatrix} g \\ 0 \end{bmatrix}
\right\}
~\subset~
\ZZ^2
~.
\]
Furthermore, the stabilizer of $\begin{bmatrix} g \\ 0 \end{bmatrix} \in \ZZ^2$ with respect to the action $\GL_2(\ZZ)^{\op} \xra{(-)^T} \GL_2(\ZZ)\lacts \ZZ^2$ is
\[
{\sf Stab}_{\GL_2(\ZZ)^{\op}}\left( \begin{bmatrix} g \\ 0 \end{bmatrix} \right)
~=~
\begin{cases}
\GL_2(\ZZ)^{\op}
&
~,
\text{ if }g = 0 
\\
\left\{
\begin{bmatrix}
1 & 0
\\
c & d
\end{bmatrix}
\right\}^{\op}
\cong
\bigl(
\ZZ
\rtimes
\sO(1)
\bigr)^{\op}
&
~,
\text{ if }g \neq 0 
\end{cases}
~.
\]
Therefore, the quotient
\[
{
\ZZ^2
}_{/\GL_2(\ZZ)}
~\simeq~
\underset{g\in \ZZ_{\geq 0}}
\coprod
\sB {\sf Stab}_{\GL_2(\ZZ)^{\op}}\left( \begin{bmatrix} g \\ 0 \end{bmatrix} \right)
~\simeq~
\BGL_2(\ZZ)
\coprod
\sB
\left(
\ZZ
\rtimes
\sO(1)
\right)^{ \amalg \NN}
~.
\]
The equivalence~(\ref{q6}) is the distribution of $\times$ over $\coprod$.  
The equivalence~(\ref{q7}) is an instance of the general base-change identity $X_{/H}\simeq \sB H \underset{\sB G} \times X_{/G}$.
The equivalence~(\ref{q8}) is an instance of Lemma~\ref{t66}.
The equivalence~(\ref{q9}) is a direct application of Proposition~\ref{t32} for the $0$-cofactor, 
and for each other cofactor it is an application of Proposition~\ref{t32} then a consequence of the diagram~(\ref{q10}) of pullbacks among spaces.

\end{proof}

For $\varphi\in \Fr(\TT^2)$ a framing of the torus, consider the orbit map of $\varphi$ for this continuous action of Lemma~\ref{t50}:
\[
{\sf Orbit}_\varphi
\colon
\Imm(\TT^2)
\xra{~( ~ {\sf constant}_{\varphi}~ , ~\id ~ )~}
\Fr(\TT^2)
\times 
\Imm(\TT^2)
\xra{~{\sf Act}~}
\Fr(\TT^2)
~,\qquad
f
\mapsto {\sf Act}(\varphi , f)
~.
\]

Recall Notation~\ref{d7}.
\begin{observation}
\label{t33}
After Observation~\ref{t4}, for each framing $\varphi \in \Fr(\TT^2)$, 
the orbit map for $\varphi$ fits into a solid diagram among topological spaces:
\begin{equation*}
\xymatrix{
\Diff(\TT^2) 
\ar[rr]
\ar@{-->}[dr]^-{\sH_1}
&&
\Imm(\TT^2)
\ar[rrrr]^-{{\sf Orbit}_\varphi}
\ar@{-->}[dr]^-{\sH_1}
&&
&&
\Fr(\TT^2)
\ar[d]^-{{\rm Cor}~\ref{t25}}_-{\simeq}
\\
&
\GL_2(\ZZ) 
\ar[rr] 
&&
\EZ
\ar[rrr]^-{ A \mapsto (A^T \vec{\varphi} , B_\varphi A ) }
&&
&
\ZZ^2 \times \GL_2(\RR)
\\
\TT^2 \rtimes \GL_2(\ZZ)
\ar[uu]^-{\Aff}
\ar[rr]
\ar[ur]_-{\pr}
&&
\TT^2 \rtimes \EZ
\ar[uu]^(.35){\Aff} | \hole
\ar[ur]_-{\pr}
&&
.
}
\end{equation*}
The existence of the fillers follows from Observation~\ref{t99}.

\end{observation}

\begin{remark}
\label{r5}
The point-set fiber of ${\sf Orbit}_{\varphi}$ over $\varphi$, which is the point-set stabilizer of the action $\Fr(\TT^2) \racts \Imm(\TT^2)$ of Lemma~\ref{t50}, 
consists of those local-diffeomorphisims $f$ for which the diagram among vector bundles, 
\[
\xymatrix{
\tau_{\TT^2}
\ar[rr]^-{\varphi}
\ar[d]_-{\sD f}
&&
\epsilon^2_{\TT^2}
\\
f^\ast 
\tau_{\TT^2}
\ar[rr]^-{f^\ast \varphi}
&&
f^\ast
\epsilon^2_{\TT^2}
\ar[u]_-=
,
}
\]
commutes.
For a generic framing $\varphi$, a local-diffeomorphism $f$ satisfies this rigid condition if and only if $f = \id_{\TT^2}$ is the identity diffeomorphism.
In the special case of the standard framing $\varphi_0$, a local-diffeomorphism $f$ satisfies this rigid condition if and only if $f = {\sf trans}_{f(0)} \circ {\sf quot}$ is translation in the group $\TT^2$ after a group-theoretic quotient $\TT^2 \xra{\rm quotient} \TT^2$. 
In particular, the point-set fiber of $\bigl( {\sf Orbit}_{\varphi_0} \bigr)_{|\Diff(\TT^2)}$ over $\varphi_0$ is $\TT^2$, and the homomorphism $\TT^2 \hookrightarrow \Diff(\TT^2)$ witnesses the inclusion of those diffeomorphisms that \emph{strictly} fix $\varphi_0$.

On the other hand, the \emph{homotopy-}fiber of~${\sf Orbit}_{\varphi_0}$ over $\varphi_0$ is more flexible: 
it consists of pairs $(f, \gamma)$ in which $f$ is a local-diffeomorphism and $\gamma$ is a homotopy 
\[
\varphi_0
~ \underset{\gamma}\sim ~
\Act(\varphi_0,f)
~.
\]
As we will see, every orientation-preserving local-diffeomorphism $f$ admits a lift to this homotopy-fiber.  
In particular, small perturbations of such $f$, such as multiplication by bump functions in neighborhoods of $\TT^2$, can be lifted to this homotopy-fiber.
\end{remark}

\begin{definition}\label{d1}
Let $\varphi\in \Fr(\TT^2)$ be a framing of the torus.  
The space of \bit{framed local-diffeomorphisms}, and the space of \bit{framed diffeomorphisms}, of the framed smooth manifold $(\TT^2,\varphi)$ are respectively the pullbacks in the $\infty$-category $\Spaces$:
\begin{equation*}
\xymatrix{
\Imm^{\sf fr}(\TT^2,\varphi)
\ar[rr]
\ar[d]
&&
\Imm(\TT^2)
\ar[d]^-{{\sf Orbit}_\varphi}
&&
\Diff^{\sf fr}(\TT^2,\varphi)
\ar[rr]
\ar[d]
&&
\Diff(\TT^2)
\ar[d]^-{{\sf Orbit}_\varphi}
\\
\ast
\ar[rr]^-{\lag \varphi \rag}
&&
\Fr(\TT^2)
&
\text{ and }
&
\ast
\ar[rr]^-{\lag \varphi \rag}
&&
\Fr(\TT^2)
~.
}
\end{equation*}
In the case that the framing $\varphi = \varphi_0$ is the standard framing, we simply denote
\[
\Imm^{\sf fr}(\TT^2)
~:=~
\Imm^{\sf fr}(\TT^2,\varphi_0)
\qquad
\text{ and }
\qquad
\Diff^{\sf fr}(\TT^2)
~:=~
\Diff^{\sf fr}(\TT^2,\varphi_0)
~.
\]
\end{definition}

The following result follows directly from Lemma~\ref{t2} of Appendix \ref{sec.A}, and Proposition~\ref{t26}(1).
\begin{cor}\label{t3}
Let $\varphi \in \Fr(\TT^2)$ be a framing.
The space $\Diff^{\fr}(\TT^2 , \varphi)$ is canonically endowed with the structure of a continuous group over $\Diff(\TT^2)$.
With respect to this structure, 
there is a canonical identification between continuous groups:
\[
\Diff^{\fr}(\TT^2, \varphi)
~\simeq~
\Omega_{[\varphi]}  \bcM_1^{\fr} 
\underset{\rm Prop~\ref{t26}(1)}{~\simeq~}
\begin{cases}
\Omega
\Bigl(
{(\CC\PP^\infty)^2}_{/ \Braid
}
\Bigr)
~\simeq~
\TT^2 \rtimes \Braid
&
,~
\text{ if } \vec{\varphi} = \vec{0}
\\
\Omega
\Bigl(
{(\CC\PP^\infty)^2}_{/ \ZZ
}
\times 
\sB \ZZ
\Bigr)
~\simeq~
(\TT^2 \rtimes \ZZ)\times \ZZ
&
,~
\text{ if } \vec{\varphi} \neq \vec{0}
\end{cases}
~.
\]

\end{cor}

\begin{observation}
\label{t43}
The kernel of $\Phi$ acts by rotating the framing, which is to say 
there is a canonically commutative diagram among continuous groups:
\[
\xymatrix{
\ZZ
\ar[d]^-{\cong}_-{\bigl\lag (\tau_1\tau_2)^6 \bigr\rag}
\ar[rr]^-{\simeq}
&&
\Omega_{\uno} \GL_2(\RR)
\ar[rr]^-{\Omega \bigl( A\mapsto A \cdot \varphi_0 \bigr)}
&&
\Omega_{\varphi_0} \Fr(\TT^2)
\ar[d]
\\
\Ker(\Phi)
\ar[rr]
&&
\Braid
\ar[rr]^-{\Aff^{\fr}}
&&
\Diff^{\fr}(\TT^2)
.
}
\]
Here $\Aff^{\fr}$ is defined in Lemma~\ref{t34}. Indeed, there is a canonically commutative diagram among spaces, in which each row is an $\Omega$-Puppe sequence:
\[
\xymatrix{
\Ker(\Phi)
\ar[rr]
\ar[d]
&&
\Braid
\ar[rr]^-{\Phi}
\ar[d]^-{\Aff^{\fr}}
&&
\GL_2(\ZZ)
\ar[rr]^-{\RR\underset{\ZZ}\ot}
\ar[d]^-{\Aff}
&&
\GL_2(\RR)
\ar[d]^-{\rm Rotate~the~framing~\varphi_0}
\\
\Omega_{\varphi_0} \Fr(\TT^2)
\ar[rr]
&&
\Diff^{\fr}(\TT^2)
\ar[rr]
&&
\Diff(\TT^2)
\ar[rr]^-{{\sf Orbit}_{\varphi_0}}
&&
\Fr(\TT^2)
.
}
\]

\end{observation}

\subsection{Proof of Theorem~\ref{Theorem A} and Corollary~\ref{t40}} \label{sec.proofs}
Theorem~\ref{Theorem A} consists of three statements. 
Theorem~\ref{Theorem A}(1) is implied by Proposition~\ref{t26}.
Theorem~\ref{Theorem A}(2a) is implied by Corollary~\ref{t3}.
Theorem~\ref{Theorem A}(2b) (as well as Theorem~\ref{Theorem A}(2a)) is implied by Lemma~\ref{t34} below.

\begin{notation}
\label{d6}
Let $\vec{v} = \begin{bmatrix} p \\ q \end{bmatrix} \in \ZZ^2$ and $r\in \ZZ$.
Denote the matrices
\[
U_{\vec{v}} ~:=~
\begin{bmatrix}
1+yz
&
z^2
\\
-u^2
&
1-yz
\end{bmatrix}^T
\qquad
\text{ and }
\qquad
D_{\vec{v},r} ~:=~
\begin{bmatrix}
1 + (r-1) xy
& 
-(r-1)xz
\\
(r-1)wy
&
1+(r-1)wz
\end{bmatrix}^T
~,
\]
for some $w,z,y,z \in \ZZ$ that solve
\begin{eqnarray}
\label{q12}
wp+xq
&
=
&
{\sf gcd}(p,q) \geq 0
\\
\nonumber
yp+zq
&
=
&
0
\\
\nonumber
wz-xy
&
=
&
1
~.
\end{eqnarray}
Denote the semi-direct continuous group, and continuous monoid, 
\[
\TT^2 
\underset{U_{\vec{v}}}
\rtimes 
\ZZ
\qquad
\text{ and }
\qquad
\TT^2 
\underset{D_{\vec{v}}, U_{\vec{v}}}
\rtimes 
(\NN^\times \ltimes \ZZ)
\]
given through the actions on the continuous group $\TT^2$:
\[
\ZZ 
\xra{b \mapsto {U_{\vec{v}}^b} }
\SL_2(\ZZ)
\lacts
\TT^2
\qquad
\text{ and }
\qquad
\ZZ
\rtimes
\NN^\times 
\xra{(b,d) \mapsto {{U_{\vec{v}}^b} D_{\vec{v},d}} }
\EZ
\lacts 
\TT^2
~.
\]

\end{notation}

\begin{remark}
\label{q13}
Observation~\ref{q11} ensures the existence of a solution to~(\ref{q12}).
Observation~\ref{q11} also implies, for $U'_{\vec{v}}$ and $D'_{\vec{v},r}$ defined by another choice of solution to~(\ref{q12}), then $U'_{\vec{v}}$ and $D'_{\vec{v},r}$ are respectively canonically conjugate with $U_{\vec{v}}$ and $D_{\vec{v},r}$, and therefore the continuous groups and continuous monoids are respectively canonically identified:
\[
\TT^2 
\underset{U_{\vec{v}}}
\rtimes 
\ZZ
~\simeq~
\TT^2 
\underset{U'_{\vec{v}}}
\rtimes 
\ZZ
\qquad
\text{ and }
\qquad
\TT^2 
\underset{U_{\vec{v}} , D_{\vec{v}}}
\rtimes 
(
\ZZ
\rtimes
\NN^\times 
)
~\simeq~
\TT^2 
\underset{U'_{\vec{v}} , D'_{\vec{v}} }
\rtimes 
(
\ZZ
\rtimes
\NN^\times 
)
~.
\]
\end{remark}

The next result extends Corollary~\ref{t3} from an assertion about $\Diff^{\fr}(\TT^2,\varphi)$ to one about $\Imm^{\fr}(\TT^2,\varphi)$.
Recall Notation~\ref{d7}.
\begin{lemma}
\label{t34}
Let $\varphi\in \Fr(\TT^2)$ be a framing of the torus.
\begin{enumerate}

\item
If $\vec{\varphi} = \vec{0}$, then there are canonical equivalences in the diagrams among continuous monoids:
\begin{equation}
\label{e54}
\xymatrix{
\TT^2 \rtimes \Ebraid
\ar@{-->}[rr]^-{\simeq}_{\Aff^{\fr}}
\ar[d]_-{\id \rtimes \Psi}
&&
\Imm^{\fr}(\TT^2,\varphi) 
\ar[d]^-{\rm forget}
&&
\TT^2 \rtimes \Braid 
\ar@{-->}[rr]^-{\simeq}_{\Aff^{\fr}}
\ar[d]_-{\id \rtimes \Phi}
&&
\Diff^{\fr}(\TT^2,\varphi) 
\ar[d]^-{\rm forget}
\\
\TT^2 \rtimes \EZ
\ar[rr]^-{\simeq}_-\Aff
&&
\Imm(\TT^2)
&
\text{ and }
&
\TT^2 \rtimes  \GL_2(\ZZ)
\ar[rr]^-{\simeq}_-\Aff
&&
\Diff(\TT^2)
.
}
\end{equation}

\item
If $\vec{\varphi} \neq \vec{0}$, then there are canonical equivalences in the diagrams among continuous monoids:
\[
\xymatrixrowsep{1cm}
\xymatrixcolsep{.45cm}
\xymatrix{
\left( \TT^2 \underset{U_{\vec{\varphi}}, D_{\vec{\varphi}} } \rtimes (\ZZ \rtimes \NN^\times ) \right) \times \ZZ
\ar@{-->}[rr]^-{\simeq}_-{\Aff^{\fr}}
\ar[d]_-{\id \rtimes \bigl( (b,d,k)\mapsto  U_{\vec{\varphi}}^b D_{\vec{\varphi},d} \bigr)}
&&
\Imm^{\fr}(\TT^2,\varphi) 
\ar[d]^-{\rm forget}
&&
\bigl(\TT^2 \underset{U_{\vec{\varphi}}} \rtimes \ZZ\bigr)
\times \ZZ
\ar@{-->}[rr]^-{\simeq}_-{\Aff^{\fr}}
\ar[d]^-{\id \rtimes \bigl( (b,k)\mapsto U_{\vec{\varphi}}^b \bigr)}
&&
\Diff^{\fr}(\TT^2,\varphi) 
\ar[d]^-{\rm forget}
\\
\TT^2 \rtimes \EZ
\ar[rr]^-{\simeq}_-\Aff
&&
\Imm(\TT^2)
&
\text{ and }
&
\TT^2 \rtimes  \GL_2(\ZZ)
\ar[rr]^-{\simeq}_-\Aff
&&
\Diff(\TT^2)
.
}
\]

\end{enumerate}

\end{lemma}

\begin{proof}
Using Observation~\ref{t21}, the canonical equivalences in the commutative diagrams on the right follow from those on the left.

Consider the diagrams in the $\infty$-category $\Spaces$, which make use of Notation~\ref{d7}.
\begin{enumerate}
\item
For $\vec{\varphi} = \vec{0}$:
\[
\xymatrix{
\TT^2 \rtimes\Ebraid
\ar[d]_-{\id \rtimes \Psi}
\ar[rr]^-{\pr}
&&
\Ebraid
\ar[d]_-{\Psi}
\ar[rrrr]^-{!}
&&
&&
\ast
\ar[d]_-{\left\lag \left(\vec{\varphi} , B_\varphi \right) \right\rag}
\\
\TT^2 \rtimes \EZ 
\ar[d]_-{\Aff}^{\simeq}
\ar[rr]_-{\pr}
&&
\EZ
\ar[rrrr]_-{ A \mapsto (A \vec{\varphi} , B_\varphi A) }
&&
&&
\ZZ^2 \times \GL_2(\RR)
\ar[d]^-{\rm Cor~\ref{t25}}_-{\simeq}
\\
\Imm(\TT^2)
\ar[rrrrrr]^-{{\sf Orbit}_\varphi}
&&
&&
&&
\Fr(\TT^2)
.
}
\]

\item
For $\vec{\varphi} \neq \vec{0}$:
\[
\xymatrix{
\left( \TT^2 \underset{U_{\vec{\varphi}} , D_{\vec{\varphi}} } \rtimes (\ZZ \rtimes \NN^\times ) \right) \times \ZZ
\ar[d]^-{\pr}
\ar[rr]^-{\pr}
&&
( \ZZ \rtimes \NN^\times ) \times \ZZ
\ar[d]^-{\pr}
\ar[rrrr]^-{!}
&&
&&
\ast
\ar[d]_-{\lag  ( \ast , B_\varphi ) \rag}
\\
\TT^2 \underset{U_{\vec{\varphi}} , D_{\vec{\varphi}} } \rtimes \left ( \ZZ \rtimes \NN^\times  \right)
\ar[d]^-{\id \rtimes \bigl((b,d)\mapsto  U_{\vec{\varphi}}^b D_{\vec{\varphi},d} \bigr)}
\ar[rr]^-{\pr}
&&
\ZZ
\rtimes
\NN^\times 
\ar[d]^-{(b,d)\mapsto  U_{\vec{\varphi}}^b D_{\vec{\varphi},d}}
\ar[rrrr]^-{(b,d)\mapsto B_\varphi  U_{\vec{\varphi}}^b D_{\vec{\varphi},d} }
&&
&&
\ast
\times
\GL_2(\RR)_{B_\varphi}
\ar[d]_-{\lag \vec{\varphi} \rag \times {\sf inc} }
\\
\TT^2 \rtimes \EZ 
\ar[d]_-{\Aff}^{\simeq}
\ar[rr]_-{\pr}
&&
\EZ
\ar[rrrr]_-{ A \mapsto (A^T \vec{\varphi} , B_\varphi A) }
&&
&&
\ZZ^2 \times \GL_2(\RR)
\ar[d]^-{\rm Cor~\ref{t25}}_-{\simeq}
\\
\Imm(\TT^2)
\ar[rrrrrr]^-{{\sf Orbit}_\varphi}
&&
&&
&&
\Fr(\TT^2)
,
}
\]
where $\GL_2(\RR)_{B_\varphi} \subset \GL_2(\RR)$ is the path-component containing $B_\varphi \in \GL_2(\RR)$.

\end{enumerate}
Observation~\ref{t33} implies that each bottom rectangle canonically commutes.
Lemma~\ref{t1} and Corollary~\ref{t25} together imply each of these bottom rectangles witnesses a pullback.
Each of the top left squares, as well as the middle left square in the lower diagram, is clearly a pullback.
Corollary~\ref{t31} states that the top right square in the upper diagram is a pullback.
Provided the top right and middle right squares in the lower diagram are pullbacks,
we would then conclude that each of the outer squares witnesses a pullback.  
The result would then follows by Definition~\ref{d1} of $\Imm^{\fr}(\TT^2,\varphi)$.

So it remains to show that the top right and middle right squares in the lower diagram are pullbacks.
The paths of matrices,
\[
[0,1]
~
\ni 
~
t
\longmapsto
\begin{bmatrix}
1+tcd
&
tz^2
\\
-ty^2
&
1-tyz
\end{bmatrix}^T
~,~
\begin{bmatrix}
1 + t(r-1) xy
& 
-t(r-1)xz
\\
t(r-1)wy
&
1+t(r-1)wz
\end{bmatrix}^T
~\in~
\GL_2(\RR)
~,
\]
determine an identification of the named map $\ZZ \rtimes \NN^\times \to \GL_2(\RR)$ with the constant map at $B_\varphi$.  
Together with the standard identification $\ZZ\simeq \Omega_{B_\varphi} \GL_2(\RR)$, this shows that the top right square in the lower diagram as a pullback.  
The middle right square of the lower diagram is a pullback because the map
\[
\ZZ
\rtimes
(\ZZ\setminus \{0\}) 
\longrightarrow
{\sf Stab}_{\EZ^{\op}}
\left(
\vec{\varphi}
\right)
~,\qquad
(b,d)
\longmapsto
\left(
\begin{bmatrix}
w & x
\\
y & z
\end{bmatrix}^{-1} 
\begin{bmatrix}
1 & b
\\
0 & d
\end{bmatrix}
\begin{bmatrix}
w & x
\\
y & z
\end{bmatrix}
\right)^T
~=~
U_{\vec{\varphi}}^b
D_{\vec{\varphi},d} 
~,
\]
is an isomorphism between monoids, where $w,x,y,z\in \ZZ$ are as in Notation~\ref{d6}.

\end{proof}

By applying the product-preserving functor $\Spaces \xra{\pi_0} {\sf Sets}$, Lemma~\ref{t34} implies the following.
\begin{cor}\label{r2}
There is a canonical isomorphism in the diagram of groups:
\[
\xymatrix{
\Braid
\ar@{-->}[rr]^-{\cong}
\ar[d]_-{\Phi}
&&
{\sf MCG}^{\sf fr}(\TT^2)
\ar[d]^-{\rm forget}
\\
\GL_2(\ZZ)
\ar[rr]^-{\cong}
&&
{\sf MCG}(\TT^2)
.
}
\]

\end{cor}

\begin{remark}\label{r1}
Proposition~\ref{t32} and Corollary~\ref{r2} grant a central extension among groups:
\[
1
\longrightarrow
\ZZ
\longrightarrow
{\sf MCG}^{\fr}(\TT^2)
\longrightarrow
{\sf MCG}^{\sf or}(\TT^2)
\longrightarrow
1
~.
\\
\]

\end{remark}

\begin{proof}[Proof of Corollary~\ref{t40}]
By construction, the diagram among spaces,
\[
\xymatrix{
\TT^2 \rtimes \EZ
\ar[rr]^-{\simeq}_-{\rm Cor~\ref{t22}}
\ar[dr]_-{\pr}
&&
\Imm(\TT^2)
\ar[dl]^-{\ev_0}
\\
&
\TT^2
&
,
}
\]
canonically commutes, in which the left vertical map is projection, and the right vertical map evaluates at the origin $0\in \TT^2$.  
Therefore, upon taking fibers over $0\in \TT^2$, the (left) commutative diagram~(\ref{e54}) among continuous monoids determines the commutative diagram among commutative monoids:
\begin{equation*}
\xymatrix{
\Ebraid
\ar[rrrr]^-{\simeq}
\ar[d]
&&
&&
\Imm^{\fr}(\TT^2 ~{\sf rel } ~0)
\ar[d]
\\
\EZ
\ar[rrrr]^-{\simeq}_-{\rm Cor~\ref{t22}}
\ar[drr]_-{\RR \underset{\ZZ}\ot}
&&
&&
\Imm(\TT^2 ~{\sf rel } ~0)
\ar[dll]^-{\sD_0}
\\
&&
\GL_2(\RR)
&&
,
}
\end{equation*}
in which the map $\RR \underset{\ZZ}\ot $ is the standard inclusion, and $\sD_0$ takes the derivative at the origin $0\in \TT^2$.
To finish, Corollary~\ref{t31} supplies the left pullback square in the following diagram among continuous groups, while the right pullback square is definitional:
\[
\xymatrix{
\Braid
\ar[rr]
\ar[d]
&&
\ast
\ar[d]
&&
\Diff(\TT^2 \smallsetminus \BB^2 ~{\sf rel}~\partial )
\ar[ll]
\ar[d]
\\
\GL_2(\ZZ)
\ar[rr]^-{\RR \underset{\ZZ}\ot}
&&
\GL_2(\RR)
&&
\Diff(\TT^2~{\sf rel}~0)
\ar[ll]_-{\sD_0}
.
}
\]
The result follows. 

\end{proof}

\subsection{Comparison with sheering}
\label{sec.sheering}
We use Theorem~\ref{Theorem A}(2) to show that the $\Diff^{\fr}(\TT^2)$ is generated by sheering.  
We quickly tour through some notions and results, which are routine after the above material.  

\begin{notation}
It will be convenient to define the projection $\TT^{2} \xra{\pr_{i}} \TT$ to be projection \emph{off} of the $i^{\rm th}$ coordinate. 
So for $\TT^{2} \ni p = (x_{p}, y_{p}),$ we have $\pr_{1}(p) = y_{p}$ and $\pr_{2}(p) = x_{p}.$
\end{notation}

Let $i\in \{1,2\}$.
Consider the topological subgroup and topological submonoid,
\[
\Diff(\TT^2 \xra{\pr_i} \TT)
~\subset~
\Diff(\TT^2)
\qquad
\text{ and }
\qquad
\Imm(\TT^2 \xra{\pr_i} \TT)
~\subset~
\Imm(\TT^2)
~,
\]
consisting of those (local-)diffeomorphisms $\TT^2\xra{f} \TT^2$ that lie over some (local-)diffeomorphism $\TT\xra{\ov{f}} \TT$:
\begin{equation}
\label{e87}
\xymatrix{
\TT^2
\ar[rr]^-f
\ar[d]_-{\pr_i}
&&
\TT^2 \ar[d]^-{\pr_i}
\\
\TT
\ar[rr]^-{\ov{f}}
&&
\TT
~.
}
\end{equation}
The topological space of \bit{framings} of $\TT^2 \xra{\pr_i} \TT$ is the subspace
\[
\Fr(\TT^2 \xra{\pr_i} \TT)
~\subset~
\Fr(\TT^2)
\]
consisting of those framings $\tau_{\TT^2}\xra{\varphi} \epsilon^2_{\TT^2}$ that lie over a framing $\tau_{\TT} \xra{\ov{\varphi}} \epsilon^1_{\TT}$:
\begin{equation}
\label{e70}
\xymatrix{
\tau_{\TT^2}
\ar[rr]^-{\varphi}_-{\cong}
\ar[d]_-{\sD \pr_i}
&&
\epsilon^2_{\TT^2}
\ar[d]^-{\pr_i \times \pr_i}
\\
\tau_{\TT}
\ar[rr]^-{\ov{\varphi}}_-{\cong}
&&
\epsilon^1_{\TT}
.
}
\end{equation}
Because $\pr_i$ is surjective, for a given $\varphi$, there is a unique $\ov{\varphi}$ as in~(\ref{e70}) if any.  
Better, $\varphi\mapsto \ov{\varphi}$ defines a continuous map:
\begin{equation}
\label{e82}
\Fr(\TT^2 \xra{\pr_i} \TT)
\longrightarrow
\Fr(\TT)
~,\qquad
\varphi
\mapsto 
\ov{\varphi}
~.
\end{equation}
Notice that the continuous right-action $\Act$ of Lemma~\ref{t50} evidently restricts as a continuous right-action:
\[
\Fr(\TT^2 \xra{\pr_i} \TT)
~\racts~
\Imm(\TT^2 \xra{\pr_i} \TT)
~.
\]
Furthermore, the map~(\ref{e82}) is evidently equivariant with respect to the morphism between topological monoids $\Imm(\TT^2 \xra{\pr_i} \TT) \xra{\rm forget} \Imm(\TT)$:
\[
\Bigl(
~
\Fr(\TT^2 \xra{\pr_i} \TT)
~\racts~
\Imm(\TT^2 \xra{\pr_i} \TT)
~\Bigr)
~
\xra{~\rm forget~}
~
\Bigl(
~
\Fr(\TT)
~\racts~
\Imm(\TT)
~\Bigr)
~,\qquad
\varphi
\mapsto 
\ov{\varphi}
~.
\]

Now let $\varphi \in \Fr(\TT^2 \xra{\pr_i} \TT)$ be a framing of the projection.
The orbit of $\varphi$ by this action is the map
\[
{\sf Orbit}_\varphi
\colon 
\Imm(\TT^2 \xra{\pr_i}\TT)
\longrightarrow
\Fr(\TT^2 \xra{\pr_i} \TT)
~,\qquad
f\mapsto \Act(\varphi,f)
~.
\]
The space of \bit{framed local-diffeomorphisms}, and the space of \bit{framed diffeomorphisms}, of $(\TT^2 \xra{\pr_i} \TT , \varphi)$ are respectively the homtopy-pullbacks among spaces:
\[
\xymatrix{
\Imm^{\sf fr}(\TT^2\xra{\pr_i} \TT,\varphi)
\ar[r]
\ar[d]
&
\Imm(\TT^2\xra{\pr_i} \TT)
\ar[d]^-{{\sf Orbit}_\varphi}
&&
\Diff^{\sf fr}(\TT^2 \xra{\pr_i} \TT,\varphi)
\ar[r]
\ar[d]
&
\Diff(\TT^2\xra{\pr_i} \TT)
\ar[d]^-{{\sf Orbit}_\varphi}
\\
\ast
\ar[r]^-{\lag \varphi \rag}
&
\Fr(\TT^2\xra{\pr_i} \TT)
&
\text{ and }
&
\ast
\ar[r]^-{\lag \varphi \rag}
&
\Fr(\TT^2\xra{\pr_i} \TT)
~.
}
\]

As in Observation~\ref{t20},
the topological space $\Fr(\TT^2 \xra{\pr_i} \TT)$ is a torsor for the topological group $\Map\bigl( \TT^2 , \GL_{\{i\}\subset 2}(\RR) \bigr)$ of smooth maps from $\TT^2$ to the subgroup 
\[
\GL_{\{i\}\subset 2}(\RR)
~:=~
\Bigl\{
A \mid
A\vec{e}_i \in {\sf Span}\{\vec{e}_i\}
\Bigr\}
~\subset~
\GL_2(\RR)
\]
consisting of those $2 \times 2$ matrices that carry the $i^{\rm th}$-coordinate line to itself. 
For each $i=1,2$, denote the intersections in $\GL_2(\RR)$:
\[
\xymatrix{
\SL_2(\ZZ)
\ar[rr]
\ar[d]
&&
\GL_2(\ZZ)
\ar[d]
&
&
\SL_{\{i\}\subset 2}(\ZZ)
\ar[rr]
\ar[d]
&&
\GL_{\{i\}\subset 2}(\ZZ)
\ar[d]
\\
\EpZ
\ar[rr]
&&
\EZ
&
\overset{- \cap \GL_{\{i\}\subset 2}(\RR)}\longmapsto
&
\sE^+_{\{i\}\subset 2}(\ZZ)
\ar[rr]
&&
\sE_{\{i\}\subset 2}(\ZZ)
}
\]

\begin{lemma}
\label{t62} 
For each $i=1,2$, the homotopy-equivalences between continuous monoids of Lemma~\ref{t1} and Corollary~\ref{t22} restrict as homotopy-equivalences between continuous monoids:
\[
\xymatrix{
\TT^2 \rtimes \GL_{\{i\}\subset 2}(\ZZ)
\ar[d]_-{\rm inclusion}
\ar@{-->}[r]^-{\Aff_i}_-{\simeq}
&
\Diff(\TT^2 \xra{\pr_i} \TT)
\ar[d]^-{\rm inclusion}
&&
\TT^2 \rtimes \sE_{\{i\}\subset 2}(\ZZ) 
\ar[d]_-{\rm inclusion}
\ar@{-->}[r]^-{\Aff_i}_-{\simeq}
&
\Imm(\TT^2 \xra{\pr_i} \TT)
\ar[d]^-{\rm inclusion}
\\
\TT^2 \rtimes \GL(\ZZ)
\ar[r]^-{\Aff}_-{\simeq}
&
\Diff(\TT^2)
&
\text{ and }
&
\TT^2 \rtimes \EZ 
\ar[r]^-{\Aff}_-{\simeq}
&
\Imm(\TT^2)
.
}
\]

\end{lemma}

\begin{proof}
Via the involution $\Sigma_2 \lacts \TT^2$ that swaps coordinates, the case in which $i=1$ implies the case in which $i=2$.
So we only consider the case in which $i=1$.

The left homotopy-equivalence is obtained from the right homotopy-equivalence by restricting to 
maximal continuous subgroups.
So we are reduced to establishing the right homotopy-equivalence.
Direct inspection reveals the indicated factorization $\Aff_1$ of the restriction of $\Aff$ to $\TT^2 \rtimes \sE_{\{1\}\subset 2}(\ZZ) \subset  \TT^2 \rtimes \EZ $.
So we are left to show that $\Aff_1$ is a homotopy-equivalence.  

Now, projection onto the $(1,1)$-entry defines a morphism between monoids, with kernel $\sK := \Bigl\{ \begin{bmatrix} 1 & b \\ 0 & d \end{bmatrix}\in \sE_{\{1\}\subset 2}(\ZZ) \Bigr\}$, which fits into a split short exact sequence of monoids:
\[
\xymatrix{
1
\ar[r] 
&
\sK
\ar[rr]
&&
\sE_{\{1\}\subset 2}(\ZZ) 
\ar[rr]_-{ (1,1)-{\rm entry}}
&&
(\ZZ \smallsetminus \{0\})^\times
\ar[r]
\ar@{-->}@(u,-)[ll]_-{{\Tiny \begin{bmatrix} a & 0 \\ 0 & 1 \end{bmatrix}}  \mapsfrom ~a}
&
1
~.
}
\]

Now, because $\pr_1$ is surjective, for a given $f\in \Imm(\TT^2 \xra{\pr_i} \TT)$, there is a unique $\ov{f}\in \Imm(\TT)$ as in~(\ref{e87}).
Better, $\Imm(\TT^2 \xra{\pr_i}\TT) \ni f\mapsto \ov{f} \in \Imm(\TT)$ defines a forgetful morphism between topological monoids, whose kernel can be identified as the topological monoid of smooth maps from $\TT$ to $\Imm(\TT)$ with value-wise monoid-structure.
This is to say
there is a bottom short exact sequence of topological monoids, which splits as indicated:
\begin{equation}
\label{e68}
\xymatrix{
1
\ar[r] 
&
\TT \rtimes \sK
\ar@{..>}[d]
\ar[rr]^-{ (\id , \lag 0 \rag) \rtimes {\rm inclusion}}
&&
\TT^2 \rtimes \sE_{\{1\}\subset 2}(\ZZ)
\ar[d]_-{\Aff_1}
\ar[rr]_-{\pr_1 \rtimes (1,1)-{\rm entry}}
&&
\TT \rtimes (\ZZ \smallsetminus \{0\})^\times
\ar@{..>}[d]
\ar[r]
\ar@{-->}@(u,-)[ll]_-{ \left((0, z)  , {\Tiny \begin{bmatrix} a & 0 \\ 0 & 1 \end{bmatrix}}\right) \mapsfrom (z, a) }
&
1
\\
1
\ar[r] 
&
\Map\bigl( \TT , \Imm(\TT) \bigr) 
\ar[rr]
&&
\Imm(\TT^2 \xra{\pr_1} \TT)
\ar[rr]_-{ f \mapsto \ov{f} }
&&
\Imm(\TT)
\ar[r]
\ar@{-->}@(u,-)[ll]_-{ \id_{\TT}\times f \mapsfrom f }
&
1
~
.
}
\end{equation}
Direct inspection of the definition of $\Aff$ reveals the downward factorizations making the commutative diagram~(\ref{e68}) among topological monoids. 
By the isotopy-extension theorem, the bottom short exact sequence among topological monoids forgets as a short exact sequence among continuous monoids. 
Using Lemma~\ref{t67},
the proof is complete upon showing that the left and right downward maps are equivalences between spaces.
It is routine to verify that the map $\Imm(\TT) \xra{\bigl(\ev_0 , \sH_1(-) \bigr)} \TT \rtimes (\ZZ \smallsetminus \{0\})^\times$ is a homotopy-inverse to the right downward map in~(\ref{e68}).

Now observe that the left downward morphism in~(\ref{e68}) fits into a diagram between short exact sequences of continuous monoids:
\begin{equation*}
\xymatrix{
1
\ar[r] 
&
\ZZ
\ar@{..>}[d]
\ar[r]^-{b \mapsto \lag 0 \rag \rtimes  {\Tiny \begin{bmatrix} 1 & b \\ 0 & 1 \end{bmatrix}}}
&
\TT \rtimes \sK 
\ar[d]
\ar[rr]_-{\id \rtimes (2,2)-{\rm entry}}
&&
\TT \rtimes (\ZZ \smallsetminus \{0\})^\times
\ar@{..>}[d]
\ar[r]
\ar@{-->}@(u,-)[ll]_-{\left(z, {\Tiny \begin{bmatrix} 1 & 0 \\ 0 & d \end{bmatrix}} \right) \mapsfrom (z, d)  }
&
1
\\
1
\ar[r] 
&
\Map\bigl(
(0\in \TT)
,
(\id \in \Imm(\TT) )
\bigr)
\ar[r]_-{\rm forget}
&
\Map\bigl( \TT , \Imm(\TT) \bigr) 
\ar[rr]_-{ \ev_0 }
&&
\Imm(\TT)
\ar[r]
\ar@{-->}@(u,-)[ll]_-{{\rm constant}_{f} \mapsfrom f}
&
1
~
.
}
\end{equation*}
The right downward map here is a homotopy-equivalence, in the same way the right downward map in~(\ref{e68}) is a homotopy-equivalence.
Through this right downward identification of $\Imm(\TT)$, the left downward map is a homotopy-equivalence, with inverse given by taking $\pi_1$.  
Using Lemma~\ref{t67}, we conclude that the middle downward map is a homotopy-equivalence, as desired.

\end{proof}

The Gram--Schmidt algorithm witnesses a deformation-retraction onto the inclusion from the intersection in $\GL_2(\RR)$:
\[
\sO(1)^2 = \sO(1)\times \sO(1) = \sO(2) \cap \GL_{\{i\}\subset 2}(\RR)
~\overset{\simeq}\hookrightarrow~ 
\GL_{\{i\}\subset 2}(\RR)
~.
\]

\begin{observation}
\label{prfr}
For each $i=1,2$, the sequence of homotopy-equivalences among topological spaces of Corollary~\ref{t25}, determined by a framing $\varphi \in \Fr(\TT^2 \xra{\pr_i} \TT)$, restricts as a sequence of homotopy-equivalences among topological spaces:
\begin{eqnarray*}
\Fr(\TT^2 \xra{\pr_i} \TT)
&
~
\xla{\cong}
~
&
\Map\bigl( \TT^2 , \GL_{\{i\}\subset 2}(\RR) \bigr)
\\
\nonumber
&
~
\xra{\simeq}
~
&
\Map \Bigl( \bigl( 0 \in \TT^2 \bigr) , \bigl( \uno \in \GL_{\{i\}\subset 2}(\RR) \bigr) \Bigr) \times \GL_{\{i\}\subset 2}(\RR)
\\
\nonumber
&
~
\xla{\simeq}
~
&
\Map \Bigl( \bigl( 0 \in \TT^2 \bigr) , \bigl( (+1 \in \sO(1) \bigr)^2 \Bigr) \times \sO(1)^2
\\
\nonumber
&
~\simeq~
&
\sO(1)^2
~.
\end{eqnarray*}

\end{observation}

\begin{observation}
\label{t63}
For each $i=1,2$, and each framing $\varphi \in \Fr(\TT^2 \xra{\pr_i} \TT)$, the diagram among topological spaces commutes:
\[
\xymatrix{
\TT^{2} \rtimes  \sE_{\{i\}\subset 2}(\ZZ) \ar[rr]^-{ {\sf Aff}_{i}} 
\ar[d]_{  \bigl(~{\rm sign~of~}(1,1){\text{-entry}} ~,~ {\rm sign~of~}(2,2){\text{-entry}} ~ \bigr) \circ {\sf proj}} 
&&
\Imm(\TT^2 \xra{\pr_i} \TT) \ar[d]^-{{\sf Orbit}_{\varphi}} 
\\
{\sf O}(1)^2
&&
\Fr(\TT^2 \xra{\pr_i} \TT) \ar[ll]_-{\rm Obs~\ref{prfr}}
.
}
\]
\end{observation}

For each $i=1,2$, the action
$
\ZZ
\xra{ \lag U_i \rag }
\sE_{\{i\}\subset 2}(\ZZ)
\lacts 
\TT^2
$
as a topological group
defines the topological submonoid 
\[
\TT^2 \underset{U_i} \rtimes \ZZ
~ \subset~ 
\TT^2 \rtimes \sE_{\{i\}\subset 2}(\ZZ)
~.
\]
After Lemma~\ref{t62} and Observation~\ref{prfr}, Observation~\ref{t63} implies the following.
\begin{cor}
\label{t64}
For each $i=1,2$, 
and each framing $\varphi \in \Fr(\TT^2 \xra{\pr_i} \TT)$, 
there are canonical identifications among continuous monoids over the identification $\Aff_i$:
\[
\xymatrix{
\TT^2 \underset{U_{i}}\rtimes  \ZZ 
\ar[r]_-{\simeq}^-{\Aff_i^{\sf fr}}
\ar[d]_-{\id \rtimes \lag 
\tau_{i}
\rag   }
&
\Diff^{\fr}(\TT^2 \xra{\pr_i} \TT , \varphi) 
\ar[d]^-{\rm forget}
&&
\TT^{2} \rtimes \sE_{\{i\}\subset 2}(\ZZ)  
\ar[r]_-{\simeq}^-{\Aff_i^{\sf fr}}
\ar[d]_-{\id \rtimes \lag 
\w{\rm inclusion}
\rag   }
&
\Imm^{\sf fr}(\TT^2 \xra{\pr_i} \TT , \varphi)
\ar[d]^-{\rm forget}
\\
\TT^2 \rtimes \Braid  
\ar[r]^-{\simeq}_-{\rm Lem~\ref{t34}}
&
\Diff^{\fr}(\TT^2 , \varphi )
&
\text{ and }
&
\TT^2 \rtimes \Ebraid 
\ar[r]^-{\simeq}_-{\rm Lem~\ref{t34}}
&
\Imm^{\fr}(\TT^2, \varphi)
.
}
\]
\end{cor}

We now explain how the presentation~(\ref{e67}) of $\Braid$ gives a presentation of the continuous group $\Diff^{\fr}(\TT^2)$.
Observe the canonically commutative diagram among continuous groups:
\[
\xymatrix{
\TT^2 \ar[r]
\ar[d]
&
\Diff^{\fr}(\TT^{2} \xra{\pr_1} \TT)
\ar[d]
\\
\Diff^{\fr}(\TT^{2} \xra{\pr_2} \TT)
\ar[r]
&
\Diff^{\fr}(\TT^2),
}
\]
which results in a morphism from the pushout, $\Diff^{\fr}(\TT^{2} \xra{\pr_1} \TT)
\underset{\TT^2}
\coprod
\Diff^{\fr}(\TT^{2} \xra{\pr_2} \TT)
\longrightarrow
\Diff^{\fr}(\TT^2).
$
Recall the element $R \in \GL_2(\ZZ)$ from~(\ref{e64}).
The two homomorphisms \begin{tikzcd}
\ZZ \arrow[rr, yshift=0.7ex, "\lag \tau_{1} \tau_{2} \tau_{1}\rag"] \arrow[swap, rr, yshift=-0.7ex, "\lag \tau_{2} \tau_{1} \tau_{2}\rag"]
& 
& \ZZ \amalg \ZZ 
\end{tikzcd} determine two morphisms among continuous groups under $\TT^2$:

\begin{equation} \label{e66}
\small{
\begin{tikzcd}
\TT^2  \underset{R}\rtimes \ZZ \arrow[rr, yshift=0.7ex, "\id \rtimes \lag \tau_{1} \tau_{2} \tau_{1}\rag"] \arrow[swap, rr, yshift=-0.7ex, "\id \rtimes \lag \tau_{2} \tau_{1} \tau_{2}\rag"]
&
&
\TT^2 \underset{U_1,U_2} \rtimes (\ZZ \amalg \ZZ)
\arrow[r, "\simeq"]
\arrow[swap, r]
&
\Diff^{\fr}(\TT^{2} \xra{\pr_1} \TT)
\underset{\TT^2}
\coprod
\Diff^{\fr}(\TT^{2} \xra{\pr_2} \TT)
\arrow[r]
&
\Diff^{\fr}(\TT^2)
~.
\end{tikzcd}
}
\end{equation}

\begin{cor}
\label{t49}
The diagram~(\ref{e66}) among continuous groups under $\TT^2$
witnesses a coequalizer.

\end{cor}

\begin{proof}
The presentation~(\ref{e67}) of $\Braid$ a coequalizer diagram among groups:
\[
\begin{tikzcd}
\ZZ 
\arrow[rr, yshift=0.7ex, " \lag \tau_{1} \tau_{2} \tau_{1}\rag"] 
\arrow[swap, rr, yshift=-0.7ex, "\lag \tau_{2} \tau_{1} \tau_{2}\rag"]
&
&
\ZZ \amalg \ZZ
\arrow[rr, "\lag \tau_{1} ~\&~ \tau_{2} \rag"]
&
&
\Braid
~.
\end{tikzcd}
\]
Taking semi-direct products with respect to the action $\Braid \xra{\Phi} \GL_2(\ZZ) \lacts \TT^2$ results in a coequalizer diagram among continuous groups:
\[
\begin{tikzcd}
\TT^2  \underset{R}\rtimes \ZZ \arrow[rr, yshift=0.7ex, "\id \rtimes \lag \tau_{1} \tau_{2} \tau_{1}\rag"] \arrow[swap, rr, yshift=-0.7ex, "\id \rtimes \lag \tau_{2} \tau_{1} \tau_{2}\rag"]
&
&
\TT^2 \underset{U_1,U_2} \rtimes (\ZZ \amalg \ZZ)
\arrow[rr, " \id \rtimes \lag \tau_1 ~ \&~ \tau_2 \rag"]
&
&
\TT^2 
\rtimes
\Braid
~.
\end{tikzcd}
\]
The result then follows from Corollary~\ref{t64}.

\end{proof}

\begin{proof}[Proof of Corollary~\ref{t49'}]
Consider the diagram among $\infty$-categories:
\[
\xymatrix{
\Mod_{\Diff^{\fr}(\TT^2)}(\cX)
\ar[r]^-=
\ar[d]
&
\Mod_{\Diff^{\fr}(\TT^2)}(\cX)
\ar[r]
\ar[d]
&
\Mod_{\TT^2\underset{U_1} \rtimes  \ZZ }(\cX)
\underset{\Mod_{\TT^2}(\cX)}\times
\Mod_{\TT^2 \underset{U_2} \rtimes \ZZ}(\cX)
\ar[d]_-{\bigl( \id \rtimes \lag \tau_{1} \tau_{2} \tau_{1}\rag \bigr)^\ast \times \bigl( \id \rtimes \lag \tau_{2} \tau_{1} \tau_{2}\rag \bigr)^\ast}
&
\Mod_{\TT^2}(\cX)^{\lag U_1 , U_2 \rag}
\ar[d]
\ar[l]_-{\simeq}^-{\rm Prop~\ref{t65}}
\\
\Mod_{\TT^2}(\cX)^{\lag R \rag}
\ar[r]^-{\simeq}_-{\rm Prop~\ref{t65}}
&
\Mod_{\TT^2 \underset{R} \rtimes \ZZ }(\cX)
\ar[r]_-{\rm diagonal}
&
\Mod_{\TT^2 \underset{R} \rtimes \ZZ}(\cX)
\underset{\Mod_{\TT^2}(\cX)}\times
\Mod_{\TT^2 \underset{R} \rtimes \ZZ}(\cX)
&
\Mod_{\TT^2}(\cX)^{\lag R , R \rag}
\ar[l]_-{\simeq}^-{\rm Prop~\ref{t65}}
.
}
\]
Corollary~\ref{t49} implies the middle square is a pullback.  
Via Proposition~\ref{t65}, which identifies modules for a semi-direct product in terms of invariants, the left and right squares are pullbacks.
Therefore, the outer square is a pullback, as desired.

\end{proof}

\section{Natural symmetries of secondary Hochschild homology}

\begin{conventions*}
\begin{enumerate}
\item[~]

\item
We fix a symmetric monoidal $\infty$-category $\cV$, and assume it is $\ot$-presentable (meaning the underlying $\infty$-category $\cV$ is presentable, and $\ot$ distributes over colimits separately in each variable).

\item
In this section, we apply the results from above only to the case of the standard framing $\varphi_0$ of the 2-torus $\TT^2$.  
So we suppress the framing $\varphi_0$ from all notation, while regarding $\TT^2$ as a framed 2-manifold. 

\end{enumerate}

\end{conventions*}

\begin{example}
For $\Bbbk$ a commutative ring, take $(\cV,\ot) = \bigl({\sf Ch}_\Bbbk[\{{\sf quasi\text{-}isos}\}^{-1}],\underset{\Bbbk}\otimes^{\LL}\bigr)$ to be the $\infty$-categorical localization of chain complexes over $\Bbbk$ on quasi-isomorphisms, with derived tensor product over $\Bbbk$ presenting the symmetric monoidal structure.
More generally, for $R$ a commutative ring spectrum, take $(\cV,\ot) := (\Mod_R,\underset{R}\wedge)$ to be the $\infty$-category of $R$-module spectra and smash product over $R$ as the symmetric monoidal structure.

\end{example}

\subsection{Hochschild homology of an associative algebra}
\label{sec.hoch1}

Let $A$ be an associative algebra in $\cV$.

Recall the paracyclic category $\para$, introduced by Getzler--Jones.
An object is a linearly ordered set $I$ with finite intervals, equipped with an order-preserving action $\ZZ \lacts I$ with the property that $i < 1\cdot i$ for each $i\in I$; a morphism is a $\ZZ$-equivariant map between linearly ordered sets.
Here are some standard facts about the paracyclic category (see, for instance,~\S4.2 of~\cite{lurie.waldhausen}).
\begin{enumerate}

\item
There is a canonical equivalence 
\[
\Hom^{\sf surj}_{\sf LinOrd}(-,[1])
\colon
\para^{\op} \xra{~\simeq~} \para
~,
\]
whose value on $(\ZZ\lacts I)$ is the set of surjective maps between linearly ordered sets from $I$ to $[1]$, equipped with inherited linear order and residual $\ZZ$-action.  

\item
The $\ZZ$-action on each object in $\para$, and the $\ZZ$-equivariance of each morphism in $\para$, assemble as an action:
\[
\sB \ZZ
~\lacts~
\para
~.
\]

\item
There is a standard functor $\bDelta \xra{[p]\mapsto [p]^{\star \ZZ}} \para$, whose value on a non-empty finite linearly ordered set its $\ZZ$-fold join as it is equipped with the $\ZZ$-action given by translating joinands.  
The resulting functor
\[
\bDelta^{\op}
\longrightarrow
\para^{\op}
~\simeq~
\para
\]
is final.

\end{enumerate}

Recall from~
\cite{loday} 
Connes' cyclic category $\bLambda$ in which an object is a cyclically ordered non-empty finite set, and a morphism is a cyclic order preserving map.
For $(\ZZ\lacts I) \in \para$ an object, the $\ZZ$-coinvariants of the underlying set $I_{/\ZZ}$ canonically retains a cyclic order; this association assembles as a functor:
\[
\para
\longrightarrow
\bLambda
~,\qquad
(\ZZ\lacts I)
\longmapsto
I_{/\ZZ}
~.
\]
This functor witnesses the $\sB \ZZ$-coinvariants:
\[
\para_{/\sB \ZZ}
\xra{~\simeq~}
\bLambda
~.
\]

Recall from~
\cite{boardman.vogt}
an explicit description of the symmetric monoidal envelope ${\sf Env}^{\ot}({\sf Assoc})$ of the associative operad.\footnote{
Specifically, an object is a finite set; a morphisms from $I$ to $J$ is a map between finite sets $I \xra{f} J$ together with a linear order on $f^{-1}(j)$ for each $j\in J$; composition is composition of maps between finite sets together with joins of finite sets; the symmetric monoidal structure is given by disjoint unions of finite sets.  
}
There is a canonical functor 
\[
\para
\longrightarrow
{\sf Env}^{\ot}({\sf Assoc})
\]
whose value an object $(\ZZ \lacts I)\in \para$ is the quotient set $I_{/ZZ}$, and whose value on a morphism $(\ZZ \lacts I) \xra{f} (\ZZ \lacts J)$ in $\para$ is the induced map between quotient sets $I_{/\ZZ} \xra{f_{/\ZZ}} J_{/\ZZ}$ together with the linear order on $f_{/\ZZ}^{-1}([j])$ inherited through the canonical bijection $I \supset f^{-1}(j) \xra{\rm bijection} f_{/\ZZ}^{-1}([j])$ for some (any) choice of $j\in [j] \in J_{/\ZZ}$.
Evidently, this functor is canonically $\sB \ZZ$ invariant, thusly canonically factoring through the $\sB \ZZ$-coinvariants:
\[
\para_{/\sB \ZZ}
~\simeq~
\bLambda
\longrightarrow
{\sf Env}^{\ot}({\sf Assoc})
~.
\]

In particular, each associative algebra $A$ in $\cV$ determines a composite functor
\[
\bBar_\bullet^{\sf cyc}(A)
\colon
\bDelta^{\op}
\longrightarrow
\para
\longrightarrow
\bLambda
\longrightarrow
{\sf Env}^{\ot}({\sf Assoc})
\xra{~A~}
\cV
~,
\]
which is the \bit{cyclic bar construction (of $A$)}.
The \bit{Hochschild homology (of $A$) (in $\cV$)} is the geometric realization of this simplicial object:
\[
\sHH(A)
~:=~
\sHH_\cV(A) 
~:= ~
A \underset{ A^{\op}\ot A} \ot A
~\simeq~
\bigl|
\bBar_\bullet^{\sf cyc}(A)
\bigr|
~\in~ \cV
~.
\]
This construction is evidently functorial in the argument $A$:
\[
\Alg_{\sf Assoc}(\cV)
\xra{~\sHH~}
\cV
~.
\]
Using finality of $\bDelta^{\op} \to \para$, the action $\TT \simeq \sB \ZZ \lacts \para$ determines an action $\TT \lacts \sHH(A)$, which is Connes' cyclic operator (see~\cite{connes}).
This action is evidently functorial in the argument $A$:
\begin{equation}
\label{e115}
\xymatrix{
&&
\Mod_{\TT}(\cV)
\ar[d]^-{\rm forget}
\\
\Alg_{\sf Assoc}(\cV)
\ar[rr]^-{\sHH}
\ar@{-->}[urr]^-{\sHH}
&&
\cV
~.
}
\end{equation}

When working over the sphere spectrum (which is to say $\cV = ({\sf Spectra},\wedge)$) so that $\sHH_{\sf Spectra}(A) = {\sf THH}(A)$ is \bit{topological Hochschild homology}, in~\cite{CT} B\"{o}kstedt-Hsaing-Madsen extend this $\TT$-action as a \bit{cyclotomic structure} on ${\sf THH}(A)$.  
In~\cite{cyclo} it is demonstrated how this cyclotomic structure on ${\sf THH}(A)$ is derived from an action of the continuous monoid $\TT \rtimes \NN^\times$ on the unstable version $\sHH_{\Spaces}(A)$.

Below, we prove Theorem~\ref{t36}, which constructs a canonical $\TT^2 \rtimes \Braid$-action on $\sHH^{(2)}(A)$, which is functorial in the 2-algebra $A$.  
We then prove Theorem~\ref{t51}, which, in the case that $\cV = (\Spaces,\times)$, extends this action to one by the continuous monoid $\TT^2 \rtimes \Ebraid$.

\subsection{Secondary Hochschild homology of 2-algebras}

In order for the Hochschild homology construction to be twice-iterated, we endow the entity $A\in \cV$ with an algebra structure among algebras.
\begin{definition}
\label{d4}
The $\infty$-category of \bit{2-algebras (in $\cV$)} is
\[
\Alg_2(\cV)
~:=~
\Alg_{\sf Assoc}\bigl(
\Alg_{\sf Assoc}(\cV)
\bigr)
~.
\]

\end{definition}

\begin{example}
A commutative algebra $A = (A,\mu)$ in $\cV$, 
determines the 2-algebra $(A,\mu,\mu)$ in $\cV$.
This association assembles as a functor
\[
\CAlg(\cV)
\longrightarrow
\Alg_2(\cV)
~,
\]
thusly supplying a host of examples of 2-algebras.

\end{example}

\begin{observation}
\label{t46}
Using that the tensor product of operads is defined by a ``hom-tensor'' adjunction, there is a canonical equivalence between $\infty$-categories:
\[
\Alg_{{\sf Assoc}\ot {\sf Assoc}}(\cV)
~\simeq~
\Alg_2(\cV)
~.
\]
In particular, swapping the two tensor-factors supplies an involution:
\[
\Sigma_2
~\lacts~ 
\Alg_2(\cV)
~.
\]

\end{observation}

\begin{remark}
\label{r12}
After Observation~\ref{t46}, a $2$-algebra in $\cV$ is an object $A\in \cV$ together with two associative algebra structures $\mu_1$ and $\mu_2$ on $A$, and compatibility between them which can be stated as either of the two equivalent structures:
\begin{itemize} 

\item
A lift of the morphism $A\ot A \xra{\mu_2} A$ in $\cV$ to a morphism $(A,\mu_1)\otimes (A,\mu_1) \xra{\mu_2} (A,\mu_1)$ in $\Alg_{\sf Assoc}(\cV)$.

\item
A lift of the morphism $A\ot A \xra{\mu_1} A$ in $\cV$ to a morphism $(A,\mu_2)\otimes (A,\mu_2) \xra{\mu_1} (A,\mu_2)$ in $\Alg_{\sf Assoc}(\cV)$.

\end{itemize}

\end{remark}

\begin{example}
\label{r13}
Consider the operad $\cE_2$ of little 2-disks.  
There is a standard morphism between operads ${\sf Assoc}\ot{\sf Assoc} \to \cE_2$ (see \cite{dunn}).
Through Observation~\ref{t46}, restriction along this morphism defines a functor between $\infty$-categories
\begin{equation}
\label{e62}
\Alg_{\cE_2}(\cV)
\longrightarrow
\Alg_2(\cV)
~,
\end{equation}
thusly supplying some rich examples of 2-algebras.
For instance, for $\Bbbk$ a commutative ring, a braided-monoidal $\Bbbk$-linear category $\sR$ is a 2-algebra in the $(2,1)$-category of $\Bbbk$-linear categories.  
Specifically, for $G$ a simply-connected reductive algebraic group over $\CC$, a choice of Killing form on its Lie algebra $\mathfrak{g}$ determines the quantum group $\cU_q \mathfrak{g}$, and thereafter the braided-monoidal category ${\sf Rep}_q( G )$ (for generic $q$).  (See~\cite{CP}, for instance.)

\end{example}

\begin{theorem}[Dunn's additivity~\cite{dunn} (see also Theorem~5.1.2.2 of~\cite{HA})]
\label{t48}
The functor~(\ref{e62})
is an equivalence between $\infty$-categories.

\end{theorem}

\begin{remark}
The action $\sO(2)
\lacts
\Alg_2(\cV)$ of Corollary~\ref{t53}, afforded by Theorem~\ref{t48}, 
extends the evident $\Sigma_2 \wr \sO(1)$-action which swaps the two associative algebra structures (as the $\Sigma_2$-factor) and takes opposites of the two associative algebra structures (as the two $\sO(1)$-factors).
\end{remark}

\begin{definition}
\label{d5}
\bit{Secondary Hochschild homology} is the composite functor, given by twice-iterating Hochschild homology:
\[
\HHt
\colon
\Alg_2(\cV)
~:=~
\Alg_{\sf Assoc}\bigl(
\Alg_{\sf Assoc}(\cV)
\bigr)
\xra{~\sHH~}
\Alg_{\sf Assoc}(\cV)
\xra{~\sHH~}
\cV
~,
\]
\[
(A,\mu_1,\mu_1)
\mapsto
\bigl( \sHH(A,\mu_1) , \sHH(\mu_2) \bigr)
\mapsto
\sHH\bigl( 
\sHH(A,\mu_1)
,
\sHH(\mu_2)
\bigr)
~.
\]

\end{definition}

The canonical lift~(\ref{e115}) supplies, for each 2-algebra $A$ in $\cV$, two commuting actions $\TT \lacts \HHt(A)$, functorially in the argument $A$:
\begin{equation}
\label{e116}
\xymatrix{
&&
\Mod_{\TT^2}(\cV)
\ar[d]
\\
\Alg_2(\cV)
\ar[rr]^-{\HHt}
\ar@{-->}[urr]^-{\HHt}
&&
\cV
~.
}
\end{equation}

\subsection{Comparison with factorization homology}

Let $n\geq 0$.
Recall from~\cite{old.fact} the symmetric monoidal $\infty$-category $\Mfld^{\fr}_n$ whose objects are (finitary) framed $n$-manifolds, whose spaces of morphisms are spaces of framed embeddings between them, and whose symmetric monoidal structure is given by disjoint union.
Let $M$ be a framed $n$-manifold.
Consider the full $\infty$-subcategories,
\[
\Disk^{\fr}_n
\overset{\iota}{~\hookrightarrow ~}
\Mfld^{\fr}_n
~\hookleftarrow ~
{\sf BDiff}^{\fr}(M)
~,
\]
respectively consisting of those framed $n$-manifolds each of whose connected components is equivalent with $\RR^n$, and by those framed $n$-manifolds that are equivalent with $M$.
The left full $\infty$-subcategory is closed with respect to the symmetric monoidal structure.
Restriction along these full $\infty$-subcategories determines the solid diagram among $\infty$-categories:
\begin{equation}
\label{e75}
\xymatrix{
\Alg_{\cE_n}(\cV)
\xla{\simeq}
\Fun^{\ot}(\Disk_n^{\fr} , \cV )
\ar@{-->}@(u,u)[rr]^-{\int}
&&
\Fun^{\ot}(\Mfld_n^{\fr} , \cV)
\ar[r]^-{\rm restrict}
\ar[ll]_-{\rm restrict}
&
\Fun\bigl({\sf BDiff}^{\fr}(M) , \cV \bigr)
\simeq
\Mod_{\Diff^{\fr}(M)}(\cV)
~.
}
\end{equation}
\bit{Factorization homology} is defined as the left adjoint to the leftward restriction functor, indicated by the dashed arrow;
factorization homology over the torus, as it is endowed with a canonical $\Diff^{\fr}(M)$-action, is the rightward composite functor:
\begin{equation}
\label{e74}
\int_{M}
\colon
\Alg_{\cE_n}(\cV)
\longrightarrow
\Mod_{\Diff^{\fr}(M)}(\cV)
~.
\end{equation}

\begin{prop}
\label{t100}
There is a canonical equivalence 
\[
\sHH 
~\simeq~ 
\int_{\TT}
\qquad
\text{ in }
\Fun
\Bigl(
\Alg_{\sf Assoc}(\cV) , \Mod_{\TT}(\cV)
\Bigr)
~.
\]

\end{prop}

\begin{proof}
Recall from~\cite{old.fact} the functor between $\infty$-categories $\Disk^{\fr}_{1/\SS^1} \xra{\rm forget} \Disk^{\fr}_1$.
Both of these a priori $\infty$-categories are ordinary categories.  
Through Example~5.1.0.7 of~\cite{HA}, taking path-components defines an equivalence between $\infty$-operads: $\cE_1 \to {\sf Assoc}$.
Proposition~2.12 of~\cite{AFT2} states an identification between symmetric monoidal $\infty$-categories: ${\sf Env}^{\ot}(\cE_1) \xra{\simeq} \Disk^{\fr}_1$.  
Consequently, 
taking path-components of disjoint unions of Euclidean spaces defines an equivalence between symmetric monoidal $\infty$-categories:
\[
\pi_0
\colon
\Disk^{\fr}_1 
~\simeq~
{\sf Env}^{\ot}(\cE_1)
\xra{~\simeq~}
{\sf Env}^{\ot}({\sf Assoc})
~.
\]
Similarly, taking path-components of disjoint unions of Euclidean spaces while remembering cyclic orders from $\SS^1$ defines a $\TT \simeq \sB \ZZ$-equivariant equivalence between $\infty$-categories filling the diagram among $\infty$-categories,
\[
\xymatrix{
&&
&&
\Disk^{\fr}_{1/\SS^1}
\ar@{-->}[d]^-{\simeq}_-{\pi_0}
\ar[rr]^-{\rm forget}
&&
\Disk^{\fr}_1
\ar[d]^-{\simeq}_-{\pi_0}
\\
\bDelta^{\op}
\ar[rr]^-{\rm final}
&&
\para
\ar[rr]^-{\rm inclusion}
&&
\para^{\tl}
\ar[rr]
&&
{\sf Env}^{\ot}({\sf Assoc})
~.
}
\]
In particular, there is a commutative diagram among $\infty$-categories:
\begin{equation}
\label{e110}
\xymatrix{
(\para)_{/\TT}
\ar[rr]
\ar[drr]
&&
(\para^{\tl})_{/\TT}
\ar[d]
&&
( \Disk^{\fr}_{1/\SS^1} )_{/\TT}
\ar[ll]_-{\simeq}
\ar[dll]
\\
&&
{\sf Env}^{\ot}({\sf Assoc})
&&
.
}
\end{equation}

We now explain the diagram among $\infty$-categories:
\[
\medmath{
\xymatrix{
&
&
\Fun\bigl( \Disk^{\fr}_{1/\SS^1}, \cV \bigr)^{\TT}
\ar[drr]^-{\colim}
&&
\\
\Alg_{\sf Assoc}(\cV)
\simeq
\Fun^{\ot}\bigl( {\sf Env}^{\ot}({\sf Assoc}) , \cV \bigr)
\ar[r]
&
\Fun\bigl( {\sf Env}^{\ot}({\sf Assoc}) , \cV \bigr)
\ar[r]
\ar[ur]
\ar[dr]
&
\Fun\bigl( \para^{\tl} , \cV \bigr)^{\TT}
\ar[d]
\ar[u]^-{\simeq}
\ar[rr]^-{\colim}
&&
\cV^{\TT}
\ar[d]^-{\simeq}
\\
&
&
\Fun\bigl( \para, \cV \bigr)^{\TT}
\ar[urr]_-{\colim}
&&
\Mod_{\TT}(\cV)
.
}
}
\]
The rightward functor on the left is the forgetful functor from symmetric monoidal functors to functors between underlying $\infty$-categories.
The equivalence on the left is the universal property of symmetric monoidal envelopes.
Restriction along the diagram~(\ref{e110}) defines the two triangles involving unlabeled functors, where the superscript denotes the $\TT$-invariants with respect to the action on the domain-argument of each functor $\infty$-category.  
The functors labeled by $\colim$ are given by taking colimits.
The right vertical equivalence is definitional, using that the $\TT$-action on $\cV$ is understood as trivial.
The upper right triangle commutes because the functor $\Disk^{\fr}_{1/\SS^1} \xra{\simeq} \para^{\tl}$ is an equivalence and in particular final.
Finality of $\bDelta^{\op} \to \para$, together with the fact that $\bDelta$ has a final object, implies the $\infty$-groupoid-completion of $\para$ contractible.
This implies the functor $\para \hookrightarrow \para^{\tl}$ is final, which proves the lower triangle commutes.

To finish, the definition of $\int_{\TT}$ is the upper composite functor, while the definition of $\sHH$ is the lower composite functor.  

\end{proof}

\begin{cor}
\label{t89}
There is a canonical equivalence 
\[
\HHt
~\simeq~ 
\int_{\TT^2}
\qquad
\text{ in }
\Fun
\Bigl(
\Alg_{2}(\cV) , \Mod_{\TT^2}(\cV)
\Bigr)
~.
\]

\end{cor}

\begin{proof}
The sought equivalence is a concatenation of the following sequence of equivalences in the $\infty$-category 
$\Fun
\Bigl(
\Alg_{2}(\cV) , \Mod_{\TT^2}(\cV)
\Bigr)$,
\begin{eqnarray*}
\HHt(-)
&
~\simeq~
&
\sHH \bigl( \sHH(-) \bigr)
\\
&
~\simeq~
&
\int_{\TT} \bigl( \int_{\TT} (-) \bigr)
\\
&
~\simeq~
&
\int_{\TT^2}(-)
\end{eqnarray*}
which we now explain.
The first equivalence is the definition of secondary Hochschild homology.
The second equivalence is two applications of Proposition~\ref{t100}.
The third equivalence is a consequence of the pushforward formula (Proposition~3.23) of~\cite{old.fact}.

\end{proof}

Swapping the order of pushforward immediately implies the following.
\begin{cor}
\label{t45}
For $A = (A,\mu_1,\mu_2)$ a 2-algebra in $\cV$, the two iterations of Hochschild homology canonically agree:
\[
\sHH
\bigl(
\sHH(A,\mu_1)
,
\sHH(\mu_2)
\bigr)
~\simeq~
\sHH
\bigl(
\sHH(A,\mu_2)
,
\sHH(\mu_1)
\bigr)
~.
\]

\end{cor}

\subsection{Comparing sheers}
Here, we show the sheer symmetries of $\HHt$ agree.

Consider the composite morphism between continuous groups:
\[
\lag \tau_1 \rag
\colon
\ZZ
\hookrightarrow
\TT^2 \underset{U_1} \rtimes \ZZ
\xra{~\Aff_1~}
\Diff^{\fr}(\pr_1)
\longrightarrow
\Diff^{\fr}(\TT^2)
~.
\]
Note that the composition $\Diff^{\fr}(\TT^2) \to \Diff(\TT^2) \xla{\simeq} \TT^2 \rtimes \GL_2(\ZZ)$ carries $\tau_1$ to the sheering matrix $U_1 = \begin{bmatrix} 1 & 1 \\ 0 & 1 \end{bmatrix}\in \GL_2(\ZZ)$.

\begin{prop}
\label{t101}
The diagram among $\infty$-categories
\[
\xymatrix{
\Alg_2(\cV)
\ar[d]_-{\ZZ \underset{{\sf Sheer}_1}\lacts \HHt}
&&
\Alg_{\cE_2}(\cV)
\ar[ll]_-{{\sf fgt}_1}
\ar[rr]^-{{\sf fgt}_2}
\ar[d]_-{\int_{\TT^2}}
&&
\Alg_2(\cV)
\ar[d]^-{\ZZ \underset{{\sf Sheer}_2^{-1}}\lacts \HHt}
\\
\Mod_{\ZZ}(\cV)
&&
\Mod_{\Diff^{\fr}(\TT^2)}(\cV)
\ar[rr]^-{\lag \tau_2 \rag^\ast}
\ar[ll]_-{\lag \tau_1 \rag^\ast}
&&
\Mod_{\ZZ}(\cV)
}
\]
canonically commutes.
In other words, for each $\cE_2$-algebra $A$ in $\cV$, there are canonical identifications between the two symmetries of $\HHt(A)$,
\begin{equation}
\label{e120}
\lag \tau_1 \rag
~\simeq~
{\sf Sheer}_1
\qquad
\text{ and }
\qquad
\lag \tau_2 \rag
~\simeq~
{\sf Sheer}_2^{-1}
~,
\end{equation}
functorially in $A \in \Alg_{\cE_2}(\cV)$.

\end{prop}

\begin{proof}
By swapping the two coordinates of $\TT^2$, commutativity of the left square implies commutativity of the right square.
So we only establish commutativity of the left square.

Notice that this diagram is functorial in the presentably symmetric monoidal $\infty$-category $\cV$.
Therefore, commutativity of this diagram for any presentably symmetric monoidal $\infty$-category $\cV$ is implied by an identification~(\ref{e120}) in the case that the pair $(A,\cV)$ is initial among presentably symmetric monoidal $\infty$-categories equipped with an $\cE_2$-algebra.

We first identify the initial presentably symmetric monoidal $\infty$-category equipped with an $\cE_2$-algebra.
Day convolution supplies a symmetric monoidal structure on the $\infty$-category $\PShv(\Disk_2^{\fr})$.
By construction, this symmetric monoidal $\infty$-category is $\ot$-presentable.  
Also, the Yoneda embedding $\Disk_2^{\fr} \xra{\rm Yoneda} \PShv(\Disk_2^{\fr})$ is canonically symmetric monoidal.
Via the equivalence $\Alg_{\cE_2}(\cV) \xla{\simeq}\Fun^{\ot}( \Disk^{\fr}_2 , \cV)$, the Yoneda functor is an $\cE_2$-algebra in $\PShv(\Disk_2^{\fr})$.  
Furthermore, it is initial among presentably symmetric monoidal $\infty$-categories equipped with an $\cE_2$-algebra.  
Indeed, for $A\in \Alg_{\cE_2}(\cV) \xla{\simeq} \Fun^{\ot}(\Disk_2^{\fr} , \cV)$ an $\cE_2$-algebra in $\cV$, left Kan extension of $A$ along the Yoneda functor is the unique colimit-preserving symmetric monoidal filler:
\[
\xymatrix{
&
\Disk_2^{\fr}
\ar[dl]_-{\rm Yoneda}
\ar[dr]^-{A}
&
\\
\PShv(\Disk_2^{\fr})
\ar@{-->}[rr]^-{\sf LKE}
&&
\cV
.
}
\]

Recall the fully faithful symmetric monoidal functor $\Disk_2^{\fr} \xra{\iota} \Mfld_2^{\fr}$.
The restricted Yoneda functor associated to $\iota$ is
\[
\Mfld_2^{\fr}
\xra{~\rm restricted~Yoneda~}
\PShv(\Disk_2^{\fr})
~,\qquad
M
\longmapsto
\Hom_{\Mfld_2^{\fr}}(\iota , M)
~,
\]
which is canonically symmetric monoidal.
The definition of factorization homology is such that there is a canonical morphism in $\Fun^{\ot}\bigl( \Mfld_2^{\fr} , \PShv( \Disk_2^{\fr} ) \bigr)$,
\begin{equation}
\label{e131}
\int_-  {\rm Yoneda}
\xra{~\simeq~}
\Hom_{\Mfld_2^{\fr}}\bigl( \iota , - \bigr)
~.
\end{equation}
This morphism is an equivalence.
Indeed, unpacking definitions, and identifying presheaves with right fibrations via the (un)straightening equivalence, the unstraightening of this morphism is a functor between right fibrations over $\Disk^{\fr}_n$:
\[
\int_- \Disk^{\fr}_{n/\RR^n}
\longrightarrow
\Disk^{\fr}_{n/-}
~.
\]
As explained in Example~\ref{r50}, this functor is an equivalence.  
In particular, we have a canonical composite equivalence
\begin{equation}
\label{e130}
\HHt({\rm Yoneda})
\underset{\rm Cor~\ref{t89}}{~\simeq~}
\int_{\TT^2} {\rm Yoneda}
\underset{(\ref{e131})}{~\simeq~}
\Hom_{\Mfld_2^{\fr}}\bigl( \iota , \TT^2 \bigr)
~\qquad
\text{ in }
\PShv(\Disk_2^{\fr}) 
~.
\end{equation}
Also, the symmetric monoidal functor
\[
\Hom_{\Mfld_2^{\fr}}\bigl( \iota , \TT \times - \bigr)
\colon
\Disk_1^{\fr}
\xra{~\TT \times-~}
\Mfld_2^{\fr}
\xra{~\rm restricted~Yoneda~}
\PShv(\Disk_2^{\fr})
\]
is the Hochschild homology of the 2-algebra in $\PShv(\Disk_2^{\fr})$ underlying the $\cE_2$-algebra $\iota$, as it is equipped with its residual associative algebra structure:
\begin{equation}
\label{e132}
\sHH( {\rm Yoneda} )
\underset{\rm Prop~\ref{t100}}{~\simeq~}
\int_{\TT \times \RR^{\sqcup \bullet} } {\rm Yoneda}
\underset{(\ref{e131})}{~\simeq~}
\Hom_{\Mfld_2^{\fr}}\bigl( \iota , \TT \times \RR^{\sqcup \bullet} \bigr)
~\qquad
\text{ in }
\Alg_{\sf Assoc}\bigl(
\PShv(\Disk_2^{\fr}) 
\bigr)
~.
\end{equation}

Now, taking mapping tori defines a map between pointed spaces $\Diff^{\fr}(\TT) \to \sB \Diff^{\fr}(\TT^2)$.
Based loops of this map is the morphism between continuous groups 
\[
\lag \tau_1 \rag
\colon
\ZZ
\xra{~\simeq~}
\Omega \TT
\xra{~\simeq~}
\Omega \Diff^{\fr}(\TT)
\xra{~\Omega (\rm mapping~torus)~}
\Diff^{\fr}(\TT^2)
~.
\]
By construction of the morphism~(\ref{e56}), this fills the diagram among continuous groups:
\[
\xymatrix{
\ZZ
\ar[d]_-{\simeq}
\ar[rrr]
&
&&
\ZZ
\ar[dl]
\ar@(-,r)[ddl]^-{{\sf Sheer}_1}
\\
\Omega \Diff^{\fr}(\TT)
\ar@{-->}[d]^-{\Omega (\rm mapping~torus)}
\ar[r]
&
\Omega 
\Aut_{\PShv(\Disk_2^{\fr}) }\Bigl(
\Hom_{\Mfld_2^{\fr}}\bigl( \iota , \TT \times \RR \bigr) 
\Bigr)
\ar[r]^-{\simeq}_-{(\ref{e132})}
&
\Omega \Aut_{\PShv(\Disk_2^{\fr}) }\Bigl(
\sHH( \iota )
\Bigr)
\ar[d]_-{(\ref{e56})}
\\
\Diff^{\fr}(\TT^2)
\ar[r]
&
\Aut_{\PShv(\Disk_2^{\fr}) }\Bigl(
\Hom_{\Mfld_2^{\fr}}\bigl( \iota , \TT^2 \bigr) 
\Bigr)
\ar[r]^-{\simeq}_-{(\ref{e130})}
&
\Aut_{\PShv(\Disk_2^{\fr}) }\Bigl(
\HHt( \iota )
\Bigr)
.
}
\]
Commutativity of the outer diagram is the sought identification~(\ref{e120}) in the universal case.

\end{proof}

\subsection{Proof of Theorem~\ref{t36} and Corollaries~\ref{t55},~\ref{t54}}
\label{sec.fact.hmlgy}
This subsection proves Theorem~\ref{t36} and Corollary~\ref{t55}, then Corollary~\ref{t54}.

Next, we explain the following diagram among $\infty$-categories: 
\begin{equation}
\label{e73}
\xymatrix{
\Alg_2(\cV)
\ar[d]_-{\ZZ \underset{{\sf Sheer}_1}\lacts \HHt}
&
&
\Alg_{\cE_2}(\cV)
\ar[ll]^-{\simeq}_-{{\sf fgt}_1}
\ar[rr]^-{{\sf fgt}_2}_-{\simeq}
\ar[d]^-{\int_{\TT^2}}
\ar[dr]^-{\int_{\pr_2}}
\ar[dl]_-{\int_{\pr_1}}
&
&
\Alg_2(\cV)
\ar[d]^-{\ZZ \underset{{\sf Sheer}_2^{-1}}\lacts \HHt}
\\
\Mod_{\ZZ}(\cV)
\ar[d]_-{\rm forget}
&
\Mod_{\Diff^{\fr}(\pr_1)}(\cV)
\ar[l]_-{\rm forget}
\ar[dr]_-{\rm forget}
&
\Mod_{\Diff^{\fr}(\TT^2)}(\cV)
\ar[d]^-{\rm forget}
\ar[r]^-{\rm forget}
\ar[l]_-{\rm forget}
&
\Mod_{\Diff^{\fr}(\pr_2)}(\cV)
\ar[r]^-{\rm forget}
\ar[dl]^-{\rm forget}
&
\Mod_{\ZZ}(\cV)
\ar[d]^-{\rm forget}
\\
\cV
&
&
\Mod_{\TT^2}(\cV)
\ar[rr]^-{\rm forget}
\ar[ll]_-{\rm forget}
&
&
\cV
.
}
\end{equation}
\begin{itemize}
\item
The functors labeled ``$\rm forget$'' are restriction along 
the canonically commutative diagram among continuous groups:
\[
\xymatrix{
\ZZ
\ar[rr]^-{\lag \tau_1 \rag}
&&
\Diff^{\fr}(\pr_1)
\ar[rr]
&&
\Diff^{\fr}(\TT^2)
&&
\Diff^{\fr}(\pr_2)
\ar[ll]
&&
\ZZ
\ar[ll]_-{\lag \tau_2 \rag}
\\
\ast
\ar[u]
&&
&&
\TT^2
\ar[u]
\ar[ull]
\ar[urr]
\ar[llll]
\ar[rrrr]
&&
&&
\ast
\ar[u]
,
}
\]
in which, for each $i=1,2$, the the morphism $\lag \tau_i \rag$ is the composite $\ZZ \hookrightarrow \TT^2 \underset{U_i} \rtimes \ZZ \xra{\Aff_i} \Diff^{\fr}(\pr_i)$.
In particular, each of the lower triangles canonically commutes.

\item
The functor $\int_{\TT^2}$ is~(\ref{e74}).

\item
For $i=1,2$, the functor $\int_{\pr_i}$ is factorization homology over the circle $\TT$ of the \bit{pushforward} along the projection $\TT^2 \xra{\pr_i} \TT$ off of the $i^{\rm th}$-coordinate, as it is endowed with its canonical $\Diff^{\fr}(\pr_i)$-action.
The \bit{pushforward formula}
$
\int_{\pr_i}
\simeq
\int_{\TT}
\int_{\TT}
$
(see Proposition~{3.23} of~\cite{old.fact}), which is manifestly $\Diff^{\fr}(\pr_i)$-equivariant,
supplies commutativity of the upper triangles.

\item
The functor $\Alg_{\cE_2}(\cV) \xra{{\sf fgt}_1} \Alg_2(\cV)$ is restriction along the standard morphism between operads ${\sf Assoc}\ot {\sf Assoc} \xra{\rm standard} \cE_2$.
The functor $\Alg_{\cE_2}(\cV) \xra{{\sf fgt}_2} \Alg_2(\cV)$ is restriction along the morphism between operads ${\sf Assoc}\ot {\sf Assoc} \xra{\rm swap} {\sf Assoc}\ot {\sf Assoc} \xra{\rm standard} \cE_2$.
Theorem~\ref{t48} implies that each of these functors are equivalences.  

\item
For $i=1,2$, the outer vertical functors are $\HHt$, as it is endowed with its canonical action
$\ZZ \underset{{\sf Sheer}_i}\lacts \HHt(A)$ of (\ref{e112}),(\ref{e113}) from~\S\ref{sec.second}, which is evidently functorial in $A\in \Alg_2(\cV)$.

\item
Commutativity of the upper tilted squares is Proposition~\ref{t101}.

\end{itemize}
In particular, for each 2-algebra $A\in \Alg_2(\cV)$, there is a canonical action $\Diff^{\fr}(\TT^2) \lacts \HHt(A)$.
Through Theorem~\ref{Theorem A}(2a), this is an action $\TT^2 \rtimes \Braid \lacts \HHt(A)$, which establishes the statement of Theorem~\ref{t36}.

After Theorem~\ref{t36}, the standard presentation~(\ref{e67}) of the braid group $\Braid$ immediately implies Corollary~\ref{t55}(1).
Via the identification $\TT^2 \underset{U_i} \rtimes \ZZ \xra{\simeq} \Diff^{\fr}(\pr_i)$ of Corollary~\ref{t64}, commutativity of the outer squares in the diagram~(\ref{e73}) directly implies Corollary~\ref{t55}(2)(3).

Next, consider the
$\sO(2)\simeq \GL_2(\RR)$-action on $\Mfld_2^{\fr}$ given by change-of-framing.
Observe that this action restricts to as one along the full $\infty$-subcategory $\Disk^{\fr}_2 \subset \Mfld^{\fr}_2$.
This implies the left adjoint, $\int$, is $\sO(2)$-equivariant.  
Therefore, for each $A\in \Alg_2(\cV)$, and each $(\Sigma,\varphi)\in \Mfld_2^{\fr}$, taking $\sO(2)$-orbits of both $A$ and $(\Sigma,\varphi)$ define a canonically commuting diagram among $\infty$-categories
\[
\xymatrix{
\sO(2)
\ar[rr]^-{{\sf Orbit}_A}
\ar[d]_-{{\sf Orbit}_{(\Sigma,\varphi)}}
&&
\Alg_2(\cV)
\ar[d]^-{\int_{(\Sigma,\varphi)}}
\\
\Mfld_2^{\fr}
\ar[rr]^-{\int A}
&&
\cV
.
}
\]
Through Observation~\ref{t43}, restricting along $\sB \ZZ\simeq \sB \Omega_{\uno} \sO(2) \to \sO(2)$ gives the commutative diagram asserted in Corollary~\ref{t54}.

\subsection{Proof of Theorem~\ref{t51}} \label{sec.B2.proof}

After Corollary~\ref{t30}, to prove Theorem~\ref{t51} we are left to extend the action 
\[
\Diff^{\fr}(\TT^2)^{\op}
\underset{(-)^{-1}}\simeq
\Diff^{\fr}(\TT^2) \underset{\simeq}{\xla{\Aff^{\fr}}} \TT^2 \rtimes  \Braid \lacts \HHt(A)
\] to an action $\Imm^{\fr}(\TT^2)^{\op} \lacts \HHt(A)$.
We do this by extending factorization homology, via the developments of~\cite{afr2}.
Namely, recall from~\cite{afr2} the $\infty$-category $\Mfd_2^{\sfr}$ of \bit{solidly 2-framed stratified spaces}.  
Consider the full $\infty$-subcategory $\bcM_{=2}^{\sfr}\subset \Mfd_2^{\sfr}$ consisting of those solidly 2-framed stratified spaces each of whose strata is 2-dimensional.
\begin{observation}
\label{t56}
Inspection of the definition of $\Mfd_2^{\sfr}$ reveals the following.
\begin{enumerate}

\item
The moduli space of objects 
\[
\Obj( \bcM_{=2}^{\sfr})
~\simeq~
\underset{[\Sigma,\varphi]} \coprod {\sf BDiff}^{\fr}(\Sigma,\varphi)
\]
is that of a framed 2-manifold.  
In other words, there is a canonical bijection between framed-diffeomorphism-types of framed 2-manifolds and equivalence-classes of objects in $\bcM_{=2}^{\sfr}$, and for $(\Sigma,\varphi)$ a framed 2-manifold, there is a canonical identification between continuous groups:
\[
\Diff^{\fr}(\Sigma,\varphi)
~\simeq~
\Aut_{\bcM_{=2}^{\sfr}}(\Sigma,\varphi)
~.
\]

\item
Let $(\Sigma,\varphi)$ and $(\Sigma',\varphi')$ be framed 2-manifolds.
The space of morphisms from $(\Sigma,\varphi)$ to $(\Sigma',\varphi')$ in $\bcM_{=2}^{\sfr}$
~,
\[
\Hom_{\bcM_{=2}^{\sfr}}\bigl( (\Sigma,\varphi) , (\Sigma',\varphi') \bigr)
~\simeq~
\underset{[\w{\Sigma}\xra{\pi}\Sigma]} \coprod \Emb^{\fr}\bigl( (\w{\Sigma},\pi^\ast \varphi) , (\Sigma',\varphi') \bigr)_{/\Diff_{/\Sigma}(\w{\Sigma})}
~,
\]
is the moduli space of finite-sheeted covers over $\Sigma$ together with a framed-embedding from its total space to $(\Sigma',\varphi')$.

\item
Composition in $\bcM_{=2}^{\sfr}$ is given by base-change of framed embeddings along finite-sheeted covers, followed by composition of framed-embeddings:
\[
\Hom_{\bcM_{=2}^{\sfr}}\bigl( (\Sigma,\varphi) , (\Sigma',\varphi') \bigr)
\times
\Hom_{\bcM_{=2}^{\sfr}}\bigl( (\Sigma',\varphi') , (\Sigma'',\varphi'') \bigr)
\xra{~\circ~}
\Hom_{\bcM_{=2}^{\sfr}}\bigl( (\Sigma,\varphi) , (\Sigma'',\varphi'') \bigr)
~,
\]
\[
\Bigl(
~
(\Sigma,\varphi) \xla{\pi} (\w{\Sigma},\pi^\ast \varphi) \xra{f} (\Sigma',\varphi')
~,~
(\Sigma',\varphi') \xla{\pi'} (\w{\Sigma}',{\pi'}^\ast \varphi') \xra{g} (\Sigma'',\varphi'')
~
\Bigr)
\]
\[
~\longmapsto~
\Bigl(
~
(\Sigma,\varphi) 
\xla{\pi \circ \pr_1} 
\bigl(
\w{\Sigma} 
\underset{\Sigma'} \times 
\w{\Sigma}',(\pr_1\circ \pi)^\ast \varphi
\bigr) 
\xra{g \circ \pr_2} 
(\Sigma'',\varphi'')
~
\Bigr)
~.
\]

\item
Evidently, framed embeddings form the left factor in a factorization system on $\bcM_{=2}^{\sfr}$, whose right factor is (the opposite of) framed finite-sheeted covers.

\item
Finite products exist in $\bcM_{=2}^{\sfr}$, and are implemented by disjoint unions of framed 2-manifolds.

\item
For each framing $\varphi$ of the 2-torus $\TT^2$, there is a canonical identification between continuous monoids:
\[
\Imm^{\fr}( \TT^2,\varphi)^{\op}
~\simeq~
\End_{\bcM_{=2}^{\sfr}}(\TT^2,\varphi)
~.
\]

\end{enumerate}

\end{observation}

Denote the full $\infty$-subcategory
\[
\iota
\colon
\bcD_{=2}^{\sfr}
~\subset~
\bcM_{=2}^{\sfr}
\]
consisting of those framed 2-manifolds that are equivalent with a finite disjoint union of framed Euclidean spaces.
Regard both $\bcD_{=2}^{\sfr}$ and $\bcM_{=2}^{\sfr}$ as symmetric monoidal $\infty$-categories, via their Cartesian monoidal structures.\footnote{Indeed, notice that the full $\infty$-subcategory $\bcD_{=2}^{\sfr}\subset \bcM_{=2}^{\sfr}$ is closed under finite products.}
Notice the evident monomorphisms of symmetric monoidal $\infty$-categories,
\[
\rho
\colon 
\Disk_2^{\fr}
~\hookrightarrow~
\bcD_{=2}^{\sfr}
\qquad
\text{ and }
\qquad
\rho
\colon 
\Mfld_2^{\fr}
~\hookrightarrow~
\bcM_{=2}^{\sfr}
~,
\]
each of whose images consists of all objects yet only those morphisms $\bigl( (\Sigma,\varphi) \xla{\pi} (\w{\Sigma},\pi^\ast \varphi) \xra{f} (\Sigma',\varphi')\bigr)$ in which $\pi$ is a diffeomorphism.\footnote{In other words, $\rho$ is the inclusion of the left factor in the factorization system of Observation~\ref{t56}(4).}

Let $\cX$ be a presentable $\infty$-category in which products distribute over colimits.
Consider the full $\infty$-subcategory
\[
\Fun^{\times}\bigl(
\bcD_{2}^{\sfr}, \cX 
\bigr)
~\subset~
\Fun\bigl(
\bcD_{2}^{\sfr}
, 
\cX 
\bigr)
\]
consisting of those functors that preserve finite products.  
\begin{prop}
[\cite{EIC}]
\label{t58}
Let $\cX$ be a presentable $\infty$-category in which products distribute over colimits.
Restriction along $\rho$ defines an equivalence between $\infty$-categories:
\[
\rho^\ast
\colon
\Fun^{\times}\bigl(
\bcD_{2}^{\sfr}, \cX 
\bigr)
\xra{~\simeq~}
\Fun^{\ot}\bigl( \Disk_2^{\fr} , \cX \bigr)
~\simeq~
\Alg_{\cE_2}(\cX)
~.
\]

\end{prop}

The inverse of restriction along $\rho$ followed by left Kan extension along $\iota$ defines a composite functor
\[
\w{\int}
\colon
\Alg_{\cE_2}(\cX)
~\simeq~
\Fun^{\ot}( \Disk_2^{\fr} , \cX )
\xra{~(\rho^\ast)^{-1}~}
\Fun^{\times}( \bcD_{=2}^{\sfr} , \cX )
\xra{~\iota_!~}
\Fun^{\times}( \bcM_{=2}^{\sfr} , \cX )
~.
\]

\begin{prop}
\label{t57}
Let $\cX$ be a presentable $\infty$-category in which products distribute over colimits.
The diagram among $\infty$-categories canonically commutes:
\[
\xymatrix{
\Alg_{\cE_2}(\cX)
\ar[rr]^-{\w{\int}}
\ar[d]_-{\int}
&&
\Fun^{\times}( \bcM_{=2}^{\sfr} , \cX )
\ar[rr]^-{\rm restriction}
&&
\Fun\bigl( {\sf BAut}_{\bcM_{=2}^{\sfr}}(\TT^2,\varphi_0) , \cX \bigr)
\ar[d]_-{\simeq}^-{\rm Obs~\ref{t56}(1)}
\\
\Fun^{\ot}( \Mfld_2^{\fr}, \cX )
\ar[rr]^-{\rm restriction}
&&
\Fun\bigl( {\sf BAut}_{\Mfld_2^{\fr}}(\TT^2,\varphi_0) , \cX \bigr)
\ar[rr]^-{\simeq}
&&
\Mod_{\Diff^{\fr}(\TT^2,\varphi_0)}(\cX)
.
}
\]

\end{prop}

\begin{proof}
Let $A\in \Alg_{\cE_2}(\cX)\simeq \Fun^{\ot}(\Disk_2^{\fr} , \cX)$.
Using Proposition~\ref{t58}, the monomorphism $\rho$ determines a canonical morphism between colimits in $\cX$:
\begin{eqnarray}
\label{e78}
\int_{\TT^2} A
&
~\simeq~
&
\colim
\Bigl(
\Disk_{2/(\TT^2,\varphi_0)}^{\fr}:=
\Disk_2^{\fr} \underset{\Mfld_2^{\fr}}\times \Mfld_{2/(\TT^2,\varphi_0)}^{\fr}
\xra{\pr}
\Disk_2^{\fr}
\xra{A}
\cX
\Bigr)
\\
\nonumber
&
\xra{~\rho~}
&
\colim
\Bigl(
\bcD_{=2/(\TT^2,\varphi_0)}^{\sfr}:=
\bcD_{=2}^{\sfr} \underset{\bcM_{=2}^{\sfr}}\times \bcM_{=2/(\TT^2,\varphi_0)}^{\sfr}
\xra{\pr}
\bcD_{=2}^{\sfr}
\xra{{\rho^\ast}^{-1}(A)}
\cX
\Bigr)
\\
\nonumber
&
~\simeq~
&
\w{\int}_{\TT^2} A
~.
\end{eqnarray}
This morphism is manifestly $\Diff^{\fr}(\TT^2)$-equivariant and functorial in $A\in \Alg_{\cE_2}(\cX)$ as so.
So the proposition is proved upon showing this morphism~(\ref{e78}) is an equivalence.
The morphism~(\ref{e78}) is an equivalence provided the canonical functor
\begin{equation}
\label{e79}
\Disk_{2/(\TT^2,\varphi_0)}^{\fr}
\longrightarrow
\bcD_{=2/(\TT^2,\varphi_0)}^{\sfr}
\end{equation}
is final.
But the factorization system of Observation~\ref{t56}(4) reveals that this functor~(\ref{e79}) is a right adjoint: its left adjoint given by projecting to the right factor of the factorization system:
\[
\bcD_{=2/(\TT^2,\varphi_0)}^{\sfr}
\longrightarrow
\Disk_{2/(\TT^2,\varphi_0)}^{\fr}
~,\qquad
\bigl(D \xla{\pi} \w{D} \xra{f} (\TT^2,\varphi_0) \bigr)
\mapsto
\bigl(
\w{D} \xra{f} (\TT^2,\varphi_0)
\bigr)
~.
\]
The sought finality of the functor~(\ref{e79}) follows.

\end{proof}

Proposition~\ref{t57}, together with Observation~\ref{t56}(6), immediately supply a filler in the commutative diagram among $\infty$-categories:
\begin{equation}
\label{e76}
\xymatrix{
&
\Fun\bigl( \fB \End_{\bcM_{=2}^{\sfr}}(\TT^2,\varphi_0) , \cX \bigr)
\ar[rr]^-{\simeq}_-{\rm Obs~\ref{t56}(6)}
&&
\Mod_{\Imm^{\fr}(\TT^2)^{\op}}(\cX)
\ar[d]^-{\rm forget}
\ar[r]^-{\simeq}_{\rm Cor~\ref{t30}}
&
\Mod_{\left( \TT^2 \rtimes \Ebraid \right)^{\op} }(\cX)
\ar[d]^-{\rm forget}
\\
\Alg_{\cE_2}(\cX)
\ar[rrr]^-{\bigl \lag \Diff^{\fr}(\TT^2) \lacts \int_{\TT^2} \bigr \rag}_-{(\ref{e75})}
\ar@(u,-)@{-->}[ur]^-{\bigl \lag \Imm^{\fr}(\TT^2)^{\op} \lacts  \w{\int}_{\TT^2} \bigr \rag}
&
&&
\Mod_{\Diff^{\fr}(\TT^2)}(\cX)
\ar[r]^-{\simeq}_{\rm Thm~\ref{Theorem A}(2a)}
&
\Mod_{\TT^2 \rtimes \Braid}(\cX)
.
}
\end{equation}
Theorem~\ref{t51} follows from this commutative diagram~(\ref{e76})
after the commutative diagram~(\ref{e73}).

\appendix
\section{some facts about continuous monoids}
\label{sec.A}
We record some simple formal results concerning continuous monoids.

\begin{lemma}\label{t2}
Let $G \lacts X$ be an action of a continuous group on a space.
Let $\ast \xra{\lag x \rag } X$ be a point in this space.
Consider the stabilizer of $x$, which is the fiber of the orbit map of $x$:
\begin{equation}
\label{e52}
\xymatrix{
{\sf Stab}_G(x)
\ar[rrrr]
\ar[d]
&&
&&
\ast
\ar[d]^-{\lag x \rag}
\\
G \simeq G\times \ast 
\ar[rr]^-{ \id \times \lag x \rag} 
\ar@(u,u)[rrrr]_-{{\sf Orbit}_x}
&&
G \times X 
\ar[rr]^-{\sf act} 
&&
X
.
}
\end{equation}
There is a canonical identification in $\Spaces$ between this stabilizer and the based-loops at $[x]\colon \ast \xra{\lag x \rag} X \xra{\rm quotient} X_{/G}$ of the $G$-coinvariants,
\[
{\sf Stab}_G(x)
~\simeq~
\Omega_{[x]} ( X_{/G} )
~,
\]
through which the resulting composite morphism $\Omega_{[x]} ( X_{/G} )
\simeq {\sf Stab}_G(x)
\to 
G
$
canonically lifts to one between continuous groups.

\end{lemma}

\begin{proof}
By definition of a $G$-action, the orbit map $G\xra{{\sf Orbit}_x} X$ is canonically $G$-equivariant.
Taking $G$-coinvariants supplies an extension of the commutative diagram~(\ref{e52}) in $\Spaces$:
\[
\xymatrix{
{\sf Stab}_G(x) 
\ar[rr] \ar[d]
&&
G 
\ar[d]^-{{\sf Orbit}_x} \ar[rr]^-{\rm quotient}
&&
G_{/G} 
\simeq 
\ast
\ar[d]^-{({\sf Orbit}_x)_{/G}}
\\
\ast \ar[rr]^-{\lag x\rag}
&&
X
\ar[rr]^-{\rm quotient}
&&
X_{/G}
.
}
\]
Through the identification $G_{/G} \simeq \ast$, the right vertical map is identified as $\ast\xra{\lag [x] \rag} X_{/G}$.
Using that groupoids in $\Spaces$ are effective, the right square is a pullback.  
Because the lefthand square is defined as a pullback, it follows that the outer square is a pullback.
The identification ${\sf Stab}_G(x) \simeq \Omega_{[x]} ( X_{/G} ) $ follows.  
In particular, the space ${\sf Stab}_G(x)$ has the canonical structure of a continuous group.

Now, this continuous group ${\sf Stab}_G(x)$ is evidently functorial in the argument $G \lacts X \ni x$.  
In particular, the unique $G$-equivariant morphism $X\xra{!} \ast$ determines a morphism between continuous groups:
\[
{\sf Stab}_x(X)
\longrightarrow
{\sf Stab}_{\ast}(\ast)
~\simeq~
G
~.
\]

\end{proof}

\begin{lemma}
\label{t66}
Let $H \to G$ be a morphism between continuous groups.
Let $H\lacts X$ be an action on a space.  
There is a canonical map between spaces over $G_{/H}$, 
\[
X_{/\Omega(G_{/H})}
\longrightarrow 
(X \times G)_{/H}
~,
\]
from the coinvariants with respect to the action 
$
\Omega (G_{/H})
\xra{\Omega\text{-\rm Puppe}}
H
\lacts
X
$.
Furthermore, if the induced map $\pi_0(H)\to \pi_0(G)$ between sets of path-components is surjective, then this map is an equivalence.

\end{lemma}

\begin{proof}
The construction of the $\Omega$-Puppe sequence is so that the morphism 
$
\Omega (G_{/H})
\to 
H
$
witnesses the stabilizer of $\ast \xra{\rm unit} G$ with respect to the action $H \to G \underset{\rm left~trans}\lacts G$:
\[
\xymatrix{
\Omega (G_{/H})
\ar[rr]
\ar[d]
&&
H
\ar[d]
\\
\ast
\ar[rr]^-{\rm unit}
&&
G
.
}
\]
In particular, there is a canonical $\Omega (G_{/H})$-equivariant map
\[
X
~\simeq~
X\times \ast
\xra{ \id \times {\rm unit}}
X \times G
~.
\]
Taking coinvariants lends a canonically commutative diagram among spaces:
\begin{equation}
\label{f1}
\xymatrix{
X_{ \Omega (G_{/H}) }
\ar[d]
\ar[r]
&
(X \times  G)_{/H}
\ar[r]
\ar[d]
&
X_{/H}
\ar[d]
\\
\sB \Omega (G_{/H})
\ar[r]
&
G_{/H}
\ar[r]
&
\sB H
.
}
\end{equation}
This proves the first assertion.

We now prove the second assertion.
Because groupoid-objects are effective in the $\infty$-category $\Spaces$,
the $H$-coinvariants functor,
\[
\Fun ( \sB H , \Spaces )
\longrightarrow
\Spaces_{/\sB H}
~,\qquad
(H \lacts X)
\mapsto 
( X_{/H} \to \sB H)
~,
\]
is an equivalence between $\infty$-categories.
In particular, it preserves products.
It follows that the right square in~(\ref{f1}) witnesses a pullback.  
By definition of coinvariants of the restricted action $\Omega (G_{/H}) \to H \lacts X$, the outer square is a pullback.  
The connectivity assumption on the morphism $H\to G$ implies the left bottom horizontal map is an equivalence.  
We conclude that the left top horizontal map is also an equivalence, as desired.  

\end{proof}

Let $\fB N \xra{\lag N \lacts M \rag } {\sf Monoids}$ be an action of a continuous monoid on a continuous monoid.
This action can be codified as unstraightening of the composite functor $\fB N \to {\sf Monoids} \xra{ \fB }\Cat^{\ast/}_{(\infty,1)}$.
We denote this unstraightening as 
\[
(\fB M)_{/^{\sf l.lax} N}
\longrightarrow
\fB N
~:~\footnote{
The notation here is intended to evoke \bit{left-lax quotient}.
Indeed, for $\cK \xra{F} \Cat_{(\infty,1)}$ a functor from an $\infty$-category, its \bit{left-lax colimit} is the $(\infty,1)$-category defined as the domain of the unstraightening of $F$:
\[
\Bigl(
\colim^{\sf l.lax}(F) 
\xra{\colim^{\sf l.lax}(!)}
\colim^{\sf l.lax}(\ast) 
\Bigr)
~:=~
\Bigl(
{\sf Un}(F)
\longrightarrow
\cK
\Bigr)
~.
\]
See Appendix~A of~\cite{non-com-geom} for a treatement of lax $(\infty,1)$-category theory.  
}
\]
it is a coCartesian fibration equipped with a section.
Because the $(\infty,1)$-category $\fB N$ is equipped with a functor $\ast \to \fB N$, the given section supplies the $(\infty,1)$-category $(\fB M)_{/^{\sf l.lax} N}$ with a distinguished point, and so we regard $(\fB M)_{/^{\sf l.lax} N}$ as a pointed $(\infty,1)$-category.
The \bit{semi-direct product (of $N$ by $M$)} is the continuous monoid
\[
M \rtimes N := \End_{(\fB M)_{/^{\sf l.lax} N}} ( \ast )
~,
\]
which is endomorphisms of the point.\footnote{The underlying space of this continuous monoid is canonically identified as $M \times N$;
the 2-ary monoidal structure $\mu_{M \rtimes N}$ is canonically identified as the composite map between spaces:
\begin{eqnarray*}
\mu_{M \rtimes N }
\colon
(M \times N) \times (M \times N)
=
M \times (N \times M) \times M
&
\xra{ \id_M \times {\rm swap} \times \id_N}
&
M \times ( M \times N ) \times N
\\
&
\xra{ \id_M \times \bigl( {\sf proj}_M , {\rm action} \bigr) \times \id_N }
&
M \times ( M \times N ) \times N
=
( M \times M ) \times ( N \times N ) 
\\
&
\xra{ \mu_M \times \mu_N }
&
M \times N 
~.
\end{eqnarray*}
}
Note the canonical morphism between monoids $M \rtimes N \to N$ whose kernel is $M$.

Dually, let $\fB N^{\op} \xra{\lag M \racts N \rag} {\sf Monoids}$ a \emph{right} be a action.
Consider the unstraightening of the composite functor $\fB N^{\op} \to {\sf Monoids} \xra{ \fB }\Cat^{\ast/}_{(\infty,1)}$ is a pointed Cartesian fibration
$
(\fB M)_{/^{\sf r.lax} N^{\op}}
\to
\fB N
$
.
The \bit{semi-direct product (of $N$ by $M$)} is the continuous monoid
\[
N \ltimes M := \End_{(\fB M)_{/^{\sf r.lax} N^{\op}}} ( \ast )
~,
\]
which is endomorphisms of the point.  
Note the canonical morphism between monoids $M \rtimes N \to N$ whose kernel is $M$.

\begin{observation}
\label{f5}
Let $N \lacts M$ be an action of a continuous monoid on a continuous monoid.
There is a canonical identification between continuous monoids under $M^{\op}$ and over $N^{\op}$:
\[
\left(
M \rtimes N
\right)^{\op}
~\simeq~
\left(
N^{\op}
\ltimes 
M^{\op}
\right)
~.
\]

\end{observation}

The next result is a characterization of semi-direct products.  
\begin{lemma}
\label{t67}
Let $A \overset{i}{\underset{r} \leftrightarrows} N$ be a retraction between continuous monoids (so $r \circ i \simeq \id_N$).
\begin{itemize}

\item
If the canonical map between spaces
\begin{equation}
\label{f2}
\Ker(r) 
\times 
N
\xra{ {\rm inclusion} \times i }
A \times A
\xra{\mu_A}
A
\end{equation}
is an equivalence,\footnote{Note that this condition is always satisfied if $N$ is a continuous group.}
then there is a canonical action $N \underset{\lambda}\lacts \Ker(r)$ \footnote{
The action map associated to $\lambda$ can be written as the composition
\[
N
\times
\Ker(r)
\xra{ i \times {\rm inclusion}}
A \times A
\xra{~\mu_A~}
A
\underset{(\ref{f2})}{\xla{~\simeq~}}
\Ker(r)\times N
\xra{~\sf proj~}
\Ker(r)
~.
\]
}
for which there is a canonical equivalence between monoids:
\[
\Ker(r) \underset{\lambda} \rtimes N
~\simeq~
A
~.
\]

\item
If the canonical map between spaces
\[
N
\times 
\Ker(r) 
\xra{ \sigma \times {\rm inclusion} }
A \times A
\xra{\mu_A}
A
\]
is an equivalence,\footnote{Note that this condition is always satisfied if $N$ is a continuous group.}
then there is a canonical action $\Ker(r) \underset{\rho}\racts N $ for which there is a canonical equivalence between monoids:
\[
\Ker(r) \underset{\rho} \rtimes N
~\simeq~
A
~.
\]

\end{itemize}

\end{lemma}

\begin{proof}
By way of Observation~\ref{f5}, the two assertions imply one another by taking Cartesian/coCartesian duals of coCartesian/Cartesian fibrations.
So we are reduced to proving the first assertion.

Consider the retraction $\fB A \overset{\fB i}{\underset{\fB r} \leftrightarrows} \fB N$ among pointed $\infty$-categories.
Note that $\fB i$ is essentially surjective.  
Note that $\Ker(r)$ is the fiber of $\fB r$ over $\ast \to \fB N$.

Let $c_1 \xra{\lag n \rag} \fB N$ be a morphism.
Consider the commutative diagram among $\infty$-categories:
\[
\xymatrix{
c_0
\ar[rr]^-{\lag \ast \rag}
\ar[d]_-{s}
&&
\fB A
\ar[d]^-{\fB r}
\\
c_1
\ar[rr]^-{\lag n \rag}
\ar@{-->}[urr]^-{\lag i(n) \rag}
&&
\fB N
.
}
\]
The assumption on the retraction implies the diagonal filler is initial among all such fillers.
This is to say that the morphism $i(n)$ in $\fB A$ is coCartesian over $\fB r$. 
Because $\fB i$ is essentially surjective, this shows that $\fB r$ is a coCartesian fibration.
The result now follows from the definition of the semi-direct product $\Ker(r) \underset{\lambda} \rtimes N$.    

\end{proof}

\begin{prop}
\label{t65}
Let $\cX$ be an $\infty$-category.
Let $\fB N \xra{\lag  N \lacts M \rag} {\sf Monoids}$ be an action of a continuous monoid $N$ on a continuous monoid $M$.
Consider the pre-composition-action:
\[
\fB N^{\op}
\xra{~ \lag N \lacts M \rag^{\op}~ } 
{\sf Monoids}^{\op}
\xra{~\Mod_{-}(\cX)~}
\Cat_{(\infty,1)}
~.
\]
There is a canonical identification over $\Mod_{M^{\op}}(\cX)$ from the $\infty$-category of $(M \rtimes N)^{\op}$-modules in $\cX$ to that of $M^{\op}$-modules in $\cX$ with the structure of being left-laxly invariant with respect to this precomposition $N^{\op}$-action:
\[
\Mod_{
\left(
M \rtimes N
\right)^{\op}
}(\cX)
~\simeq~
\Mod_{M^{\op}}(\cX)^{{\sf l.lax} N^{\op}}
~.
\]
In particular, there is a canonical fully faithful functor from the (strict) $N$-invariants,
\[
\Mod_{M^{\op}}(\cX)^{N}
~\hookrightarrow~
\Mod_{\left( M \rtimes N\right)^{\op}}(\cX)
~.
\]
which is an equivalence if the continuous monoid $N$ is a continuous group.

\end{prop}

\begin{proof}
The second assertion follows immediately from the first.  
The first assertion is proved upon justifying the sequence of equivalences among $\infty$-categories, each which is evidently over $\Mod_M(\cX)$:
\begin{eqnarray*}
\Mod_{\left( M \rtimes N \right)^{\op} }(\cX)
&
\underset{\rm (a)}\simeq
&
\Fun\bigl( 
~
\fB (M \rtimes N )^{\op}
~,~ 
\cX 
~
\bigr)
\\
&
\underset{\rm (b)}\simeq
&
\Fun\bigl( 
~
\fB ( N^{\op} \ltimes M^{\op} )
~,~ 
\cX 
~
\bigr)
\\
&
\underset{\rm (c)}\simeq
&
\Fun_{/\fB N^{\op}}
\Bigl(
~
\fB N^{\op}
~,~
\Fun^{\sf rel}_{\fB N^{\op}}
\bigl( 
\fB ( N^{\op} \ltimes M^{\op} )
,
\cX \times \fB N^{\op}
\bigr)
~
\Bigr)
\\
&
\underset{\rm (d)}\simeq
&
\Fun_{/\fB N^{\op}}\Bigl( 
~
\fB N^{\op}
~,~ 
\Fun^{\sf rel}_{\fB N^{\op}}
\bigl( 
(\fB M^{\op})_{/^{\sf r.lax} N}
,
\cX \times \fB N^{\op}
\bigr)
~
\Bigr)
\\
&
\underset{\rm (e)}\simeq
&
\Fun_{/\fB N^{\op}}\Bigl( 
~
\fB N^{\op}
~,~ 
\Fun(
\fB M^{\op}
,
\cX)_{/^{\sf l.lax} N^{\op}}
~
\Bigr)
\\
&
\underset{\rm (f)}\simeq
&
\Fun_{/\fB N^{\op}}\Bigl( 
~
\fB N^{\op}
~,~
\Mod_{M^{\op}}(\cX)_{/^{\sf l.lax} N^{\op}}
~
\Bigr)
\\
&
\underset{\rm (g)}\simeq
&
\Mod_{M^{\op}}(\cX)^{{\sf l.lax}N^{\op}}
~.
\end{eqnarray*}
The identifications~(a) and~(f) are both the definition of $\infty$-categories of modules for continuous monoids in $\cX$.
The identification~(b) is Observation~\ref{f5}.
By definition of semi-direct product monoids, 
the Cartesian unstraightening the composite functor 
$
\fB N
\xra{\lag N \lacts M^{\op}\rag } 
{\sf Monoids}
\xra{\fB}
\Cat_{(\infty,1)}
$
is the Cartesian fibration:
\[
\fB (N^{\op} \ltimes M^{\op}) 
\longrightarrow 
\fB N^{\op}
~.
\]
Being a Cartesian fibration ensures the existence of the \bit{relative functor $\infty$-category} (see~\cite{fib}).
The identification~(c) is direct from the definition of relative functor $\infty$-categories.  
Furthermore, there is a definitional identification of the \bit{right-lax coinvariants} $\fB (N^{\op} \ltimes M^{\op})  \simeq (\fB M^{\op})_{/^{\sf r.lax}N}$ over $\fB N^{\op}$ (see Appendix~A of~\cite{non-com-geom}), which determines the identification~(d).
The identification~(e) follows from the codification of the $N^{\op}$-action on $\Fun(\fB M^{\op} , \cX)$ in the statement of the proposition.
The identification~(g) is the definition of \bit{left-lax invariants} (see Appendix~A of~\cite{non-com-geom}).

\end{proof}

The commutativity of the topological group $\TT^2$ determines a canonical identification $\TT^2 \cong (\TT^2)^{\op}$ between topological groups, and therefore between continuous groups.
Together with Observation~\ref{t60}, we have the following consequence of Proposition~\ref{t65}.
\begin{cor}
\label{f6}
For $\cX$ an $\infty$-category, there is a canonical identification between $\infty$-categories over $\Mod_{\TT^2}(\cX)$:
\[
\Mod_{\left(\TT^2 \rtimes \Ebraid \right)^{\op}}(\cX)
~\simeq~
\Mod_{\TT^2}(\cX)^{{\sf l.lax} \Ebraid}
~.
\]

\end{cor}

\section{Some facts about the braid group and braid monoid}
\label{sec.B}
Here we collect some facts about the braid group on 3 strands, and the braid monoid on 3 strands.

\subsection{Ambidexterity of $\Ebraid$} \label{ambidex}

\begin{observation}
\label{t60}
Taking transposes of matrices identifies the nested sequence among monoids with the nested sequence of their opposites:
\[
\Bigl(
~
\SL_2(\ZZ)
~
\subset
~
\EpZ
~
\subset
~
\GL_2^+(\RR)
~
\Bigr)
~\overset{{}^{~{}~}(-)^T}\cong~
\Bigl(
~
\SL_2(\ZZ)^{\op}
~
\subset
~
\EpZ^{\op}
~
\subset
~
\GL_2^+(\RR)^{\op}
~
\Bigr)
~.
\]
By covering space theory, these identifications canonically lift as identifications between nested sequences among monoids and their opposites:
\[
\Bigl(
~
\Braid
~
\subset
~
\Ebraid
~
\subset
~
\w{\GL}_2^+(\RR)
~
\Bigr)
~\overset{{}^{~{}~}(-)^T}\cong~
\Bigl(
~
\Braid^{\op}
~
\subset
~
\Ebraid^{\op}
~
\subset
~
\w{\GL}_2^+(\RR)^{\op}
~
\Bigr)
~.
\]
\end{observation}

\begin{cor}
\label{t61}
For each $\infty$-category $\cX$, there are canonical identifications
\[
\Mod_{\Braid}(\cX)
~\simeq~
\Mod_{\Braid^{\op}}(\cX)
\qquad
\text{ and }
\qquad
\Mod_{\Ebraid}(\cX)
~\simeq~
\Mod_{\Ebraid^{\op}}(\cX)
\]
between $\infty$-categories of (left-)modules in $\cX$ and those of right-modules in $\cX$.

\end{cor}

\begin{remark}
The composite isomorphism
$\Braid \underset{\cong}{\xra{(-)^{T}}} \Braid^{\op}  \underset{\cong}{\xra{(-)^{-1}}} \Braid$ 
is the involution of $\Braid$ given in terms of the presentation~(\ref{e67}) by exchanging $\tau_1$ and $\tau_2$. 
Similarly, the involution
$\SL_2(\ZZ) \underset{\cong}{\xra{(-)^{T}}} \SL_2(\ZZ)^{\op}  \underset{\cong}{\xra{(-)^{-1}}} \SL_2(\ZZ)$ exchanges $U_1$ and $U_2$.  

\end{remark}

\subsection{Comments about $\Braid$ and $\Ebraid$}

\begin{observation}
In $\Braid$ (recall the presentation of (\ref{e67}), there is an identity of the generator of $\Ker(\Phi)$:
\[
(\tau_{1} \tau_{2} \tau_{1})^4 
=
( \tau_{1} \tau_{2})^6
=
(\tau_{2} \tau_{1} \tau_{2})^4
\in {\sf Ker}(\Phi)
~.
\]
For that matter, since the matrix
\begin{equation}
\label{e64}
R~:=~ U_1 U_2 U_1 ~=~ 
\begin{bmatrix}
0 & 1 
\\
-1 & 0
\end{bmatrix}
~=~
U_2 U_1 U_2
~\in~\GL_2(\ZZ)
\end{equation}
implements rotation by $-\frac{\pi}{2}$, 
then $R^4 = \uno$ in $\GL_2(\ZZ)$.

\end{observation}

The following result is an immediate consequence of how $\Ebraid$ is defined in equation (\ref{e77}), using that the continuous group $\GL^+_2(\RR)$ is a path-connected 1-type.
\begin{cor}
\label{t31}
There are pullbacks among continuous monoids:
\begin{equation*}
\xymatrix{
\Braid
\ar[rr]
\ar[d]_-{\Phi}
&&
\Ebraid
\ar[d]_-{\Psi} \ar[rr]
&&
\ast \ar[d]^-{\lag \uno \rag}
\\
\GL_2(\ZZ)
\ar[rr]
&&
\EZ
\ar[rr]^-{\RR \underset{\ZZ}\ot}
&&
\GL_2(\RR)
.
}
\end{equation*}
In particular, there is a canonical identification between continuous groups over $\GL_2(\ZZ)$:
\[
\Braid
~\simeq~
\Omega\bigl(
\GL_2(\RR)_{/\GL_2(\ZZ)}
\bigr)
\qquad
\bigl({\rm ~over~}
\GL_2(\ZZ)
~\bigr)
~.
\]

\end{cor}

\begin{observation}
\label{t39}
The inclusion $\SL_2(\ZZ)\subset \EpZ$ between submonoids of $\GL_2^+(\RR)$ determines an inclusion between topological monoids:
\begin{equation}
\label{e4}
\TT^2 \rtimes \Braid
\longrightarrow
\TT^2 \rtimes \Ebraid
~.
\end{equation}
After Observation~\ref{t21}, this inclusion~(\ref{e4}) witnesses the maximal subgroup, both as topological monoids and as monoid-objects in the $\infty$-category $\Spaces$.

\end{observation}

\begin{remark}
\label{r4}
We give an explicit description of $\Ebraid$.
In~\cite{rawn}, the author gives an explicit description for the universal cover of $\SP_{2}(\RR) = \SL_{2}(\RR)$ (and goes on to establish the pullback square of Proposition~\ref{t32}).
Following those methods, consider the maps
\[
\phi
\colon
\GL_2(\RR)
\longrightarrow
\SS^1
~,
\qquad
 A
\mapsto \frac{(a + d) + i(b - c)}{|(a + d) + i(b - c)|}
~,
\]
where
$
A =
\begin{bmatrix}
a & b 
\\
c & d
\end{bmatrix}.
$
As in \cite{rawn}, consider a map
$
\eta \colon \GL_2(\RR) \times \GL_2(\RR) \rightarrow \RR
$
for which
\[
e^{i\eta(A, B)} = \frac{1 - \alpha_{A}\overline{\alpha_{B^{-1}}}}{|1 - \alpha_{A}\overline{\alpha_{B^{-1}}}|}
\qquad
\text{ where }
\qquad
\alpha_{A} = \frac{a^{2} + c^{2} - b^{2} - d^{2} - 2i(ad + bc)}{(a + d)^{2} + (b - c)^{2}}
~.
\]
In these terms, the monoid $\Ebraid$ can be identified as the subset
\[
\Ebraid
~:=~
\bigl\{
(A, s)  \mid  \phi(A) = e^{is}
\bigr\}
~\subset~
\EpZ \times \RR
~,~
\text{with monoid-law }(A, s) \cdot (B, t) := \bigl(AB, s + t + \eta(A, B) \bigr)
~.
\]
\end{remark}

\subsection{Group-completion of $\Ebraid$}
The continuous group $\GL_2^+(\RR)$ is path-connected with $\pi_1\bigl( \GL_2^+(\RR) , \uno \bigr) \cong \ZZ$.
Consequently, there is a central extension
\begin{equation}
\label{e84}
1
\longrightarrow
\ZZ
\longrightarrow
\w{\GL}_2^+(\RR)
\xra{\rm universal~cover}
\GL_2^+(\RR)
\longrightarrow
1
~.
\end{equation}
Consider the inclusion as scalars $\RR_{>0}^\times \underset{\rm scalars} \hookrightarrow \GL_2^+(\RR)$.
Contractibility of the topological group $\RR_{>0}^\times$ implies base-change of this central extension~(\ref{e84}) along this inclusion as scalars splits.
In particular, 
for $\RR\underset{\QQ}\ot \colon \GL_2^+(\QQ)\subset \GL_2^+(\RR)$ the subgroup with rational coefficients, 
there are lifts among continuous monoids in which the squares are pullbacks:
{\Small
\[
\xymatrix{
\NN^\times
\ar[rr]
\ar[drrrrrr]_-{\rm scalars}
\ar@{-->}@(u,-)[rrrrrr]^-{~}
&&
\QQ_{>0}^{\times}
\ar[rr]
\ar[drrrrrr]
\ar@{-->}@(u,u)[rrrrrr]^-{\rm \w{scalars}}
&&
\RR_{>0}^{\times}
\ar[drrrrrr]
\ar@{-->}@(-,u)[rrrrrr]^-{~}
&&
\Ebraid
\ar[d]
\ar[rr]_-{\w{\QQ\underset{\ZZ}\ot}}
&&
\w{\GL}_2^+(\QQ)
\ar[d]
\ar[rr]_-{\w{\RR\underset{\QQ}\ot}}
&&
\w{\GL}_2^+(\RR)
\ar[d]^-{\rm cover}_-{\rm universal}
\\
&&&
&&&
\EpZ
\ar[rr]_-{\QQ\underset{\ZZ}\ot}
&&
\GL_2^+(\QQ)
\ar[rr]_-{\RR\underset{\QQ}\ot}
&&
\GL_2^+(\RR)
.
}
\]
}

\begin{prop}
\label{t59}
Each of the diagrams among continuous monoids
\[
\xymatrix{
\NN^\times
\ar[rr]^-{\rm scalars}
\ar[d]_-{\rm inclusion}
&&
\EZ
\ar[d]^-{\QQ\underset{\ZZ}\ot}
&&
\NN^\times
\ar[rr]^-{\rm \w{scalars}}
\ar[d]_-{\rm inclusion}
&&
\Ebraid
\ar[d]^-{\w{\QQ\underset{\ZZ}\ot}}
\\
\QQ_{>0}^{\times}
\ar[rr]^-{\rm scalars}
&&
\GL_2(\QQ)
&
\text{ and }
&
\QQ_{>0}^{\times}
\ar[rr]^-{\rm \w{scalars}}
&&
\w{\GL}_2^+(\QQ)
}
\]
witnesses a pushout.
In particular, because $\NN^\times \xra{\rm inclusion}\QQ^\times_{>0}$ witnesses group-completion among continuous monoids, then each of the right downward morphisms witnesses group-completion among continuous monoids.

\end{prop}

\begin{proof}
We explain the following commutative diagram among spaces:
\[
\xymatrix{
\EZ
\ar[rrrr]^-{\RR\underset{\ZZ}\ot }
\ar@{-->}[dr]_-{\rm (a)}
\ar[drrr]
&
&&
&
\GL_2(\QQ)
\ar@{-->}[dlll]_-{\rm (b)}
\\
&
\underset{\NN^{\sf div}}\colim
~\EZ
\ar@{-->}[rr]_-{\rm (c)}
&&
\EZ[ (\NN^\times)^{-1} ]
\ar[ur]_-{\ov{\RR\underset{\ZZ}\ot}}
&
}
\]
The top horizontal arrow is the standard inclusion.
Here, scalar matrices embed the multiplicative monoid of natural numbers $\NN^\times \underset{\rm scalars}\subset \EZ$.
The bottom right term, equipped with the diagonal arrow to it, is the indicated localization (among continuous monoids). 
The up-rightward arrow is the unique morphism between continuous monoids under $\EZ$, which exists because the continuous monoid $\GL_2(\QQ)$ is a continuous group.
The solid diagram of spaces is thusly forgotten from a diagram among continuous monoids.

Next, the poset $\NN^{\sf div}$ is the natural numbers with partial order given by divisibility: $r\leq s$ means $r$ divides $s$.   
Consider the functor 
\[
F_{\EZ} \colon \NN^{\sf div} \longrightarrow \Sets
\hookrightarrow
\Spaces
~,\qquad
r\mapsto \EZ
~{}~
\text{ and }
~{}~
(r\leq s)
\mapsto 
\bigl(
~
\EZ \xra{ \frac{s}{r} \cdot -} \EZ
~
\bigr)
~.
\]
The colimit term in the above diagram is 
$
\colim\bigl( 
F_{\EZ}
\bigr)
$,
which can be identified as the classifying space of the poset
\[
{\sf Un}\bigl( F_{\EZ} \bigr)
~=~
\Bigl(
~
\NN\times \EZ
\text{ , with partial order }(r,A) \leq (s,B) \text{ meaning }
r\leq s \text{ in }\NN^{\sf div}
\text{ and }
\frac{s}{r} \cdot A=B
~
\Bigr)
~.
\]
\begin{itemize}
\item
The dashed arrow~$\rm (a)$ is the canonical map from the $1$-cofactor of the colimit.  

\item
The dashed arrow~$\rm (b)$ is implemented by the map $\w{\rm (b)}\colon \GL_2(\QQ)\xra{A\mapsto (r_A,r_A\cdot A)} \NN\times \EZ$ where $r_A\in \NN$ is the smallest natural number for which the matrix $r_A\cdot  A\in \EZ$ has integer coefficients.
The triangle with sides $\rm (a)$ and $\rm (b)$ evidently commutes.   

\item
The dashed arrow~$\rm (c)$ is implemented by the map $\w{\rm (c)}\colon
{\sf Un}\bigl( F_{\EZ} \bigr)
 \xra{(r,A)\mapsto r^{-1}A} \EZ[ (\NN^\times)^{-1} ]$.
The triangle with sides~$\rm (a)$ and $\rm (c)$ evidently commutes.
We now argue that the map~$\rm (c)$ is an equivalence between spaces.

Observe the identification between continuous monoids 
\[
\underset{p~{\rm prime}} \bigoplus (\ZZ_{\geq 0},+)
\xra{~\cong~}
\NN^\times
~,\qquad
\bigl(
\{ p~{\rm prime}\}
\xra{\eta}
\ZZ_{\geq 0}
\bigr)
\mapsto 
\underset{p~\rm prime}\prod p^{\eta(p)}
~,
\]
as a direct sum, indexed by the set of prime numbers, of free monoids each on a single generator.
For $S$ a set of prime numbers, denote by $\lag S \rag^{\times} \subset \NN^\times$ the submonoid generated by $S$.  
For $S$ a set of primes, and for $p\in S$, 
the above identification as a direct sum of monoids restricts as an identification $(\ZZ_{\geq 0},+ ) \times \lag S \smallsetminus \{p\} \rag^{\times} \cong  \lag \{p\} \rag^{\times}  \times \lag S \smallsetminus \{p\} \rag^{\times}   \cong \lag S \rag^{\times}$.

Next, observe an identification of the poset $\NN^{\sf div} \simeq (\fB \NN^\times)^{\ast/}$ as the undercategory of the deloop.  
Through this identification, and the above identification supplies an identification between posets from the direct sum (based at initial objects) indexed by the set of prime numbers:
\[
\underset{p~\rm prime} \bigoplus (\ZZ_{\geq 0},\leq)
\xra{~\cong~}
\NN^{\sf div}
~,\qquad
\bigl(
\{ p~{\rm prime}\}
\xra{\chi}
\ZZ_{\geq 0}
\bigr)
\mapsto 
\underset{p~\rm prime}\prod p^{\chi(p)}
~.
\]
For $S$ a set of prime numbers, denote by $\lag S \rag^{\sf div} \subset \NN^{\sf div}$ the full subposet generated by $S$.  
For $S$ a set of primes, and for $p\in S$, 
the above identification as a direct sum of posets restricts as an identification $(\ZZ_{\geq 0},\leq ) \times \lag S \smallsetminus \{p\} \rag^{\sf div} \cong  \lag \{p\} \rag^{\sf div}  \times \lag S \smallsetminus \{p\} \rag^{\sf div}   \cong \lag S \rag^{\sf div}$.
In particular, 
the standard linear order on the set of prime natural numbers determines the sequence of functors
\begin{equation}
\label{e85}
\NN^{\sf div}
\xra{{\sf loc}_2}
\lag p>2 \rag^{\sf div}
\xra{{\sf loc}_3}
\lag p > 3 \rag^{\sf div}
\xra{{\sf loc}_5}
\lag p > 5 \rag^{\sf div}
\xra{{\sf loc}_7}
\dots
~,
\end{equation}
each which is isomorphic with projection off of $(\ZZ_{\geq 0},\leq)$.  
In particular, each projection is a coCartesian fibration, so left Kan extension along each functor is computed as a sequential colimit.  
Because $\NN^\times \underset{\rm scalars}\subset \EZ$ is (strictly) central, 
so too is $(\ZZ_{\geq 0},+)\cong \lag \{p\}\rag^{\times}\subset \EZ$.  
The following claim follows from these observations, using induction on the standardly ordered set of primes. 
\begin{itemize}
\item[{\bf Claim.}]
For each prime $q$, 
left Kan extension of $F_{\EZ}$ along the composite functor $\NN^{\sf div}\xra{{\sf loc}_q} \lag p>q \rag^{\sf div}$ is the functor
\[
F_{\EZ\bigl[ (\lag p' \leq q \rag^{\times})^{-1}\bigr]}
\colon \lag p > q \rag^{\sf div} 
\xra{~({\sf loc}_q)_!(\EZ)~} 
\Spaces
~,
\]
\[
r\mapsto \EZ\bigl[ (\lag p' \leq q \rag^{\times})^{-1}\bigr]
\qquad
\text{ and }
\qquad
(r\leq s)
\mapsto 
\bigl(
\EZ\bigl[ (\lag p' \leq q \rag^{\times})^{-1}\bigr] \xra{ \frac{s}{r} \cdot -} \EZ\bigl[ (\lag p' \leq q \rag^{\times})^{-1}\bigr]
\bigr)
~,
\]
that evaluates on each $r$ as the localization $\EZ\bigl[ (\lag p' \leq q \rag^{\times})^{-1} \bigr]$, and on each relation $r \leq s$ in $\NN^{\sf div}$ as scaling by $\frac{s}{r}$.
\end{itemize}
Next, the colimit of this sequence~(\ref{e85}) is $\underset{q ~ \rm prime}\bigcap \lag p > q \rag^{\sf div} \simeq \ast$ terminal.  
Consequently, there is a canonical identification 
\begin{eqnarray}
\nonumber
\colim( F_{\EZ} )
&
~\simeq ~
&
\underset{q \in \{2<3<5\cdots\}}\colim 
\Bigl(
({\sf loc}_q)_! \bigl( F_{\EZ} \bigr)
\Bigr)
~\simeq~
\underset{q \in \{2<3<5\cdots\}}\colim 
\Bigl(
F_{\EZ\bigl[ (\lag p' \leq q \rag^{\times})^{-1}\bigr]}
\Bigr)
\\
\nonumber
&
~\simeq~
&
\EZ\Bigl[ 
\bigl(
\underset{q \in \{2<3<5<\cdots\}} \bigcup \lag p' \leq q \rag^{\times} 
\bigr)^{-1}
\Bigr]
~=~
\EZ[ (\NN^\times)^{-1} ]
~.
\end{eqnarray}

\item
By inspection, the resulting self-map of $\GL_2(\QQ)$ is the identity.  
Indeed, the natural transformation
\[
\xymatrix{
{\sf Un}\bigl( F_{\EZ} \bigr)
\ar@(u,u)[rr]^-{\id}
\ar[d]_-{\w{\rm (c)}}
&
\Uparrow
&
{\sf Un}\bigl( F_{\EZ} \bigr)
\\
\EZ[(\NN^\times)^{-1}]
\ar[rr]^{\ov{\RR\underset{\ZZ}\ot}}
&&
\GL_2(\QQ)
\ar[u]_-{\w{\rm (b)}}
,
}
\]
given by, for each $(s,B)\in {\sf Un}\bigl( F_{\EZ} \bigr)$, the relation $\bigl( r_{s^{-1}\cdot B} , r_{s^{-1}\cdot B} \cdot (s^{-1}\cdot B) \bigr)  \leq (s,B)$, witnesses an identification of the resulting self-map of $\underset{\NN^{\sf div}}\colim \EZ$ with the identity.  
\end{itemize}
We conclude that the map $\EZ[(\NN^\times)^{-1}]\xra{\ov{\RR\underset{\ZZ}\ot}} \GL_2(\QQ)$ is an equivalence.
It follows that the left square in the statement of the proposition is a pushout because the morphism $\NN^\times \xra{\rm inclusion}\QQ_{>0}^{\times}$ witnesses a group-completion (among continuous monoids).

The same argument also implies the square
\[
\xymatrix{
\NN^\times
\ar[rr]^-{\rm scalars}
\ar[d]_-{\rm inclusion}
&&
\EpZ 
\ar[d]^-{\QQ\underset{\ZZ}\ot}
\\
\QQ_{>0}^\times
\ar[rr]^-{\rm scalars}
&&
\GL_2^+(\QQ)
}
\]
also witnesses a pushout among continuous monoids.
Base-change along the central extension~(\ref{e84}) among continuous groups
reveals that the right square is also a pushout among continuous groups.

\end{proof}

\subsection{Relationship with the finite orbit category of $\TT^2$}
Recall the $\infty$-category ${\sf Orbit}_{\TT^2}^{\sf fin}$ of transitive $\TT^2$-spaces with finite isotropy, and $\TT^2$-equivariant maps between them. 
Recall that the action $\Ebraid \to \EZ \lacts \TT^2$ on the topological group determines an action 
\begin{equation}
\label{e92}
\Ebraid \underset{\rm Obs~\ref{t60}}\simeq \Ebraid^{\op} \lacts {\sf Orbit}^{\sf fin}_{\TT^2}
~.
\end{equation}

\begin{prop}
\label{t68}
There is a canonical identification of the $\infty$-category of coinvariants with respect to the action~(\ref{e92}): 
\[
\Bigl(
{\sf Orbit}^{\sf fin}_{\TT^2}
\Bigr)_{/\Ebraid}
\xra{~\simeq~}
\fB \bigl(
\TT^2 \rtimes \Ebraid
\bigr)
~.
\]

\end{prop}

\begin{proof}
Recall that $\Ebraid \subset \w{\GL}^+_2(\RR)$ is defined as a submonoid of a group.
As a result, the left-multiplication action by its maximal subgroup,
$\w{\GL}_2^+(\ZZ) \lacts \Ebraid$, is free.  
Consequently, the space of objects $\Obj\bigl( (\fB \Ebraid)^{\ast/} \bigr) \simeq \Ebraid_{/\w{\GL}^+_2(\ZZ)}\xra{\cong} \EpZ_{/\GL_2^+(\ZZ)}$ is simply the quotient set of $\Ebraid$ by its maximal subgroup acting via left-multiplication, which is bijective with the quotient of $\EpZ$ by its maximal subgroup via the canonical projection $\Ebraid \to \EpZ$.  
The space of morphisms between objects represented by $A, B\in \EpZ$,
\[
\Hom_{(\fB \Ebraid)^{\ast/}}\bigl( [A] , [B] \bigr)
~\simeq~
\Bigl\{X \in \EpZ \mid XA=B\Bigr\}
~\subset~ 
\EpZ
~,
\]
is simply the set of factorizations in $\EpZ$ of $B$ by $A$.
In particular, the $\infty$-category $(\fB \Ebraid)^{\ast/}$ is a poset.  
We now identify this poset essentially through Pontryagin duality.

Consider the poset $\sP^{\sf fin}_{\TT^2}$ of finite subgroups of $\TT^2$ ordered by inclusion.  
We now construct mutually-inverse functors between posets:
\begin{equation}
\label{e90}
(\fB \Ebraid)^{\ast/}
\xra{~[A]
\mapsto 
\Ker\bigl( \TT^2 \xra{A} \TT^2 \bigr)
~}
\sP^{\sf fin}_{\TT^2}
\qquad
\text{ and }
\qquad
\sP^{\sf fin}_{\TT^2}
\xra{~C
\mapsto 
\bigl[ \ZZ^2 \xra{A_C} \ZZ^2 \bigr]
~}
(\fB \Ebraid)^{\ast/}
~.
\end{equation}
The first functor assigns to $[A]$ the kernel of the endomorphism of $\TT^2$ induced by a representative $A \in \EpZ \lacts \TT^2$.   
The second functor assigns to $C$ the endomorphism $(\ZZ^2 \xra{A_C}\ZZ^2)\in \EpZ$ defined as follows.
The preimage $\ZZ^2 \subset \quot^{-1}(C) \subset \RR^2 \xra{\quot} \RR^2_{/\ZZ^2} =:\TT^2$ by the quotient is a lattice in $\RR^2$ that contains the standard lattice cofinitely.
There is a unique pair of non-negative-quadrant vectors $(u_1,u_2)\in (\RR_{\geq 0})^2\times (\RR_{\geq 0})^2$ that generate this lattice $\quot^{-1}(C)$ and agree with the standard orientation of $\RR^2$.
Then $A_C \in \EpZ$ is the unique matrix for which $A_C \vec{u}_i = \vec{e}_i$ for $i=1,2$.  
It is straight-forward to verify that the two assignments in~(\ref{e90}) indeed respect partial orders, and are mutually-inverse to one another.
Observe that the action~(\ref{e92}) descends as an action $\Ebraid^{\op}\lacts \sP^{\sf fin}_{\TT^2}$, with respect to which the equivalences~(\ref{e90}) are $\Ebraid^{\op}$-equivariant.

Next, reporting the stabilizer of a transitive $\TT^2$-space defines a functor
${\sf Orbit}^{\sf fin}_{\TT^2} \xra{ (\TT^2 \lacts T)\mapsto {\sf Stab}_{\TT^2}(t)} \sP^{\sf fin}_{\TT^2}$.
Evidently, this functor is conservative.
Notice also that this functor is a left fibration; 
its straightening is the composite functor
\begin{equation}
\label{e93}
\sP^{\sf fin}_{\TT^2}
\xra{~C\mapsto \frac{\TT}{C}~}
{\sf Groups}
\xra{~\sB~}
\Spaces
~.
\end{equation}

The result follows upon constructing a canonical filler in the diagram among $\infty$-categories witnessing a pullback:
\[
\xymatrix{
{\sf Orbit}_{\TT^2}^{\sf fin}
\ar@{-->}[rrrr]
\ar[d]
&&
&&
\Ar\bigl( \fB (\TT^2 \rtimes \Ebraid ) \bigr)
\ar[d]^-{\Ar(\fB {\sf proj}) }
\\
\sP^{\sf fin}_{\TT^2}
\ar[rr]^-{\simeq}_-{(\ref{e90})}
&&
(\fB \Ebraid)^{\ast/}
\ar[rr]^-{\rm forget}
&&
\Ar\bigl( \fB \Ebraid \bigr)
.
}
\]
By definition of semi-direct products,
the canonical functor $\fB (\TT^2 \rtimes \Ebraid) \xra{\fB \sf proj} \fB \Ebraid$ is a coCartesian fibration.
Because the $\infty$-category $\fB \TT^2 =\sB \TT^2$ is an $\infty$-groupoid, this coCartesian fibration is conservative, and therefore a left fibration.
Consequently, the functor 
\[
\Ar\bigl( \fB (\TT^2 \rtimes \Ebraid) \bigr)
\to
\Ar( \fB \Ebraid )
\]
is also a left fibration.
Therefore, the base-change of this left fibration along $
(\fB \Ebraid)^{\ast/}
\xra{\rm forget}
\Ar\bigl( \fB \Ebraid \bigr)
$ is again a left fibration:
\begin{equation}
\label{e91}
\Ar\bigl( \fB (\TT^2 \rtimes \Ebraid) \bigr)^{|\sB \TT^2}
\longrightarrow
(\fB \Ebraid)^{\ast/}
~\underset{(\ref{e90})}\simeq~
\sP^{\sf fin}_{\TT^2}
~.
\end{equation}
Direct inspection identifies the straightening of this left fibration~(\ref{e91}) as~(\ref{e93}).

\end{proof}

\end{document}